\documentclass[12pt,draftcls, onecolumn]{IEEEtran}
\usepackage{amsthm,amsmath,amssymb}
\usepackage{graphicx,xcolor}
\usepackage{subcaption}
\usepackage{nicefrac}
\usepackage[normalem]{ulem}
\usepackage{setspace}
\usepackage[font=footnotesize,labelfont=bf]{caption}
\usepackage{accents}
\usepackage{algpseudocode}
\usepackage{tikz}
\usetikzlibrary{calc}
\usetikzlibrary{arrows}

\newcommand\str{\bgroup\markoverwith{\textcolor{red}{\rule[0.5ex]{2pt}{1.5pt}}}\ULon} 

\newtheorem{theorem}{Theorem}[section]

\newtheorem{assumption}{Assumption}[section]
\newtheorem{lemma}[theorem]{Lemma}

\DeclareMathOperator*{\argmin}{\arg\!\min}

\def\grad{\nabla}

\def\be{\mathbf{e}}

\def\bp{\mathbf{p}}
\def\bq{\mathbf{q}}

\def\bs{\mathbf{s}}

\def\bu{\mathbf{u}}
\def\bv{\mathbf{v}}
\def\bw{\mathbf{w}}
\def\bx{\mathbf{x}}  
\def\by{\mathbf{y}}
\def\bz{\mathbf{z}}

\def\bD{\mathbf{D}}

\def\bL{\mathbf{L}}

\def\bQ{\mathbf{Q}}

\def\x{{\boldsymbol{\xi}}}

\def\cA{\mathcal{A}}
\def\cB{\mathcal{B}}
\def\cC{\mathcal{C}}

\def\cE{\mathcal{E}}
\def\cF{\mathcal{F}}
\def\cG{\mathcal{G}}

\def\cK{\mathcal{K}}
\def\cL{\mathcal{L}}

\def\cN{\mathcal{N}}
\def\cO{\mathcal{O}}
\def\cP{\mathcal{P}}

\def\cR{\mathcal{R}}
\def\cS{\mathcal{S}}

\def\cX{\mathcal{X}}
\def\cY{\mathcal{Y}}

\def\smskip{\smallskip}

\def\texitem#1{\par\smskip\noindent\hangindent 25pt
               \hbox to 25pt {\hss #1 ~}\ignorespaces}

\def\abs#1{\left|#1\right|}
\def\norm#1{\left\|#1\right\|}

\newcommand{\BEAS}{\begin{eqnarray*}}
\newcommand{\EEAS}{\end{eqnarray*}}
\newcommand{\BEA}{\begin{eqnarray}}
\newcommand{\EEA}{\end{eqnarray}}
\newcommand{\BEQ}{\begin{eqnarray}}
\newcommand{\EEQ}{\end{eqnarray}}
\newcommand{\BIT}{\begin{itemize}}
\newcommand{\EIT}{\end{itemize}}
\newcommand{\BNUM}{\begin{enumerate}}
\newcommand{\ENUM}{\end{enumerate}}

\newcommand{\BA}{\begin{array}}
\newcommand{\EA}{\end{array}}

\newcommand{\ones}{\mathbf 1}

\newcommand{\reals}{\mathbb{R}}
\newcommand{\integers}{\mathbb{Z}}

\newcommand{\Rank}{\mathop{\bf rank}}

\newcommand{\diag}{\mathop{\bf diag}}

\newcommand{\dom}{\mathop{\bf dom}}

\newif\ifpagenumbering
\pagenumberingtrue

\pagenumberingfalse

\newsavebox{\theorembox}
\newsavebox{\lemmabox}
\newsavebox{\defnbox}
\newsavebox{\corollarybox}
\newsavebox{\remarkbox}
\newsavebox{\assbox}
\savebox{\theorembox}{\noindent\bf Theorem}
\savebox{\lemmabox}{\noindent\bf Lemma}
\savebox{\defnbox}{\noindent\bf Definition}
\savebox{\corollarybox}{\noindent\bf Corollary}
\savebox{\remarkbox}{\noindent\bf Remark}
\newtheorem{remark}{\usebox{\remarkbox}}[section]
\newtheorem{defn}{\usebox{\defnbox}}

\def\fprod#1{\left\langle#1\right\rangle}
\def\prox#1{\mathbf{prox}_{#1}}
\def\ind#1{\mathbb{I}_{#1}}
\def\T{\mathsf{T}}
\def\id{\mathbf{I}}
\def\zero{\mathbf{0}}
\def\one{\mathbf{1}}

\def\btheta{\boldsymbol{\theta}}
\def\blambda{\boldsymbol{\lambda}}
\def\bmu{\boldsymbol{\nu}}
\def\bxi{\boldsymbol{\xi}}
\def\bom{\boldsymbol{\omega}}
\def\st{\mathbf{s.t.}}
\def\Ct{\widetilde{\cC}}
\def\tcR{\widetilde{\cR}}

\def\ttau{\tilde{\tau}}
\newcommand{\ubar}[1]{\underaccent{\bar}{#1}}

\def\sy{\mathbb{S}}
\def\ns{\mathbf{Null}}
\def\spn{\mathbf{Span}}

\title{\vspace*{-3mm}{Multi-agent constrained optimization of\\ a strongly convex function over\\ time-varying directed networks}\vspace*{-2mm}}
\begin{document}
\author{\IEEEauthorblockN{
Erfan Yazdandoost Hamedani\IEEEauthorrefmark{1}, and
Necdet Serhat Aybat\IEEEauthorrefmark{1}}\\
\IEEEauthorblockA{\IEEEauthorrefmark{1}Industrial \& Manufacturing Engineering Department,\\ The Pennsylvania State University, PA, USA. \\Emails: {\tt evy5047@psu.edu}, {\tt nsa10@psu.edu}\vspace*{-12mm}}
}
\maketitle
\begin{abstract}
\vspace*{-3mm}
We consider cooperative multi-agent consensus optimization problems over both static and {time-varying communication networks}, where only local communications are allowed. The objective is to minimize {the sum of agent-specific possibly non-smooth composite convex functions} over agent-specific private conic constraint sets; hence, the optimal consensus decision should lie in the intersection of these private sets. {Assuming the sum function is strongly convex}, we provide convergence rates in sub-optimality, infeasibility and consensus violation; examine the effect of underlying network topology on the convergence rates of the proposed decentralized algorithms.\vspace*{-3mm}
\end{abstract}
\section{Introduction}
Decentralized optimization over communication networks has various applications: i) distributed parameter estimation in wireless sensor networks~\cite{predd2006distributed,schizas2008consensus}; ii) multi-agent cooperative control and coordination in multirobot networks~\cite{zhou2011multirobot,cao2013overview}; iii) distributed spectrum sensing in cognitive radio networks~\cite{bazerque2010distributed,bazerque2011group}; iv) processing distributed big-data in (online) machine learning~\cite{tsianos2012consensus,duchi2012dual,cortes1995support,dekel2011optimal,towfic2011collaborative}; v) power control problem in cellular networks~\cite{ram2009distributed}, to name a few application areas.
In many of these network applications the communication network may be directed, i.e., communication links can be unidirectional, and/or the network in the wireless setting may be time-varying, e.g., communication links can be on/off over time due to failures or the links may exist among agents depending on their inter-distances. {In the context of decentralized optimization, time-varying directed networks can also arise in wired networks as uni-directional asynchronous protocols are desired over bi-directional communication protocols which create deadlocks due to lack of enforcement rule to block a third node when the other two neighbors are exchanging local variables between themselves~\cite{tsianos2012consensus}. In majority of the applications discussed above, other than the topology being time-invariant~(static) or time-varying, or the network having undirected or directed links, one common characteristic shared by today's big-data networks is that the network size is usually prohibitively large} for centralized optimization, which requires a fusion center that collects the physically distributed data and runs a centralized optimization method. This process has expensive communication overhead, requires large enough memory to store and process the data, and also may violate data privacy in case agent are not willing to share their data even though they are collaborative agents~\cite{chen2012diffusion,olfati2007consensus}.

In this paper, from a broader perspective, we aim to study constrained distributed optimization of a strongly convex function over static or time-varying communication networks $\cG^t=(\cN,\cE^t)$ for $t\geq 0$; in particular, from an application perspective, we are motivated to design an efficient \emph{decentralized} solution method for \emph{constrained LASSO}~(C-LASSO) problems~\cite{gaines2016algorithms} with distributed data. C-LASSO, having the generic form $\min_x\{\lambda\norm{x}_1+\norm{Cx-d}_2^2:\ Ax\leq b\}$, is an important class of problems {in statistics}, which includes fused LASSO, constrained regression, and generalized LASSO problems 
as its special cases~\cite{hofner2016unified,gaines2016algorithms,james2013penalized} to name a few. In the rest, we provide our results for a more general {setting of constrained decentralized optimization}. We assume that \textbf{i)} each node $i\in\cN$ has a \emph{local} conic convex constraint set $\chi_i$, for which projections are not easy to compute, and a \emph{local} convex objective function $\varphi_i$ (possibly non-smooth) such that $\sum_{i\in\cN}\varphi_i(x)$ is strongly convex, and \textbf{ii)} nodes are willing to collaborate, without sharing their \emph{private} data defining $\chi_i$ and $\varphi_i$, to compute an optimal consensus decision minimizing the sum of local functions and satisfying all local constraints; moreover, \textbf{iii)} nodes are only allowed to communicate with the neighboring nodes over the links in the network.
Although we assume that $\sum_{i\in\cN}\varphi_i(x)$ is strongly convex, it is possible that none of the local functions $\{\varphi_i\}_{i\in\cN}$ are strongly convex. This kind of structure arises in LASSO problems; in particular, let $\varphi_i:\reals^n\rightarrow\reals$ such that $\varphi_i(x)=\lambda\norm{x}_1+\norm{C_ix-d_i}_2^2$ for $C_i\in\reals^{m_i\times n}$ and $d_i\in\reals^{m_i}$ for $i\in\cN$. Note that while $\varphi_i$ is merely convex for all $i\in\cN$, $\sum_{i\in\cN}\varphi_i(x)$ is strongly convex when $m_i<n$ for $i\in\cN$ and $\Rank(C)=n\leq \sum_{i\in\cN}m_i{\triangleq m}$ where $C=[C_i]_{i\in\cN}{\in\reals^{m\times n}}$. Therefore, it is important to note that in the centralized formulation of this problem $\min_x\sum_{i\in\cN}\varphi_i(x)$, the objective is strongly convex; however, in the decentralized formulation, this is not the case where we minimize $\sum_{i\in\cN}\varphi_i(x_i)$ while imposing consensus among local variables $\{x_i\}_{i\in\cN}$. In the numerical section, we considered a distributed C-LASSO problem under a similar strong convexity setting.

Many of the real-life application problems discussed above are special cases of the generic conic constrained decentralized optimization framework discussed in this paper. With the motivation of designing an efficient decentralized solution method for the distributed conic constrained problem over {static or time-varying communication networks}, as we briefly described above, we propose \emph{distributed} primal-dual algorithms: DPDA for static and DPDA-TV for time-varying communication networks. DPDA and DPDA-TV are both based on the primal-dual algorithm~(PDA), recently proposed in~\cite{chambolle2015ergodic} for convex-concave saddle-point problems which for sake of completeness will be discussed in {detail in Section~\ref{sec:pda}.}

\textbf{Problem Description.} {Let $\{\cG^t\}_{t\in\reals_+}$ denote a time-varying graph of $N$ computing nodes. More precisely, for all $t\geq 0$, the graph has the form $\cG^t=(\cN,\cE^t)$, where $\cN\triangleq\{1,\ldots,N\}$ is the set of nodes and $\mathcal{E}^t\subseteq \mathcal{N}\times \mathcal{N}$ 
is the set of (possibly directed) edges at time $t$.} Suppose that each node $i\in\cN$ has a \emph{private} (local) cost function ${\varphi_i}:\reals^n\rightarrow\reals\cup\{+\infty\}$ such that
{
\begin{equation}
\label{eq:F_i}
{\varphi_i}(x)\triangleq \rho_i(x) + f_i(x),
\end{equation}
}%
where $\rho_i: \mathbb{R}^n \rightarrow \mathbb{R}\cup\{+\infty\}$ is a possibly \emph{non-smooth} convex function, and $f_i: \mathbb{R}^n \rightarrow \mathbb{R}$ is a \emph{smooth} convex function. We assume $f_i$ is differentiable on an open set containing $\dom \rho_i$ with a Lipschitz continuous gradient $\grad f_i$, of which Lipschitz constant is $L_i$; and the prox map of $\rho_i$,
{
\begin{equation}
\label{eq:prox}
\prox{\rho_i}(x)\triangleq\argmin_{y \in \reals^n} \left\{ \rho_i(y)+\tfrac{1}{2}\norm{y-x}^2 \right\},
\end{equation}
}%
is \emph{efficiently} computable for $i\in\cN$, where $\norm{.}$ denotes the Euclidean norm.
Consider the following minimization problem:
{
\begin{align}\label{eq:central_problem}
x^*\in \argmin_{x\in\reals^n}\  \bar{\varphi}(x) \triangleq \sum_{i\in \mathcal{N}}\varphi_i(x)\quad \hbox{s.t.}\quad A_ix -b_i \in \mathcal{K}_i,\quad \forall{i}\in\mathcal{N},
\end{align}
}%
\noindent where $A_i\in \mathbb{R}^{m_i\times n}$, $b_i\in \mathbb{R}^{m_i}$ and $\mathcal{K}_i\subseteq{R}^{m_i}$ is a closed, convex cone. Suppose that projections onto $\cK_i$ can be computed efficiently, while the projection onto the \emph{preimage} ${\chi_i\triangleq} A_i^{-1}(\cK_i+b_i)$ is assumed to be \emph{impractical}, e.g., when $\cK_i$ is the positive semidefinite cone, projection to preimage requires solving an SDP.
\begin{assumption}
\label{assump:saddle-point}
The duality gap for \eqref{eq:central_problem} is zero, and a primal-dual solution to~\eqref{eq:central_problem} exists.
\end{assumption}
\vspace*{-1mm}
A sufficient condition 
is the existence of a Slater point, i.e., there exists $\bar{x}\in \mathbf{relint}(\dom \bar{\varphi})$ such that $A_i\bar{x}-b_i\in \mathbf{int}(\cK_i)$ for $i\in\cN$, where $\dom\bar{\varphi}=\cap_{i\in\cN}\dom\varphi_i$.
\begin{defn}
\label{def:sconvex}
A differentiable function ${f}:\reals^n \rightarrow \reals$ is {\it strongly convex} 
with modulus $\mu>0$
if the following inequality holds
{
\begin{align*}
{f(x)\geq f(\bar{x})+\fprod{\nabla f(\bar{x}),x-\bar{x}}}+\frac{\mu}{2}\norm{x-\bar{x}}^2\qquad \forall x,\bar{x}\in\reals^n.
\end{align*}}
\end{defn}
\begin{assumption}
\label{assump:sconvex}
{Suppose ${\bar{f}(x)}\triangleq\sum_{i\in\cN} f_i(x)$ is strongly convex with modulus $\bar{\mu}>0$; and each $f_i$ is strongly convex with modulus $\mu_i\geq 0$ 
for $i\in\cN$, and define {$\ubar{\mu}\triangleq\min_{i\in\cN}\{\mu_i\}\geq 0$}. 
}
\end{assumption}
\begin{remark}
{Clearly $\bar{\mu}\geq \sum_{i\in\cN}\mu_i$ is always true, and it is possible that $\mu_i=0$ for all $i\in\cN$ but still $\bar{\mu}>0$; moreover, $\bar{\mu}>0$ implies that $x^*$ is the unique optimal solution to \eqref{eq:central_problem}.}
\end{remark}

\textbf{Previous Work.} Consider $\min_{x\in\reals^n}\{\bar{\varphi}(x):\ x\in\cap_{i\in\cN}\chi_i\}$ over a communication network of computing agents $\cN$, where $\bar{\varphi}(x)=\sum_{i\in\cN}\varphi_i(x)$.  Although the \emph{unconstrained} consensus optimization, i.e., $\chi_i=\reals^n$, is well studied for static or time-varying networks -- see~\cite{Nedic08_1J,nedic2014distributed} and the references therein,
the \emph{constrained} case 
is still an area of active research, e.g.,
\cite{Nedic08_1J,nedic2014distributed,nedic2010constrained,srivastava2010distributed,zhu2012distributed,yuan2011distributed,chang2014distributed,mateos2015distributed,chang2014proximal,aybat2016primal}.
Our focus in this paper is on the case where $\bar{\varphi}$ is strongly convex such that each $\varphi_i=\rho_i+f_i$ is composite convex, and $\chi_i$ has the form $A_i^{-1}(\cK_i+b_i)$ for $i\in\cN$. In this section, we briefly review the existing work related to our setup.

{Unconstrained} minimization of a \emph{strongly} convex objective function {$\bar{f}(x)\triangleq\sum_{i\in\cN}f_i(x)$} in the multi-agent setting has been investigated in many papers, e.g., \cite{makhdoumi17,shi2015extra,zeng2015extrapush,Dextra16,xi2016add} considered static communication networks $\cG=(\cN,\cE)$ while~\cite{nedic2016stochastic,nedich2016achieving} studied the time-varying networks. In the rest, suppose that $\mu_i\geq 0$ denotes the convexity modulus of $f_i$ for $i\in\cN$. {In~\cite{makhdoumi17}, Makhdoumi and Ozdaglar proposed a distributed ADMM to solve $\min_x\bar{f}(x)$ over a time-invariant~(static), undirected network; they show that when $f_i$ has Lipschitz continuous gradient with constant $L_i$ and when $\mu_i>0$ for each $i\in\cN$, the local iterates at all nodes are within an $\epsilon$-ball of the optimal solution after at most $\cO(\sqrt{\kappa}~\log(1/\epsilon))$ iterations, where $\kappa=L_{\max}/\ubar{\mu}$, $L_{\max}\triangleq \max_{i\in\cN}L_i$ and $\ubar{\mu}\triangleq\min_{i\in\cN}\{\mu_i\}$; on the other hand, since each iteration requires exact minimization of an augmented function involving $f_i$ at each $i\in\cN$, iterations can be very costly depending on $f_i$.} In~\cite{chang2015multi}, Chang et al. considered the composite convex minimization problem, $\min_x \sum_{i\in\cN} \rho_i(x)+f_i(C_i x)$, over a static undirected network $\cG$, where $\rho_i$ is merely convex and $f_i$ is strongly convex with a Lipschitz continuous gradient for $i\in\cN$. A method based on ADMM taking proximal-gradient steps, IC-ADMM, is proposed to reduce the computational work of ADMM due to exact minimizations required in each iteration. Under the assumption that the smallest eigenvalue of the un-oriented Laplacian of $\cG$ is known at all agents, it is shown that IC-ADMM sequence converges when each $f_i$ is strongly convex -- no rate result is provided for this case; on the other hand, linear convergence is established in the absence of the merely convex (possibly non-smooth) term $\rho_i$ and assuming each $C_i$ has full column-rank in addition to the previous assumptions required for establishing the convergence result. In a similar spirit, to overcome the costly exact minimizations required in ADMM, an exact first-order algorithm (EXTRA) is proposed in~\cite{shi2015extra} for minimizing 
$\bar{f}$ over {an} undirected static network $\cG$. When 
$\bar{f}$ is smooth and strongly convex with modulus $\bar{\mu}>0$, it is shown that the algorithm has linear convergence {without assuming each 
$f_i$ to be strongly convex provided that the step-size $\alpha>0$, constant among all the nodes, is sufficiently small, i.e., $\alpha=\cO(\bar{\mu}/L_{\max}^2)$}. In a follow up work, Extra-Push~\cite{zeng2015extrapush} has been proposed that extends EXTRA to handle strongly connected, directed static networks using push-sum protocol. Convergence of Extra-Push, without providing any rate, has been shown under boundedness assumption on the iterate sequence; moreover, under the assumption that the stationary distribution, $\phi\in\reals^{|\cN|}$, of the column-stochastic mixing matrix that represents the static directed network is known, i.e., each node $i\in\cN$ knows $\phi_i>0$, they relax the boundedness assumption on the iterate sequence, and show that a variant of Extra-Push converges at a linear rate if each $f_i$ is smooth and strongly convex with $\mu_i>0$ for $i\in\cN$ -- note that assuming each node $i\in\cN$ knows $\phi_i$ \emph{exactly} is a fairly strong assumption in a decentralized optimization setting. In~\cite{Dextra16}, Xi et al. also combined EXTRA with the push-sum protocol to obtain DEXTRA to minimize strongly convex $\bar{f}$ over a static directed network. In addition to assumptions on $\{f_i\}_{i\in\cN}$ in~\cite{zeng2015extrapush}, by further assuming that $\grad f_i$ bounded over $\reals^n$ for $i\in\cN$, which implies boundedness of the iterate sequence,
it is shown that the iterate sequence converges linearly when the constant step-size $\alpha$, fixed for all $i\in\cN$, is chosen carefully belonging to a non-trivial interval $[\alpha_{\min},\alpha_{\max}]$ such that $\alpha_{\min}>0$ -- note that the boundedness on each $\grad f_i$ is a strong requirement and clearly it is not satisfied by commonly used quadratic loss function. 
In a follow up paper~\cite{xi2016add}, Xi and Khan proposed Accelerated Distributed Directed Optimization (ADD-OPT) where they improved on the nontrivial step-size condition of DEXTRA and showed that the iterates converge linearly when the constant step-size $\alpha$ is chosen sufficiently small -- assuming that the directed network topology is static and each $f_i$ is strongly convex with Lipschitz continuous gradients (without assuming boundedness as in~\cite{Dextra16}). In a more general setting, Nedi\'{c} and Olshevsky~\cite{nedic2016stochastic} proposed a stochastic (sub)gradient-push for minimizing strongly convex $\bar{f}$ on time-varying directed graphs without assuming differentiability when {the stochastic error in subgradient samples has zero mean and bounded standard deviation}.
{When $\mu_i>0$ for all $i\in\cN$, choosing a diminishing step-size sequence, they were able to show $\cO(\log(k)/k)$ rate result provided that the iterate sequence stays bounded -- the boundedness assumption on the iterate sequence can be removed by assuming that functions are smooth, having Lipschitz continuous gradients.}
In \cite{nedich2016achieving}, Nedi\'{c} et al. proposed distributed inexact gradient methods referred to as DIGing and Push-DIGing for time-varying undirected and directed networks, respectively. Assuming $f_i$ is strongly convex with Lipschitz continuous gradient for each $i\in\cN$, it is shown that the iterate sequence converges linearly provided that the constant step-size $\alpha$, fixed for all $i\in\cN$, is chosen sufficiently small.

For constrained consensus optimization, other than few exceptions, e.g.,~\cite{zhu2012distributed,yuan2011distributed,chang2014distributed,mateos2015distributed,chang2014proximal,aybat2016primal}, the existing methods require that each node compute a projection on the local set $\chi_i$ in addition to consensus and (sub)gradient steps, e.g., \cite{nedic2010constrained,srivastava2010distributed}. Moreover, among those few exceptions, 
only \cite{chang2014distributed,mateos2015distributed,chang2014proximal,aybat2016primal} can handle agent-specific constraints without assuming global knowledge of the constraints by all agents. However, \emph{no} rate results in terms of suboptimality, local infeasibility, and consensus violation exist for the primal-dual distributed methods in~\cite{chang2014distributed,mateos2015distributed,chang2014proximal} when implemented for the agent-specific \emph{conic} constraint sets $\chi_i=\{x:A_ix-b_i\in \cK_i\}$ studied in this paper. 
In~\cite{chang2014distributed}, a consensus-based distributed primal-dual perturbation (PDP) algorithm using a diminishing step-size sequence is proposed. The objective is to minimize a composition of a global network function (smooth) with the sum of local objective functions (smooth), i.e., $\cF(\sum_{i\in\cN}f_i(x))$, subject to local compact sets and inequality constraints on the summation of agent specific constrained functions, i.e., $\sum_{i\in\cN}g_i(x)\leq 0$, over a time-varying directed network. They showed that the local primal-dual iterate sequence converges to a global optimal primal-dual solution; however, no rate result was provided. The proposed PDP method can also handle non-smooth constraints with similar convergence guarantees.
In a recent work~\cite{mateos2015distributed},
the authors proposed a distributed algorithm on time-varying directed networks for solving saddle-point problems subject to consensus constraints. The algorithm can also solve consensus optimization problems with \emph{inequality} constraints that can be written as summation of local convex functions of local and global variables. 
It is shown that using a carefully selected decreasing step-size sequence, the ergodic average of primal-dual sequence converges with $\mathcal{O}(1/\sqrt{k})$ rate in terms of saddle-point evaluation error; however, when applied to constrained optimization problems, \emph{no} rate in terms of either suboptimality or infeasibility is provided. 
{In~\cite{chang2014proximal}, a closely related paper to ours, a proximal dual consensus ADMM method, PDC-ADMM,
is proposed by Chang to minimize $\bar{\varphi}$ subject to a coupling equality and agent-specific constraints over both static and time-varying undirected networks -- for the time-varying topology, they assumed that agents are on/off and communication links fail randomly with certain probabilities. 
Each agent-specific set is assumed to be an intersection
of a polyhedron and a ``simple" compact set. More precisely, the goal is to solve $\min_x\{\sum_i\varphi_i(x_i): \sum_{i\in\cN}C_ix_i=d,~x_i\in\chi_i~i\in\cN\}$ where $\varphi_i=\rho_i+f_i$ is composite convex, $\chi_i=\{x_i: A_ix_i\geq b_i, x_i\in\cS_i\}$ and $\cS_i$ is a convex compact set. Clearly, by properly choosing the primal constraint $\sum_{i\in\cN}C_ix_i=d$ one can impose consensus on $\{x_i\}_{i\in\cN}$. The polyhedral constraints defining $\chi_i$ are handled using a penalty formulation without requiring projection onto them. 
It is shown that both for static and time-varying cases, PCD-ADMM have $\cO(1/k)$ ergodic convergence rate in the mean for suboptimality and infeasibility when each $f_i$ is strongly convex and differentiable with a Lipschitz continuous gradient for $i\in\cN$.} More recently, in~\cite{aybat2016primal}, Aybat and Yazdandoost Hamedani proposed a distributed primal-dual method to solve \eqref{eq:central_problem} when $\varphi_i=\rho_i+f_i$ is composite convex. Assuming $f_i$ is smooth, $\cO(1/k)$ ergodic rate is shown for suboptimality and infeasibility. In this paper, we aim to improve on this rate by further assuming $\sum_i\varphi_i$ is strongly convex to achieve $\cO(1/k^2)$ ergodic rate.

\textbf{Contribution.} To the best of our knowledge, only a handful of methods, e.g.,~\cite{chang2014distributed,mateos2015distributed,chang2014proximal,aybat2016primal} can handle consensus problems, similar to \eqref{eq:central_problem}, with agent-specific \emph{local} constraint sets $\{\chi_i\}_{i\in\cN}$ without requiring each agent $i\in\cN$ to project onto $\chi_i$. However, no rate results in terms of suboptimality, local infeasibility, and consensus violation exist for the 
distributed methods in \cite{chang2014distributed,mateos2015distributed,chang2014proximal} when implemented for \emph{conic}
sets $\{\chi_i\}_{i\in\cN}$ studied in this paper; moreover, none of these four methods exploits the strong convexity of the sum function $\bar{\varphi}=\sum_{i\in\cN}\varphi_i$. We believe 
DPDA and DPDA-TV proposed in this paper is one of the first decentralized algorithms to solve \eqref{eq:central_problem} with $\cO(1/k^2)$ ergodic rate guarantee on both sub-optimality and infeasibility.
More precisely, we show that when $\bar{\varphi}$ is strongly convex and each $\varphi_i$ is composite convex with smooth $f_i$ for $i\in\cN$, our proposed method reduces the suboptimality and infeasibility with $\mathcal{O}(1/{k^2})$ rate as $k$, the number primal-dual iterations, increases, and it requires $\cO(k)$ and $\cO(k\log(k))$ local communications for all $k$ iterations in total when the network topologies are static and time-varying, respectively. To the best of our knowledge, this is the best rate result for our 
setting. Moreover, the proposed methods
do not require the agents 
to know any global parameter depending on the entire network topology, e.g., the second smallest eigenvalue of the Laplacian.

\textbf{Notation.} Throughout 
$\norm{.}$ denotes either the Euclidean norm {or the spectral norm}, and $\fprod{\theta,w}\triangleq \theta^\top w$ for $\theta,w\in\reals^n$. Given a convex set $\cS$, let $\sigma_{\cS}(.)$ denote its support function, i.e., $\sigma_{\cS}(\theta)\triangleq\sup_{w\in \cS}\fprod{\theta,~w}$, let $\ind{S}(\cdot)$ denote the indicator function of $\cS$, i.e., $\ind{S}(w)=0$ for $w\in\cS$ and equal to $+\infty$ otherwise, and let $\cP_{\cS}(w)\triangleq\argmin\{\norm{v-w}:\ v\in\cS\}$ denote the projection onto $\cS$. For a closed convex set $\cS$, we define the distance function as $d_{\cS}(w)\triangleq\norm{\cP_{\cS}(w)-w}$. Given a convex cone $\cK\in\reals^m$, let $\mathcal{K}^*$ denote its dual cone, i.e., $\mathcal{K}^*\triangleq\{\theta\in\reals^{m}: \ \langle \theta,w\rangle \geq 0\ \ \forall w\in\mathcal{K}\}$, and $\cK^\circ\triangleq -\cK^*$ denote the polar cone of $\cK$. Note that for a given cone $\cK\in\reals^m$, $\sigma_{\cK}(\theta)=0$ for $\theta\in\cK^\circ$ and equal to $+\infty$ if $\theta\not\in\cK^\circ$, i.e., $\sigma_{\cK}(\theta)=\ind{\cK^\circ}(\theta)$ for all $\theta\in\reals^m$.
Given a convex function $g:\reals^n\rightarrow\reals\cup\{+\infty\}$, its convex conjugate is defined as $g^*(w)\triangleq\sup_{\theta\in\reals^n}\fprod{w,\theta}-g(\theta)$. $\otimes$ denotes the Kronecker product, $\one_n\in \mathbb{R}^{n}$ be the vector all ones, $\id_n$ is the $n\times n$ identity matrix. $\sy^n_{++}$ $(\sy^n_+)$ denotes the cone of symmetric positive (semi)definite matrices. For $Q\succ 0$, i.e., 
$Q\in\sy^n_{++}$, $Q$-norm is defined as $\norm{z}_Q\triangleq \sqrt{z^\top Q z}$. Given $Q\in\sy^n_{+}$, $\lambda_{\min}^+(W)$ denotes the smallest positive eigenvalue of $Q$. $\Pi$ denotes the Cartesian product. Finally, for $\theta\in\reals^n$, we adopt $(\theta)_+\in\reals^n_{+}$ to denote $\max\{\theta, \textbf{0}\}$ where max is computed componentwise.
\vspace*{-3mm}
\subsection{Preliminary} \label{sec:pda}
Let $\cX$ and $\cY$ be finite-dimensional vector spaces. In a recent paper, Chambolle and Pock~\cite{chambolle2015ergodic} proposed a primal-dual algorithm~(PDA) for the following convex-concave saddle-point problem:\vspace*{-2mm}
{\small
\begin{align}
\label{eq:saddle-problem}
\min_{\bx\in\cX}\max_{\by\in\cY}\cL(\bx,\by)\triangleq\Phi(\bx)+\fprod{T\bx,\by}-h(\by),\quad \hbox{\normalsize where}\quad \Phi(\bx)\triangleq\rho(\bx)+g(\bx)\vspace*{-2mm}
\end{align}
}%
{is a {\it strongly convex} function with modulus $\mu$ such that} $\rho$ and $h$ are possibly non-smooth convex functions, $g$ is convex and has a Lipschitz continuous gradient defined on $\dom \rho$ with Lipschitz constant $L$, and {$T$ is a linear map}. Given some positive step-size sequences $\{\tau^k,\kappa^k,\eta^k\}_{k\geq 0}$ and the initial iterates $\bx^0,\by^0$, PDA consists of two proximal-gradient steps:
\vspace*{-2mm} 
{\small
\begin{subequations}
\label{eq:PDA}
\begin{align}
\by^{k+1}&\gets\argmin_{\by} h(\by)-\fprod{T(\bx^k+\eta^k(\bx^{k}-\bx^{k-1})),~\by}+{D_k(\by,{\by}^k)},
\label{eq:PDA-y}\\
\bx^{k+1}&\gets\argmin_{\bx} \rho(\bx)+g(\bx^k)+\fprod{\grad g(\bx^k),~\bx-\bx^k}+\fprod{T\bx,~\by^{k+1}}+\frac{1}{2\tau^k}\norm{\bx-{\bx}^k}^2,
\label{eq:PDA-x}
\end{align}
\end{subequations}
}%
{where $D_k$ is a Bregman distance function such that $D_k(\by,\bar{\by})\geq\tfrac{1}{2\kappa^k}\norm{\by-\bar{\by}}^2$ for any $\by$ and $\bar{\by}$ and $k\geq 0$.}
In~\cite{chambolle2015ergodic}, a simple proof for the ergodic convergence is provided for~\eqref{eq:PDA}; indeed, it is shown that
{when the convexity modulus for $\rho$ and {$g$} are $\mu$ and $0$, resp.,} and
if $\tau^k,\kappa^k,\eta^k>0$ are chosen such that $\frac{1}{\tau^k}+\mu\geq \frac{1}{\tau^{k+1}\eta^{k+1}}$, $(\frac{1}{\tau^k}-L)\geq\norm{T}^2\kappa^k$, and $\kappa^k=\kappa^{k+1}\eta^{k+1}$ for all $k\geq 0$, then
{\small
\begin{equation}
\label{eq:DPA-rate}
N_K\left(\cL(\bar{\bx}^K,\by)-\cL(\bx,\bar{\by}^K)\right)+\frac{\kappa^{K}}{\tau^K}~\frac{1}{2\kappa^0}\norm{\bx-\bar{\bx}^K}^2\leq \frac{1}{2\tau^0}\norm{\bx-\bx^0}^2+{D_0(\by,\by^0)},
\end{equation}
}%
for all $\bx,\by\in\cX\times\cY$, where $N_K\triangleq\sum_{k=1}^{K}\frac{\kappa^{k-1}}{\kappa^0}$, $\bar{\bx}^{K}\triangleq N_K^{-1}\sum_{k=1}^{K}\frac{\kappa^{k-1}}{\kappa^0}\bx^k$ and $\bar{\by}^{K}\triangleq N_K^{-1}\sum_{k=1}^{K}\frac{\kappa^{k-1}}{\kappa^0}\by^k$ for all $K\geq 1$. {In~\cite{chambolle2015ergodic}, it is shown that $\{\tau^k,\kappa^k,\eta^k\}_{k\geq 0}$ can be chosen such that $N_k=\cO(k^2)$, $\tau^k=\cO(1/k)$ and $\kappa^k=\cO(k)$ for $k\geq 0$.}

First, in Section~\ref{sec:static}, we discuss a special case of \eqref{eq:saddle-problem}, which will help us develop a decentralized primal-dual algorithm, DPDA, for the consensus optimization problem in~\eqref{eq:central_problem} when the communication network topology is static, and we provide the main results 
{for the} static case in Theorem~\ref{thm:static-error-bounds}.
Next, in Section~\ref{sec:dynamic}, we propose a decentralized algorithm DPDA-TV to solve \eqref{eq:central_problem} when the network topology is time-varying, and we extend our convergence results to time-varying case in Theorem~\ref{thm:dynamic-error-bounds}. Finally, in Section~\ref{sec:numerics}, we test the performance of the proposed methods for solving distributed constrained LASSO problems. 
\noindent \section{A distributed method for a static network topology}
\label{sec:static}
In this section we discuss how 
PDA, stated in \eqref{eq:PDA}, can be implemented to compute an $\sqrt{\epsilon}$-optimal solution to \eqref{eq:central_problem} in a distributed way using only $\cO(1/\sqrt{\epsilon})$ communications over a \emph{static} communication network $\cG$ using only local communications.
{Let $\cG=(\cN,\cE)$ denote a \emph{connected} undirected graph of $N$ computing nodes, where $\cN\triangleq\{1,\ldots,N\}$ and $\mathcal{E}\subseteq \mathcal{N}\times \mathcal{N}$ denotes the set of edges -- without loss of generality assume that $(i,j)\in \mathcal{E}$ implies $i< j$. Suppose nodes $i$ and $j$ can exchange information only if $(i,j) \in \cE$. Let $\mathcal{N}_i\triangleq\{j\in\mathcal{N} : (i, j) \in \mathcal{E} \text{ or } (j, i) \in \mathcal{E}\}$ denote the set of neighboring nodes of $i \in \mathcal{N}$, and $d_i\triangleq |\mathcal{N}_i|$ is the degree of node $i\in \mathcal{N}$.}

Let $x_i\in\reals^n$ denote the \emph{local} decision vector of node $i\in\cN$. By taking advantage of the fact that $\cG$ is {\it connected}, we can reformulate \eqref{eq:central_problem} as a 
{\it consensus} optimization problem:
{\small
\begin{align}\label{eq:dist_problem}
\min_{\substack{x_i\in\reals^n,\ i\in\cN}} \left\{\sum_{i\in \mathcal{N}}{\varphi_i(x_i)}~|~x_i=x_j:~\lambda_{ij},\ \forall (i,j)\in \mathcal{E},\quad  A_ix_i-b_i \in \mathcal{K}_i:\ \theta_i,\ \forall{i}\in\mathcal{N} \right\},
\end{align}
}%
where $\lambda_{ij}\in\reals^n$ and $\theta_i\in\reals^{m_i}$ are the corresponding dual variables. Let $\bx=[x_i]_{i\in\cN}\in\reals^{n|\cN|}$. The consensus constraints $x_i=x_j$ for $(i,j)\in \mathcal{E}$ can be formulated as $M{\bf x}=0$, where $M\in \mathbb{R}^{n|\mathcal{E}|\times n|\mathcal{N}|}$ is a block matrix such that $M=H\otimes \id_n$ where $H$ is the oriented edge-node incidence matrix, i.e., the entry $H_{(i,j),l}$, corresponding to edge $(i,j)\in \mathcal{E}$ and node $l\in \mathcal{N}$, is equal to $1$ if $l=i$, $-1$ if $l=j$, and $0$ otherwise.
Note that $M^\T M=H^\T H\otimes \id_n=\Omega\otimes \id_n$, where $\Omega\in \mathbb{R}^{|\mathcal{N}|\times |\mathcal{N}|}$ denotes the graph Laplacian of $\cG$, i.e., $\Omega_{ii}=d_i$, $\Omega_{ij}=-1$ if $(i,j)\in\cE$ or $(j,i)\in\cE$, and equal to $0$ otherwise. 

Since $x^*$ is the unique solution to \eqref{eq:central_problem} and \eqref{eq:dist_problem}, and since $\bx^*\triangleq\ones\otimes x^*$ satisfies $(\Omega\otimes\id_n)\bx^*=0$, one can reformulate \eqref{eq:dist_problem} as a saddle point problem. Indeed, let $\bx=[x_i]_{i\in\cN}$, $\by=[\btheta^\top \blambda^\top]^\top$ such that $\btheta=[\theta_i]_{i\in\cN}$ and $\blambda=[\lambda_{ij}]_{(i,j)\in\cE}$, then for any $\alpha\geq 0$, one can compute a primal-dual optimal solution to \eqref{eq:central_problem} through solving
{\small
\begin{align}\label{eq:static-saddle}
\min_{{\bf x}}\max_{\by}\mathcal{L}({\bf x}, \by)\triangleq \frac{\alpha}{2}\norm{\bx}_{\Omega\otimes\id_n}^2+ \sum_{i\in\mathcal{N}}\bigg(\varphi_i(x_i) +\langle \theta_i,A_ix_i-b_i \rangle -\sigma_{\mathcal{K}_i}(\theta_i) \bigg)+\langle \blambda, M{\bf x}\rangle.
\end{align}
}%
Next, we consider implementation of PDA in \eqref{eq:PDA} 
to solve \eqref{eq:static-saddle} for some $\alpha\geq 0$.
\begin{defn}
\label{def:bregman-s}
Let $\cX\triangleq\Pi_{i\in\cN}\reals^n$ and $\cX\ni\bx=[x_i]_{i\in\cN}$; $\cY\triangleq\Pi_{i\in\cN}\reals^{m_i}\times\reals^{n|\cE|}$,  $\cY\ni\by=[\btheta^\top \blambda^\top]^\top$ such that $\btheta=[\theta_i]_{i\in\cN}\in\reals^m$ and $\blambda=[\lambda_{ij}]_{(i,j)\in\cE}\in\reals^{m_0}$, where $m\triangleq\sum_{i\in\cN}{m_i}$, and {$m_0\triangleq n|\cE|$}. Given parameters $\gamma^k>0$ and $\kappa_i^k>0$ for $i\in\cN$, let $\mathbf{D}_{\gamma^k}\triangleq\frac{1}{\gamma^k}\id_{m_0}$, $\mathbf{D}_{\kappa^k}\triangleq\diag([\frac{1}{\kappa_i^k}\id_{m_i}]_{i\in\cN})$,
 and $\mathbf{D}_{\kappa^k,\gamma^k}\triangleq \begin{bmatrix} \bD_{\kappa^k}& \zero \\ \zero & \bD_{\gamma^k} \end{bmatrix}$.
\end{defn}
\begin{defn}
\label{def:problem-components-static}
Let $\Phi,\varphi:\cX\rightarrow\reals\cup\{\infty\}$ such that $\Phi(\bx)=\rho(\bx)+g(\bx)$ and $\varphi(\bx)=\rho(\bx)+f(\bx)$ where $\rho(\bx)\triangleq\sum_{i\in\cN}\rho_i(x_i)$, $f(\bx)\triangleq\sum_{i\in\cN}f_i(x_i)$, and $g(\bx)\triangleq f(\bx)+\frac{\alpha}{2}\norm{\bx}_{\Omega\otimes\id_n}^2$, and let $h:\cY\rightarrow\reals\cup\{\infty\}$ such that $h(\by)\triangleq \sum_{i\in\cN}\sigma_{\mathcal{K}_i}(\theta_i)+\fprod{b_i,\theta_i}$. Define the block-diagonal matrix $A\triangleq\diag([A_i]_{i\in\mathcal{N}})\in\reals^{m \times n|\cN|}$ and $T=[A^\top~M^\top]^\top$.
\end{defn}
Given some positive parameters $\gamma^k,\tau^k>0$, $\kappa_i^k>0$ for $i\in\cN$ -- we shortly discuss how to select them, we define the Bregman function $D_k(\by,\bar{\by})=\frac{1}{2}\norm{\by-\bar{\by}}_{\mathbf{D}_{\kappa^k,\gamma^k}}^2$ for each $k\geq 0$. Hence, given $\Phi$, $h$ and $T$ as in Definition~\ref{def:problem-components-static}, and the initial iterates $\bx^0$ and $\by^0=[{\btheta^0}^\top {\blambda^0}^\top]^\top$, the PDA iterations given in \eqref{eq:PDA} take the following form: 
\vspace*{-1mm}
{\small
\begin{subequations}
\label{eq:pock-pd}
\begin{align}
\theta_i^{k+1}\gets\argmin_{\theta_i}~&\sigma_{\mathcal{K}_i}(\theta_i) -\langle A_i(x_i^{k}+\eta^k(x_i^k-x_i^{k-1}))-b_i,~\theta_i\rangle +{1\over 2\kappa_i^k}\|\theta_i-\theta_i^k\|^2, \ \ \ \ i\in\mathcal{N} \label{eq:pock-pd-theta-s}\\
{\blambda^{k+1}}\gets\argmin_{\blambda}~& -\langle M({\bf x}^{k}+\eta^k ({\bf x}^k-\bx^{k-1})),~\blambda \rangle +{1\over 2\gamma^k}\|\blambda-\blambda^k\|^2 \label{eq:pock-pd-lambda-s}\\
&=\blambda^k+\gamma^k M({\bf x}^{k}+\eta^k({\bf x}^k-\bx^{k-1})), \label{eq:pock-pd-lambda} 
\\
{\bf x}^{k+1}\gets \argmin_{{\bf x}}~&~\langle \blambda^{k+1},~ M{\bf x} \rangle+ \fprod{\nabla g(\bx^k),~\bx}+\sum_{i\in\mathcal{N}}\Big[ \rho_i(x_i)
+\langle A_ix_i-b_i,\theta_i^{k+1}\rangle+\frac{1}{2\tau^k}\|x_i-x_i^k\|^2\Big]  \nonumber \\
=\argmin_{{\bf x}}~&\sum_{i\in\mathcal{N}}\Big[ \rho_i(x_i)+\langle {\nabla f_i(x_i^k)},~x_i\rangle +\langle A_ix_i-b_i,\theta_i^{k+1}\rangle+{1\over 2\tau^k}\|x_i-x_i^k\|^2\Big]  \nonumber \\
&+\langle \blambda^{k+1}, M{\bf x} \rangle + \alpha\fprod{(\Omega\otimes \id_n)\bx^k,~\bx}. 
\end{align}
\end{subequations}
}%
Since $\cK_i$ is a cone, $\prox{\kappa_i^k\sigma_{\cK_i}}(\cdot)=\cP_{\cK_i^\circ}(\cdot)$; hence, $\theta_i^{k+1}$ can be written in closed form as   
{\small
\begin{align}
\theta_i^{k+1}=\cP_{\cK_i^\circ}\Big(\theta_i^k+\kappa_i^k\Big(A_i(x_i^{k}+\eta^k(x_i^k-x_i^{k-1}))-b_i\Big)\Big), \ \ \ \ i\in\mathcal{N}. \nonumber
\end{align}
}%
Using recursion in~\eqref{eq:pock-pd-lambda}, 
we can write $\blambda^{k+1}$ as a partial summation of primal iterates {$\{\bx^\ell\}_{\ell=0}^k$},
i.e., $\blambda^{k+1}=\blambda^0+\sum_{\ell=0}^{k} \gamma^\ell M({\bf x}^{\ell}+{\eta^\ell}({\bf x}^\ell-\bx^{\ell-1}))$ for $k\geq 0$. Let $\blambda^0\gets {\mathbf{0}}$, and define $\{\bs^k\}_{k\geq 0}$ such that {$\bs^0=\mathbf{0}$} and {$\bs^{k+1} = \bs^k+ \gamma^k \big(\bx^k+ \eta^k (\bx^k - \bx^{k -1})\big)$} for $k\geq 0$; {hence, $\blambda^k=M\bs^{k}$ for $k\geq 0$. Using the fact that $M^\top M=\Omega\otimes \id_n$, we obtain}
{\small
\[
\textstyle\langle M{\bf x},~\blambda^{k+1} \rangle=~\langle {\bf x},~(\Omega \otimes \id_n)\bs^{k+1}\rangle=\sum_{i\in\mathcal{N}}\langle { x_i},~\sum_{j\in\mathcal{N}_i}(s^{k+1}_i-s^{k+1}_j)\rangle.\]
}%
Thus, PDA iterations given in~\eqref{eq:pock-pd} for the static graph $\cG$ can be computed in a \emph{decentralized} way, via the node-specific computations as in time-invariant distributed primal dual algorithm displayed in Fig.~\ref{alg:PDS} below.
\begin{figure}[htpb]
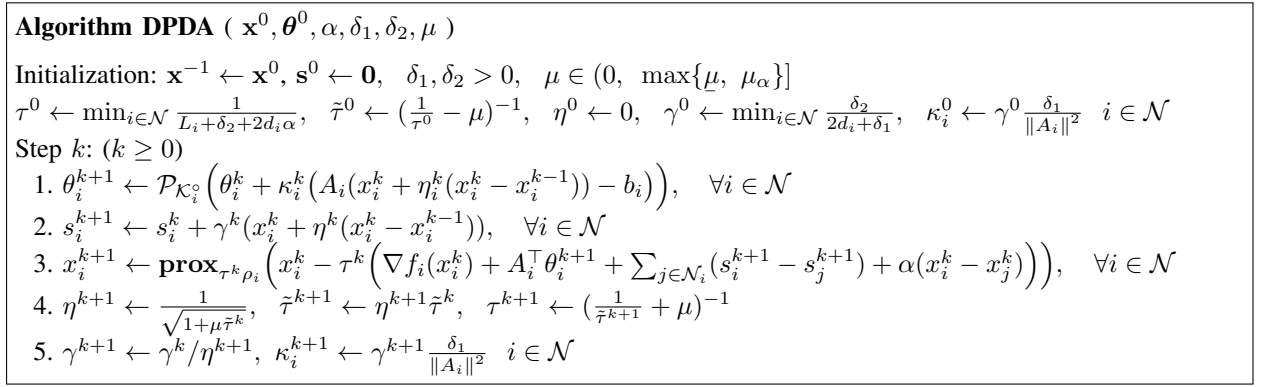

\centering
\framebox{\parbox{0.99\columnwidth}{
{\small
\textbf{Algorithm DPDA} ( $\bx^{0},\btheta^0,\alpha,{\delta_1,\delta_2},\mu$ ) \\[1.5mm]
Initialization: $\bx^{-1}\gets \bx^0$, $\bs^0\gets\mathbf{0},\ \ \delta_1,\delta_2>0,\ \ \mu\in(0,~\max\{\ubar{\mu},~\mu_\alpha\}]$\\
$\tau^0\gets\min_{i\in\cN}\frac{1}{L_i+\delta_2+2d_i\alpha},\ \ \tilde{\tau}^0\gets(\frac{1}{\tau^0}-\mu)^{-1},\ \ \eta^0\gets 0,\ \ \gamma^0\gets\min_{i\in\cN}\frac{\delta_2}{2d_i+\delta_1},\ \ \kappa_i^0\gets\gamma^0\frac{\delta_1}{\norm{A_i}^2}\ \ i\in\cN$\\
Step $k$: ($k \geq 0$)\\
\text{ } 1. $\theta_i^{k+1}\gets\cP_{\cK_i^\circ}\Big(\theta_i^k+\kappa_i^k\big(A_i(x_i^{k}+\eta_i^k(x_i^k-x_i^{k-1}))-b_i\big)\Big),\quad \forall i\in\cN$\\
\text{ } 2. $s_i^{k+1}\gets  s_i^k+\gamma^k (x_i^{k}+\eta^k(x_i^k-x_i^{k-1})),\quad \forall i\in\cN$ \\
\text{ } 3. $x_i^{k+1}\gets\prox{\tau^k\rho_i}\Big(x_i^k-\tau^k\Big(\nabla f_i(x_i^k)+A_i^\top\theta^{k+1}_i+\sum_{j\in\mathcal{N}_i}(s^{k+1}_i-s^{k+1}_j)+\alpha(x_i^k-x_j^{k})\Big)\Big),\quad \forall i\in\cN$ \\
\text{ } 4. $\eta^{k+1}\gets\frac{1}{\sqrt{1+\mu {\ttau}^k}},\ \ {\ttau}^{k+1}\gets \eta^{k+1} {\ttau}^k,\ \ \tau^{k+1}\gets(\frac{1}{\ttau^{k+1}}+\mu)^{-1}$\\
\text{ } 5. $\gamma^{k+1}\gets\gamma^k/\eta^{k+1},\ \kappa_i^{k+1}\gets\gamma^{k+1}\frac{{\delta_1}}{\norm{A_i}^2}\ \ i\in\cN$
 }}}
\caption{\small Distributed Primal Dual Algorithm for static $\cG$ (DPDA)}
\label{alg:PDS}
\vspace*{-3mm}
\end{figure}

\begin{defn}
\label{def:W}
Let $W\in\mathbb{S}^{|\cN|}$ such that $W_{ij}=W_{ji}<0$ for $(i,j)\in\cE$, $W_{ij}=W_{ji}=0$ for $(i,j)\notin \cE$, and $W_{ii}=-\sum_{j\in{\cN_i}}W_{ij}$ for $i\in\cN$.
\end{defn}
\begin{remark}
\label{rem:rs-convexity}
According Assumption~\ref{assump:sconvex}, when $\ubar{\mu}>0$, $f(\bx)=\sum_{i\in\cN}f_i(x_i)$ is strongly convex
with modulus $\ubar{\mu}$. That said, as emphasized in the introduction, although $\bar{f}(x)=\sum_{i\in\cN} f_i(x)$ is strongly convex with modulus $\bar{\mu}>0$, it is possible that $f$ may not when $\ubar{\mu}=0$.
\end{remark}
Inspired from Proposition 3.6. in \cite{shi2015extra}, we show that by suitably regularizing $f$, one can obtain a strongly convex function when $\ubar{\mu}=0$.
\begin{lemma}\label{lem:restricted-convex-static}
Consider $f(\bx)=\sum_{i\in\cN}f_i(x_i)$ under Assumption~\ref{assump:sconvex}, suppose $\ubar{\mu}=0$. Given $\alpha>0$ and {$W$ as in Definition~\ref{def:W}}, let $f_\alpha(\bx)\triangleq f(\bx)+\alpha~r(\bx)$, where 
$r(\bx)\triangleq \tfrac{1}{2}\norm{\bx}_{W\otimes\id_n}^2$.
Then $f_\alpha$ is strongly convex
with modulus ${\mu}_\alpha \triangleq \frac{\bar{\mu}/|\cN|~+\alpha {\lambda}_2}{2}-\big(\big(\frac{\bar{\mu}/|\cN|~-\alpha {\lambda}_2}{2}\big)^2+4\bar{L}^2\big)^{\nicefrac{1}{2}}>0$ for any $\alpha>\frac{4}{ {\lambda}_2 \bar{\mu}}\sum_{i\in\cN}L_i^2$, where $\bar{L}=\sqrt{\frac{\sum_{i\in\cN}L_i^2}{|\cN|}}$ and ${\lambda}_2=\lambda_{\min}^+(W)$. 
\end{lemma}
\begin{remark}
\label{rem:g}
When $\ubar{\mu}>0$, i.e., all $f_i$'s are strongly convex,
the parameter $\alpha$ can be set to zero; hence, $g(\bx)=f(\bx)$ is strongly convex
with modulus $\mu_g=\ubar{\mu}$. Otherwise, when $\ubar{\mu}=0$, $\alpha$ should be chosen according to Lemma \ref{lem:restricted-convex-static}; hence, $g(\bx)=f_\alpha(\bx)$ is strongly convex
with modulus $\mu_g=\mu_\alpha$. The condition $\alpha>\frac{4}{ \bar{\mu}\lambda^+_{\min}(W)}\sum_{i\in\cN}L_i^2$ is similar to the one in~\cite{shi2015extra}, where they also have a parameter ${W}\in\sy^{|\cN|}_+$ for their algorithm and $\alpha$ should be greater than $\frac{|\cN| L_{\max}^2}{2\bar{\mu}\lambda_{\min}^+({W})}$ and $L_{\max}=\max_{i\in\cN}L_i$.
\end{remark}
Next, 
we quantify the suboptimality and infeasibility of the DPDA iterate sequence. 
\begin{theorem}\label{thm:static-error-bounds}
{Suppose Assumption~\ref{assump:saddle-point} holds.}
Let {$\{\bx^k,\btheta^k\}_{k\geq 0}$} be the 
sequence generated by Algorithm DPDA, displayed in Fig.~\ref{alg:PDS}, initialized from an arbitrary $\bx^0$ and $\btheta^0=\mathbf{0}$. Then $\{\bx^k\}_{k\geq 0}$ converges to 
$\bx^*=\ones\otimes x^*$ such that $x^*$ is the optimal solution to \eqref{eq:central_problem};
moreover, the following error bounds hold for all $K\geq 1$:
{\small
\begin{align*}
\max\big\{ \abs{\Phi(\bar{\bx}^K)-\varphi({\bf x}^*)},\ \norm{M\bar{\bf x}^K}+\sum_{i\in\cN}\norm{\theta_i^*} d_{\mathcal{K}_i}(A_i\bar{\bf x}_i^K-b_i) \big\}\leq \Theta_0/ N_K,\quad 
{\norm{\bx^K-\bx^*}^2\leq \frac{\tilde{\tau}^K}{\gamma^K}~2\gamma^0\Theta_0,}
\end{align*}
}%
where $\Theta_0\triangleq {1\over 2\gamma^0}+\sum_{i\in\mathcal{N}}\Big[{1\over 2\tau^0}\|x_i^0-{x^*}\|^2+{2\over \kappa_i^0}\|\theta^*_i\|^2\Big]$, $\bar{\bx}^{K}=N_K^{-1}\sum_{k=1}^{K}{\gamma^{k-1}}\bx^k$, and $N_K=\sum_{k=1}^{K}{\gamma^{k-1}}=\cO(K^2)$. {Moreover, ${\tilde{\tau}^K}/{\gamma^K}=\cO(1/K^2)$.}
\end{theorem}
\begin{remark}
Note that the result in Theorem \ref{thm:static-error-bounds} can be extended to weighted graphs by replacing the Laplacian matrix $\Omega$ in $g(\bx)=f(\bx)+\tfrac{\alpha}{2}\norm{\bx}_{\Omega\otimes\id_n}^2$  with a weighted Laplacian $W$ as in Definition~\ref{def:W}, and also replacing consensus constraint $M\bx=0$ in \eqref{eq:dist_problem} with $(W\otimes\id_n)\bx=0$.
\end{remark}\vspace*{-6mm}
\section{A distributed method for a time-varying communication network}
\label{sec:dynamic}
In this section we develop a distributed primal-dual algorithm for solving \eqref{eq:central_problem} when the communication network topology is \emph{time-varying}. We will adopt the following definition and assumption for the \emph{time-varying} network model.

\begin{defn}
\label{def:neighbors}
Given $t\geq 0$, for an undirected graph $\cG^t=(\cN,\cE^t)$, let $\cN_i^t\triangleq\{j\in\cN:\ (i,j)\in\cE^t~\hbox{ or }~(j,i)\in\cE^t\}$ denote the set of neighboring nodes of $i \in \mathcal{N}$, and $d_i^t\triangleq |\mathcal{N}_i^t|$ represent the degree of node $i\in \mathcal{N}$ at time $t$; for a directed graph $\cG^t=(\cN,\cE^t)$, let $\cN^{\,t,{\rm in}}_i\triangleq\{j\in\cN:\ (j,i)\in\cE^t\}\cup\{i\}$ and $\cN^{\,t,{\rm out}}_i\triangleq\{j\in\cN:\ (i,j)\in\cE^t\}\cup\{i\}$ denote the in-neighbors and out-neighbors of node $i$ at time $t$, respectively; and $d_i^t\triangleq |\cN^{\,t,{\rm out}}_i|$ be the out-degree of node $i$.
\end{defn}
\begin{assumption}
\label{assump:communication_general}
Suppose that $\{\cG^t\}_{t\in\reals_+}$ is a collection of either all directed or all undirected graphs. When $\cG^t$ is an undirected graph, node $i\in\cN$ can send and receive data to and from $j\in\cN$ at time $t$ only if $j\in\cN_i^t$, i.e., $(i,j) \in \cE^t$ or $(j,i) \in \cE^t$; on the other hand, when $\cG^t$ is a directed graph, node $i\in\cN$ can receive data from $j\in\cN$ only if $j\in\cN_i^{\,t,{\rm in}}$, i.e., $(j,i)\in\cE^t$, and can send data to $j\in\cN$ only if $j\in\cN_i^{\,t,{\rm out}}$, i.e., $(i,j)\in\cE^t$.
\end{assumption}

We assume a \emph{compact} domain, i.e., let ${\Delta_i}\triangleq\max_{x_i,x'_i\in \dom \rho_i}\|x-x'\|$ and ${\Delta}\triangleq\max_{i\in\mathcal{N}}\Delta_i<\infty$.  Let $\cB_0\triangleq\{x\in\reals^n:\ \norm{x}\leq {2\Delta}\}$ and $\mathcal{B}\triangleq\Pi_{i\in\cN}\cB_0$; 
and let $\cC$ and $\Ct$ be the sets of consensus and bounded consensus decisions respectively:
\newpage
{
\vspace*{-1.5cm}
\begin{align}
\label{eq:consensus_set}
\cC \triangleq\{{\bf x}\in \mathbb{R}^{n|\mathcal{N}|}:\ \exists\bar{x}\in\mathbb{R}^n\ \st\ x_i=\bar{x}\quad \forall\ i\in\mathcal{N} \}, \qquad \Ct \triangleq \cC\cap\cB.
\end{align}}%
Since $x^*$ is the unique solution to \eqref{eq:central_problem} and since $\bx^*\triangleq\ones\otimes x^*$ satisfies $\cP_\cC(\bx^*)=0$, one can reformulate \eqref{eq:central_problem} as a saddle point problem using $\Ct$. Indeed, Indeed, let $\bx=[x_i]_{i\in\cN}\in\reals^{n|\cN|}$, $\by=[\btheta^\top \blambda^\top]^\top$ such that $\btheta=[\theta_i]_{i\in\cN}$ and $\blambda\in\reals^{n|\cN|}$, then for any $\alpha\geq 0$, one can compute a primal-dual optimal solution to \eqref{eq:central_problem} through solving
{\small
\begin{align}\label{eq:dynamic-saddle}
\min_{\bx} \max_{\by} \cL(\bx,\by)\triangleq \frac{\alpha}{2}d^2_\cC(\bx)+\sum_{i\in\mathcal{N}}\Big(\varphi_i(x_i)+\langle \theta_i,A_ix_i-b_i\rangle-\sigma_{\mathcal{K}_i}(\theta_i)\Big)+\fprod{\blambda,~\bx}-\sigma_{\Ct}(\blambda).
\end{align}}%
Next, we consider a slightly different implementation of PDA in \eqref{eq:PDA} 
to solve \eqref{eq:dynamic-saddle}. 
\begin{defn}
\label{def:bregman}
Let $\cX\triangleq\Pi_{i\in\cN}\reals^n$ and $\cX\ni\bx=[x_i]_{i\in\cN}$; $\cY\triangleq\Pi_{i\in\cN}\reals^{m_i}\times\reals^{m_0}$,  $\cY\ni\by=[\btheta^\top \blambda^\top]^\top$ and $\btheta=[\theta_i]_{i\in\cN}\in\reals^m$, where $m\triangleq\sum_{i\in\cN}{m_i}$ and $m_0\triangleq n|\cN|$. Given parameters $\gamma^k>0$, $\kappa_i^k>0$ for $i\in\cN$, let $\mathbf{D}_{\gamma^k}\triangleq\frac{1}{\gamma^k}\id_{m_0}$, $\mathbf{D}_{\kappa^k}\triangleq\diag([\frac{1}{\kappa_i^k}\id_{m_i}]_{i\in\cN})$, and $\mathbf{D}_{\kappa^k,\gamma^k}\triangleq \begin{bmatrix} \bD_{\kappa^k}& \zero \\ \zero & \bD_{\gamma^k} \end{bmatrix}$.
\end{defn}
\begin{defn}
\label{def:problem-components}
Let $\Phi,{\varphi}:\cX\rightarrow\reals\cup\{\infty\}$ such that $\Phi(\bx)=\rho(\bx)+g(\bx)$ and $\varphi(\bx)=\rho(\bx)+f(\bx)$ where $\rho(\bx)\triangleq\sum_{i\in\cN}\rho_i(x_i)$, $g(\bx)\triangleq f(\bx)+\frac{\alpha}{2}d^2_\cC(\bx)$ and $f(\bx)\triangleq\sum_{i\in\cN}f_i(x_i)$, and let $h:\cY\rightarrow\reals\cup\{\infty\}$ such that $h(\by)\triangleq\sigma_{\cC}(\blambda)+\sum_{i\in\cN}\sigma_{\mathcal{K}_i}(\theta_i)+\fprod{b_i,\theta_i}$. Define the block-diagonal matrix $A\triangleq\diag([A_i]_{i\in\mathcal{N}})\in\reals^{m \times n|\cN|}$ and $T=[A^\top~\id_{n|\cN|}]^\top$.
\end{defn}
Given some positive parameters $\gamma^k,\tau^k>0$, $\kappa_i^k>0$ for $i\in\cN$ -- we shortly discuss how to select them, we define the Bregman function $D_k(\by,\bar{\by})=\frac{1}{2}\norm{\by-\bar{\by}}_{\mathbf{D}_{\kappa^k,\gamma^k}}^2$ for each $k\geq 0$. Hence, given $\Phi$, $h$ and $T$ as in Definition~\ref{def:problem-components}, and the initial iterates $\bx^0$ and $\by^0=[{\btheta^0}^\top {\blambda^0}^\top]^\top$,
the PDA iterations given in \eqref{eq:PDA} 
take the following form for $k\geq 0$:
{\small
\begin{subequations}\label{eq:pock-pd-2}
\begin{align}
&\theta_i^{k+1}\gets\argmin_{\theta_i}\sigma_{\mathcal{K}_i}(\theta_i) -\langle A_i(\xi_i^k+\eta^k(\xi_i^{k}-\xi_i^{k-1}))-b_i,~\theta_i\rangle +{1\over 2\kappa_i^k}\|\theta_i-\theta_i^k\|_2^2,\quad i\in\cN \label{eq:pock-pd-2-theta}\\
&\blambda^{k+1}\gets\argmin_{\bmu} \sigma_{\Ct}(\bmu)-\langle {\bxi}^{k}+\eta^k(\bxi^k-{\bxi}^{k-1}),~\bmu \rangle +{1\over 2\gamma^k}\|\bmu-\bmu^k\|_2^2, \label{eq:pock-pd-2-lambda} \\
& \bmu^{k+1}\gets\blambda^{k+1}, \label{eq:pock-pd-2-mu}\\
&\bx^{k+1}\gets\argmin_{{\bxi}} \rho(\bxi)+\langle \nabla {g(\bxi^k)},~\bxi \rangle +\langle A\bxi-b,~\btheta^{k+1}\rangle+\langle \bxi,~\bmu^{k+1} \rangle +{1\over 2\tau^k}\|\bxi-\bxi^k\|^2, \label{eq:pock-pd-2-x} \\
&\bxi^{k+1}\gets \bx^{k+1}, \label{eq:pock-pd-2-xi}
\end{align}
\end{subequations}
}%
where $\bxi^{-1}=\bxi^0=\bx^0$ and $\bmu^0=\blambda^0$. For $k\geq 0$, using extended Moreau decomposition for proximal operators, $\blambda^{k+1}$ in \eqref{eq:pock-pd-2-lambda} can be computed as
{\small
\begin{align}\label{eq:extended-Moreau}
 \blambda^{k+1}&
 =\prox{\gamma^k\sigma_{\Ct}}(\bmu^k+\gamma^k({\bxi}^{k}+\eta^k(\bxi^k-{\bxi}^{k-1}))=\gamma^k\left(\bom^k-\mathcal{P}_{\Ct}(\bom^k)\right), 
\end{align}
}%
where $\bom^k\triangleq\frac{1}{\gamma^k}\bmu^k+{\bxi}^{k}+\eta^k(\bxi^k-{\bxi}^{k-1})$ for {$k\geq 0$}. Moreover, $\nabla g$ for the $\bx$-step in \eqref{eq:pock-pd-2-x} can be computed as
{\small
\begin{equation}\label{grad-g}
\nabla g(\bxi^k)=\nabla f(\bxi^k)+{\alpha} \cP_{\cC^{\circ}}(\bxi^k)=\nabla f(\bxi^k)+{\alpha}\big(\bxi^k-\cP_{\cC}(\bxi^k)\big).
\end{equation}
}%
For any $\bx=[x_i]_{i\in\cN}\in\cX$, $\mathcal{P}_{\Ct}(\bx)$ and $\cP_\cC(\bx)$ can be computed as
{
\begin{equation}\label{eq:average}
\mathcal{P}_{\Ct}(\bx)=\one\otimes \argmin_{\xi\in\cB_0}\|\xi-{1\over |\mathcal{N}|}\sum_{i\in\mathcal{N}}x_i\|^2=\mathcal{P}_{\mathcal{B}}(\cP_\cC(\bx)),\quad \hbox{and}\quad \cP_\cC(\bx)=\one\otimes p(\bx),
\end{equation}
}%
where $p(\bx)\triangleq{1\over |\mathcal{N}|}\sum_{i\in\mathcal{N}}x_i$, $\cP_{\cB}(\bx)=[\cP_{\cB_0}(x_i)]_{i\in\cN}$ and $\cP_{\cB_0}(x_i)=x_i\min\{1,\frac{2\Delta}{\norm{x_i}}\}$ for $i\in\cN$. Equivalently, $\cP_\cC(\bx)=(W\otimes\id_n)\bx$
for $W\triangleq\frac{1}{|\cN|}\one\one^\top\in\reals^{|\cN|\times|\cN|}$.

Although $\btheta$-step of the PDA implementation in \eqref{eq:pock-pd-2} can be computed locally at each node, computing $\bx$-step and $\blambda$-step require communication among the nodes to evaluate $\cP_{\cC}(\bom^k)$ and $\cP_{\cC}(\bxi^k)$. Indeed, evaluating the average operator $p(.)$ is \emph{not} a simple operation in a decentralized computational setting which only allows for communication among the neighbors. In order to overcome this issue, we will approximate the average operator $p(.)$ using multi-communication rounds, and analyze the resulting iterations as an \emph{inexact} primal-dual algorithm.
We define a \emph{communication round} at time $t$ as an operation over $\cG^t$ such that every node simultaneously sends and receives data to and from its neighboring nodes according to Assumption~\ref{assump:communication_general} -- the details of this operation will be discussed shortly. 
We assume that communication among neighbors occurs \emph{instantaneously}, and nodes operate \emph{synchronously}; and we further assume that for each PDA iteration $k\geq 0$, there exists an approximate averaging operator $\cR^k(\cdot)$ 
which can be computed in a decentralized fashion and approximate $\mathcal{P}_{\cC}(\cdot)$ 
with decreasing approximation error as $k$, {the number of PDA iterations, increases. This \emph{inexact} version of PDA using approximate averaging operator $\cR^k(\cdot)$ and running on time-varying communication network $\{\cG^t\}$ will be called DPDA-TV, of which details will be explained next.}
\begin{assumption}
\label{assump:approximate-average}
Given a time-varying network $\{\cG^t\}_{t\in\reals_+}$ such that $\cG^t=(\cN,\cE^t)$ for $t\geq 0$, suppose that there is a global clock known to all $i\in\cN$. Assume that the local operations requiring to compute $\Pi_{\cK_i}$ as in \eqref{eq:pock-pd-2-theta}, and $\prox{\rho_i}$ and $\grad f_i$ as in \eqref{eq:pock-pd-2-xi} can be completed between two ticks of the clock for all $i\in\cN$ and $k\geq 0$; and  every time the clock ticks a communication round with instantaneous messaging between neighboring nodes takes place subject to Assumption~\ref{assump:communication_general}. Suppose that for each $k\geq 0$ there exists $\cR^k(\cdot)=[\cR_i^k(\cdot)]_{i\in\cN}$ such that $\cR_i^k(\cdot)$ can be computed with local information available to node $i\in\cN$, and decentralized computation of $\cR^k$ requires $q_k$ communication rounds. Furthermore, we assume that there exist $\Gamma>0$ and $\alpha\in (0,1)$ such that for all $k\geq 0$, $\cR^k$ satisfies
{
\begin{align}
\label{eq:approx_error-for-full-vector-x}
{ \|\cR^k(\bw)-\mathcal{P}_{\cC}(\bw)\| \leq N~\Gamma \beta^{q_k}\norm{\bw}},\quad\forall~\bw\in{\reals^{m_0}}.
\end{align}}%
\end{assumption}
{Now we briefly talk about such operators. Let $V^t\in\reals^{|\cN|\times|\cN|}$ be a matrix encoding the topology of $\cG^t=(\cN,\cE^t)$ in some way for $t\in\integers_+$. We define $W^{t,s}\triangleq V^tV^{t-1}...V^{s+1}$ for any $t,s\in\integers_+$ such that $t\geq s+1$.  For directed time-varying graph ${\cG^t}$, set $V^t\in\reals^{|\cN|\times |\cN|}$ as:}
\begin{align}
\label{eq:directed-weights}
V^t_{ij}={1\over d_j^t}\ \hbox{ if }\ j\in\cN^{\,t,{\rm in}}_i;\quad V^t_{ij}=0\ \hbox{ if }\ j\not\in\cN^{\,t,{\rm in}}_i,\quad i\in\cN.
\end{align}
Let $t_k\in\integers_{+}$ be the total number of \emph{communication rounds} done before the $k$-th iteration of {DPDA-TV}, and let $q_k\in\integers_{+}$ be the number of communication rounds to be performed within the $k$-th iteration while evaluating $\cR^k$. For $\bx=[x_i]_{i\in\cN}\in\cX$ such that $x_i\in\reals^n$ for $i\in\cN$, define
{
\begin{equation}
\label{eq:approx-average-dual-directed}
\cR^k(\bx)\triangleq{\diag(W^{t_k+q_k,t_k}\ones_{|\cN|})^{-1}}(W^{t_k+q_k,t_k}\otimes\id_n)~\bx
\end{equation}
}%
to approximate $\mathcal{P}_{\cC}(\cdot)$. Note that $\mathcal{R}^k(\cdot)$ can be computed in a \emph{distributed fashion} requiring $q_k$ communication rounds -- $\cR^k$ is nothing but the push-sum protocol~\cite{kempe2003gossip}. Assuming that the digraph sequence $\{\cG^t\}_{t\in\integers_+}$ is uniformly strongly connected (M-strongly connected), it follows from \cite{kempe2003gossip,nedic2015distributed} that $\cR^k$ satisfies Assumption~\ref{assump:approximate-average}. When $\{\cG^t\}_{t\in\integers_+}$ is undirected time-varying network, then choosing $\{V^t\}$ according to Metropolis weights, one can show that
\begin{align}
\label{eq:approx-average-dual-undirected}
\cR^k(\bx)\triangleq (W^{t_k+q_k,t_k}\otimes\id_m)\bx
\end{align}
satisfies Assumption~\ref{assump:approximate-average} under certain conditions, {e.g., see~\cite{nedic2009distributed}.}

Note that for $\tcR^k(\cdot)\triangleq\cP_{\cB}(\cR^k(\cdot))$, we have $\tcR^k(\bw)\in\cB$, and $\|\tcR^k(\bw)-\mathcal{P}_{\Ct}(\bw)\|\leq N~\Gamma \beta^{q_k}\norm{\bw}$ for $\bw\in\reals^{{m_0}}$ due to non-expansivity of $\cP_\cB$.
Consider the $k$-th iteration of PDA as shown in \eqref{eq:pock-pd-2}. Instead of setting $\bmu^{k+1}$ to $\blambda^{k+1}$ and $\bxi^{k+1}$ to $\bx^{k+1}$, which require computing $\cP_\cC$, we propose replacing these assignment operations in \eqref{eq:pock-pd-2-mu} and \eqref{eq:pock-pd-2-xi} with an operation that {uses} the inexact averaging operator $\cR^k$ to approximate $\cP_\cC$. This way, 
we obtain \emph{inexact} variant of \eqref{eq:pock-pd-2} replacing \eqref{eq:pock-pd-2-mu} and \eqref{eq:pock-pd-2-xi} with
{\small
\begin{subequations}\label{eq:inexact-rule}
\begin{align}
\bmu^{k+1} &\gets \gamma^k\left(\bom^k-\cP_\cB\big(\cR^k(\bom^k)\big)\right),\quad \hbox{where}\quad \bom^k=\frac{1}{\gamma^k}\bmu^k+{\bxi}^{k}+\eta^k(\bxi^k-{\bxi}^{k-1}), \label{eq:inexact-rule-mu}\\
\bxi^{k+1}&\gets \prox{\tau^k\rho}\Big(\bxi^k-\tau^k\Big(\nabla f(\bxi^k)+A^\top{\btheta^{k+1}+\bmu^{k+1}}+\alpha\big(\bxi^k-\cR^k(\bxi^k)\big)\Big)\Big). \label{eq:inexact-rule-xi}
\end{align}
\end{subequations}}%
Thus, PDA iterations given in~\eqref{eq:pock-pd-2} can be computed inexactly, but in \emph{decentralized} way for a {time-varying} connectivity network $\{\cG^t\}_{t\geq 0}$, via the node-specific computations as in time-varying distributed primal dual algorithm displayed in Fig.~\ref{alg:PDD} below. Indeed, the iterate sequence $\{\bxi^k,\bmu^k,\btheta^k\}_{k\geq 0}$ generated by~DPDA-TV displayed in Fig.~\ref{alg:PDD} is the same sequence generated by the recursion in \eqref{eq:pock-pd-2-theta}, \eqref{eq:inexact-rule-mu}, and \eqref{eq:inexact-rule-xi}. 
The sequences $\{\bx^k\}_{k\geq 0}$ and $\{\blambda^k\}_{k\geq 0}$ will not be explicitly computed, instead we will use it in the analysis of the inexact algorithm.
\begin{figure}[htpb]
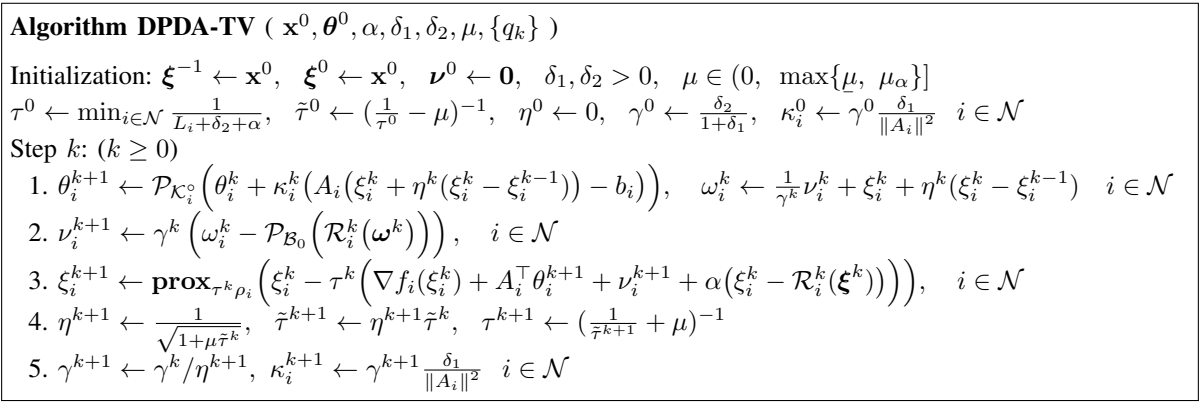

\centering
\framebox{\parbox{0.95\columnwidth}{
{\small
\textbf{Algorithm DPDA-TV} ( $\bx^{0},\btheta^0,\alpha,\delta_1,\delta_2,\mu,\{q_k\}$ ) \\[1.5mm]
Initialization: $\bxi^{-1}\gets\bx^0,\ \ \bxi^{0}\gets\bx^0,\ \ \bmu^0\gets\mathbf{0},\ \ \delta_1,\delta_2>0,\ \ \mu\in(0,~\max\{\ubar{\mu},~\mu_\alpha\}]$\\
$\tau^0\gets\min_{i\in\cN}\frac{1}{L_i+\delta_2+\alpha},\ \ \tilde{\tau}^0\gets(\frac{1}{\tau^0}-\mu)^{-1},\ \ \eta^0\gets 0,\ \ \gamma^0\gets\frac{\delta_2}{1+\delta_1},\ \ \kappa_i^0\gets\gamma^0\frac{\delta_1}{\norm{A_i}^2}\ \ i\in\cN$\\
Step $k$: ($k \geq 0$)\\
\text{ } 1. $\theta_i^{k+1}\gets\cP_{\cK_i^\circ}\Big(\theta_i^k+\kappa_i^k\big(A_i\big(\xi_i^{k}+\eta^k(\xi_i^k-\xi_i^{k-1})\big)-b_i\big)\Big),  \quad \omega^k_i\gets\tfrac{1}{\gamma^k}{\nu}_i^k+\xi_i^{k}+\eta^k(\xi_i^k-\xi_i^{k-1}) \quad i\in\cN$\\
\text{ } 2. ${\nu}_i^{k+1}\gets\gamma^k \left(\omega_i^k - \cP_{\cB_0}\Big(\cR_i^k\big(\bom^k\big)\Big)\right),\quad i \in \cN$ \\
\text{ } 3. $\xi_i^{k+1}\gets\prox{\tau^k\rho_i}\Big(\xi_i^k-\tau^k\Big(\nabla f_i(\xi_i^k)+A_i^\top{\theta^{k+1}_i+\nu_i^{k+1}}+{\alpha}\big(\xi_i^k-\cR_i^k(\bxi^k)\big)\Big)\Big),\quad i \in \cN$\\
\text{ } 4. $\eta^{k+1}\gets\frac{1}{\sqrt{1+\mu {\ttau}^k}},\ \ {\ttau}^{k+1}\gets \eta^{k+1} {\ttau}^k,\ \ \tau^{k+1}\gets(\frac{1}{\ttau^{k+1}}+\mu)^{-1}$\\
\text{ } 5. $\gamma^{k+1}\gets\gamma^k/\eta^{k+1},\ \kappa_i^{k+1}\gets\gamma^{k+1}\frac{\delta_1}{\norm{A_i}^2}\ \ i\in\cN$
}}}\\
\caption{\small time-varying Distributed Primal Dual Algorithm (DPDA-TV)}
\label{alg:PDD}
\vspace*{-5mm}
\end{figure}

{Recall Remark~\ref{rem:rs-convexity}, it is possible that $\ubar{\mu}=0$. In the next lemma, similar to Lemma~\ref{lem:restricted-convex-static},} we generalize the result in Proposition 3.6. of \cite{shi2015extra}, making it suitable for time-varying topology, and show that by suitably regularizing $f$, one can obtain a strongly convex function when $\ubar{\mu}=0$.
\begin{lemma}\label{lem:restricted-convex-general}
Consider $f(\bx)=\sum_{i\in\cN}f_i(x_i)$ under Assumption~\ref{assump:sconvex} and suppose $\ubar{\mu}=0$. Given $\alpha>0$, let $f_\alpha(\bx)\triangleq f(\bx)+\alpha~r(\bx)$, where 
$r(\bx)\triangleq \tfrac{1}{2}d_{\cC}^2(\bx)$.
Then $f_\alpha$ is strongly convex
with modulus ${\mu}_\alpha \triangleq \frac{\bar{\mu}/|\cN|~+\alpha}{2}-\sqrt{\left(\frac{\bar{\mu}/|\cN|~-\alpha}{2}\right)^2+4\bar{L}^2}>0$ for any $\alpha>\frac{4}{ \bar{\mu}}\sum_{i\in\cN}L_i^2$, where $\bar{L}=\sqrt{\frac{\sum_{i\in\cN}L_i^2}{|\cN|}}$.
\end{lemma}
Next, 
we quantify the suboptimality and infeasibility of the DPDA-TV iterate sequence. {Recall that if $\ubar{\mu}>0$, then we set $\alpha=0$ and set $g=f$; otherwise, when $\ubar{\mu}=0$, it follows from Lemma~\ref{lem:restricted-convex-general} that for any $\alpha>\frac{4}{\bar{\mu}}\sum_{i\in\cN}L_i^2$, $f_\alpha$ is strongly convex with modulus $\mu_\alpha>0$; hence, we set $g=f_\alpha$ -- See also Remark~\ref{rem:rs-convexity}.}
\begin{theorem}\label{thm:dynamic-error-bounds}
{Suppose Assumptions \ref{assump:saddle-point}, \ref{assump:sconvex}, \ref{assump:communication_general} and \ref{assump:approximate-average} hold.}
Starting from $\bmu^0=\mathbf{0}$, $\btheta^0=\mathbf{0}$, and an arbitrary $\bx^0$, let $\{\bxi^k,\btheta^k,\bmu^k\}_{k\geq 0}$ be the iterate sequence generated by Algorithm DPDA-TV, displayed in Fig.~\ref{alg:PDD}, using $q_k\geq (5+c)\log_{1/\beta}(k+1)$ communication rounds for the $k$-th iteration for $k\geq 0$. Then $\{\bxi^k\}_{k\geq 0}$ converges to $\bx^*=\ones\otimes x^*$ such that $x^*$ is the optimal solution to \eqref{eq:central_problem}. 

Moreover, the following bounds hold for all $K\geq 1$:
{\small
\begin{subequations}\label{eq:rate_result-d}
\begin{align}
&\max\left\{|\Phi(\bar{\bxi}^K)-\varphi({\bf x}^*)|,~d_{\cC}(\bar{\bxi}^K)+\sum_{i\in\cN}\norm{\theta_i^*} d_{\mathcal{K}_i}(A_i\bar{\bxi}_i^K-b_i)\right\}\leq \frac{\Theta(K)}{N_K}=\cO\left(\frac{1}{K^2}\right),\\
&\norm{\bxi^K-\bx^*}^2\leq \frac{\tilde{\tau}^K}{\gamma^K}2\gamma^0~\Theta(K)=\cO\left(\frac{1}{K^2}\right),
\end{align}
\end{subequations}
}%
{and the parameters satisfy $N_K=\cO(K^2)$ and $\tilde{\tau}^K/\gamma^K=\cO(1/K^2)$, where $N_K=\sum_{k=1}^{K}{\gamma^{k-1}}$, $\bar{\bx}^{K}=N_K^{-1}\sum_{k=1}^{K}{\gamma^{k-1}}\bx^k$, and $\Theta(K)=\cO\big(\sum_{k=1}^K\beta^{q_{k-1}} k^4\big)$; hence, $\sup_{K\in\integers_+}\Theta(K)<\infty$.}
\end{theorem}
\begin{remark}
Note that, {at the $K$-th iteration}, the suboptimality, infeasibility and consensus violation are $\cO\left(\tfrac{1}{N_K}~\Theta(K)\right)$ in the ergodic sense, and the distance of iterates to $\bx^*$ is $\cO\left(\tfrac{\tilde{\tau}^K}{\gamma^K}~\Theta(K)\right)$ where $\Theta(K)$ denotes the error accumulations due to average approximation.
Moreover, $\Theta(K)$ can be bounded above for all $K\geq 1$ as $\Theta_2(K)\leq C_1 \sum_{k=1}^K\beta^{q_{k-1}} k^4$; therefore, for any $c>0$, choosing $\{q_k\}_{k\in\integers_+}$ as stated in Theorem~\ref{thm:dynamic-error-bounds} ensures that $\sum_{k=1}^{\infty}\beta^{q_{k-1}}k^4<1+\tfrac{1}{c}$.
Moreover, for any $c>0$, setting $q_k=(5+c)\log_{\tfrac{1}{\beta}}(k+1)$ for $k\geq 0$ implies that the total number of communication rounds right before the $K$-th iteration is equal to $t_K=\sum_{k=0}^{K-1}q_k\leq(5+c)K\log_{\tfrac{1}{\beta}}(K)$.
\vspace*{-5mm}
\end{remark}
\section{Numerical Section}
\label{sec:numerics}
In this section, we illustrate the performance of DPDA and DPDA-TV for solving synthetic C-LASSO problems. We first test the effect of network topology on the performance of proposed algorithms, and then we compare DPDA and DPDA-TV with other distributed primal-dual algorithms, DPDA-S and DPDA-D, proposed in~\cite{aybat2016primal} for solving \eqref{eq:central_problem} -- it is shown in~\cite{aybat2016primal} that both DPDA-S and DPDA-D converge
with $\cO(1/K)$ ergodic rate when $\bar{\varphi}$ is merely convex. In fact, when $\bar{\varphi}$ is strongly convex with modulus $\mu>0$, using the fact that ${\varphi}(\bx^*)-{\varphi}(\bar{\bx}^K)\geq \frac{\mu}{2}\norm{\bar{\bx}^K-\bx^*}^2$, it immediately follows that $\norm{\bar{\bx}^K-\bx^*}^2\leq \cO(1/K)$.

We consider an isotonic C-LASSO problem over network $\cG^t=(\cN,\cE^t)$ for $t\geq 0$. This problem can be formulated in a centralized form as
$x^*\triangleq\argmin_{x\in\reals^{n}} ~\left\{ \frac{1}{2}\norm{Cx-d}^2+\lambda\norm{x}_1:\ Ax\leq {\bf 0}\right\}$,
where the matrix $C=[C_i]_{i\in\cN}\in\reals^{m|\cN|\times n}$, $d=[d_i]_{i\in \cN}\in\reals^{m|\cN|}$, and $A\in\reals^{n-1\times n}$. In fact, the matrix $A$ captures the isotonic feature of vector $x^*$, and can be written explicitly as, $A(\ell,\ell)=1$ and $A(\ell,\ell+1)=-1$, for $1\leq \ell \leq n-1$, otherwise it is zero.
Each agent $i$ has access to $C_i$, $d_i$, and $A$; hence, by making local copies of $x$, the decentralized formulation 
can be expressed as
{\small
\begin{align}\label{prob:lasso-dist}
\min_{\bx=[x_i]_{i\in\cN}\in\cC} ~\left\{ \frac{1}{2}\sum_{i\in\cN}\norm{C_ix_i-d_i}^2+\frac{\lambda}{|\cN|}\sum_{i\in\cN}\norm{x_i}_1~: \quad Ax_i\leq {\bf 0},\quad i\in\cN \right\},
\end{align}}%
where $\cC$ is the consensus set - see \eqref{eq:consensus_set}.

In the rest, we set $n=20$, $m=n+2$, $\lambda=0.05$ and $\cK_i=-\reals^{n-1}_{+}$ for $i\in\cN$. Moreover, for each $i\in\cN$, we generate $C_i\in\reals^{m\times n}$ as follows: after $mn$ entries i.i.d. with Gaussian distribution are sampled, the condition number of $C_i$ is normalized by sampling the singular values from $[1,3]$ uniformly at random. We generate the first 5 and the last 5 components of $x^*$ by sampling from $[-10,0]$ and $[0,10]$ uniformly at random in ascending order, respectively, and the other middle 10 components are set to zero; hence, $[x^*]_j\leq [x^*]_{j+1}$ for $j=1,\ldots,n-1$. Finally, we set $d_i=C_i({x}^*+\epsilon_i)$,  where $\epsilon_i\in\reals^n$ is a random vector with i.i.d. components following Gaussian distribution with zero mean and standard deviation of $10^{-3}$.\\

{\bf Generating static undirected network:} $\cG=(\cN,\cE)$ is generated as a random small-world network. Given $|\cN|$ and the desired number of edges $|\cE|$, we choose $|\cN|$ edges creating a random cycle over nodes, and then the remaining $|\cE|-|\cN|$ edges are selected uniformly at random.\\

{\bf Generating time-varying undirected network:} Given $|\cN|$ and the desired number of edges $|\cE_0|$ for the initial graph, we generate a random small-world $\cG_0=(\cN,\cE_0)$ as described above. Given $M\in\integers_+$, and $p\in(0,1)$, for each $k\in\integers_+$, we generate $\cG^t=(\mathcal{N},\mathcal{E}^t)$, the communication network at time $t\in\{(k-1)M,\ldots,kM-2\}$ by sampling $\lceil p |\cE_0|\rceil$ edges of $\cG_0$ uniformly at random and we set $\mathcal{E}^{kM-1}=\cE_0\setminus \bigcup_{t=(k-1)M}^{kM-2}\cE^t$. In all experiments, we set $M=5$, $p=0.8$ and the number of communications per iteration is set to $q_k=10\ln(k+1)$.
\subsection{Effect of Network Topology}
In this section, we test the performance of DPDA and DPDA-TV on \emph{undirected} communication networks. To illustrate the effect of network topology, we consider four scenarios in which the number of nodes $|\cN|\in\{10,~40\}$ and the average number of edges per node $(|\cE|/|\cN|)$ is either $\approx 1.5$ or $\approx 4.5$. For each scenario, we plot 
both the 
relative error, i.e., $\max_{i\in\cN}\norm{x_i^k-x^*}/ \norm{x^*}$ and the infeasibility, i.e., $\max_{i\in\cN}d_{\cK_i}(A\bar{x}_i^k)=\max_{i\in\cN}\norm{(A\bar{x}_i^k)_+}$ versus iteration number $k$. {All the plots show the average statistics over all 25 randomly generated replications.}
\newpage
{\bf Testing DPDA on static undirected communication networks:} 
We generated the static small-world networks $\cG=(\cN,\cE)$ as described above for $(|\cN|,|\cE|)\in\{(10,15),~(10,45),$ $(40,60),~(40,180)\}$ and solve the saddle-point formulation \eqref{eq:static-saddle} corresponding to \eqref{prob:lasso-dist} using DPDA. 
For DPDA, displayed in Fig.~\ref{alg:PDS}, we chose $\delta_1=\max_{i\in\cN}d_i=d_{\max}$ and $\delta_2=2\max_{i\in\cN}L_i=2L_{\max}$, which lead to the initial step-sizes as $\gamma^0=\tfrac{2}{3}\frac{L_{\max}}{d_{\max}}$, $\tau^0=\frac{1}{3L_{\max}}$, and $\kappa^0=\tfrac{2}{3}\frac{L_{\max}}{\norm{A}^2}$.

In Fig.~\ref{static-net}, we plot $\max_{i\in\cN}\norm{x_i^k-x^*}/ \norm{x^*}$ and $\max_{i\in\cN}\norm{(A\bar{x}_i^k)_+}$ statistics for DPDA versus iteration number $k$. Note that compared to average edge density, the network size has more influence on the convergence rate, i.e., the smaller the network faster the convergence is. On the other hand, for fixed size network, as expected, higher the density faster the convergence is.
\begin{figure}[h]
\centering
\includegraphics[scale=0.5]{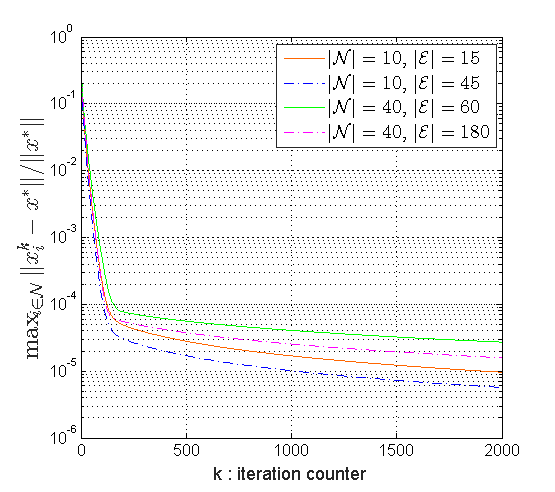}
\includegraphics[scale=0.5]{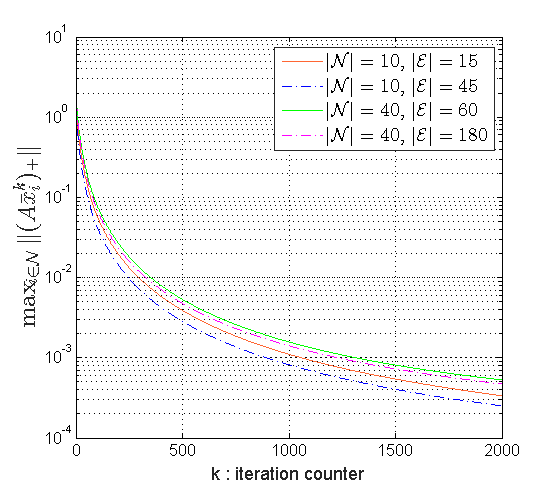}
\caption{Effect of network topology on the convergence rate of DPDA}
\label{static-net}
\end{figure}

{\bf Testing DPDA-TV on time-varying undirected communication networks:} We first generated an undirected graph $\cG_u=(\cN,\cE_u)$ as in the static case, and let $\cG_0=\cG_u$. Next, we generated $\{\cG^t\}_{t\geq 1}$ as described above by setting $M=5$ and $p=0.8$. For each consensus round $t\geq 1$, $V^t$ is formed according to Metropolis weights, i.e., for each $i\in\cN$, $V^t_{ij}=1/(\max\{d_i,d_j\}+1)$ if $j\in\cN_i^t$, $V^t_{ii}=1-\sum_{i\in\cN_i}V^t_{ij}$, and $V^t_{ij}=0$ otherwise -- see~\eqref{eq:approx-average-dual-undirected} for our choice of $\cR^k$.

For DPDA-TV, displayed in Fig.~\ref{alg:PDD}, we chose $\delta_1=\delta_2=1$, which lead to the initial step-sizes as {$\gamma^0=\frac{1}{2}$, $\tau^0=\frac{1}{L_{\max}+1}$, and $\kappa^0=\frac{1}{2\norm{A}^2}$.}
In Fig.~\ref{dynamic-net}, we plot $\max_{i\in\cN}\norm{\xi_i^k-x^*}/ \norm{x^*}$ and $\max_{i\in\cN}\norm{(A\bar{\xi}_i^k)_+}$ statistics for DPDA-TV versus iteration number $k$ --  we used $\{\bxi^k\}$ to compute the error statistics instead of $\{\bx^k\}$ as $\bx^k$ is never actually computed for DPDA-TV. Note that network size and average edge density have the same impact on the rate as in the static case.
\begin{figure}[h]
\centering
\includegraphics[scale=0.5]{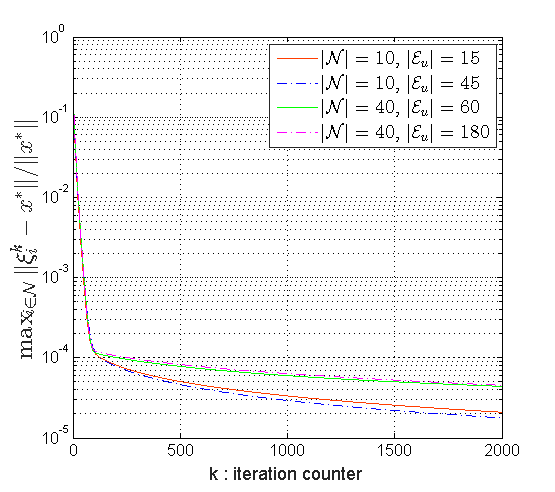}
\includegraphics[scale=0.5]{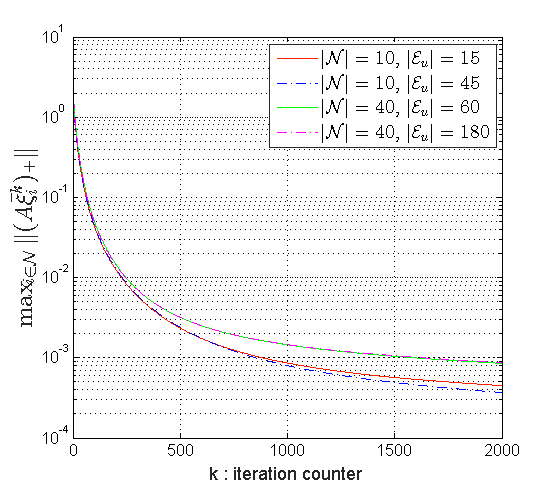}
\caption{Effect of network topology on the convergence rate of DPDA-TV}
\label{dynamic-net}
\end{figure}
\subsection{Comparison with other methods}
We also compared our methods with DPDA-S and DPDA-D, in terms of the relative error and infeasibility of the ergodic iterate sequence, i.e., $ \max_{i\in\cN}\norm{\bar{x}_i^k-x^*}/ \norm{x^*}$ and $\max_{i\in\cN}\norm{(A\bar{x}_i^k)_+}$. We further report the performance of our algorithms in terms of relative error of the actual iterate sequence, i.e., $\max_{i\in\cN}\norm{{x}_i^k-x^*}/ \norm{x^*}$. For DPDA-D and DPDA-TV, we used $\{\bxi^k\}$ sequence to compute the error statistics instead of $\{\bx^k\}$ as $\bx^k$ is never actually computed. 
In this section we fix the number of nodes to $|\cN|=10$ and the average edge density to $|\cE|/|\cN|=4.5$ -- we observed the same convergence behavior for the other network scenarios discussed in the previous section.\\%
\indent{\bf Static undirected network:} We generated $\cG=(\cN,\cE)$ and chose the algorithm parameters as in the previous section. Moreover, the step-sizes of DPDA-S are set to the initial steps-sizes of DPDA. As it can be seen in Fig. \ref{static-compare}, DPDA has faster convergence when compared to DPDA-S.
\begin{figure}[h]
\centering
\includegraphics[scale=0.5]{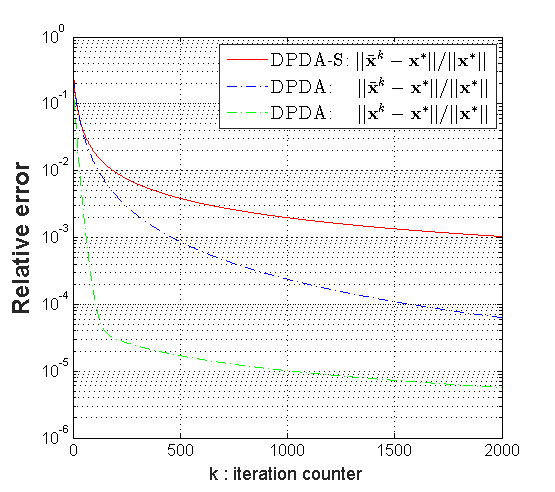}
\includegraphics[scale=0.5]{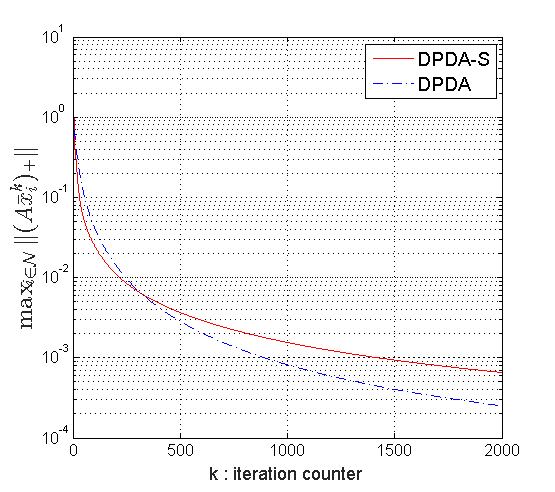}
\caption{Comparison of DPDA and DPDA-S over undirected static network}
\label{static-compare}
\end{figure}

{\bf Time-varying undirected network:} We generated the network sequence $\{\cG^t\}_{t\geq 0}$ and chose the parameters as in the prvious section. Moreover, the step-sizes of DPDA-D are set to the initial steps-sizes of DPDA-TV. As it can be seen in Fig. \ref{undirected-compare}, DPDA-TV has faster convergence when compared to DPDA-D.
\begin{figure}[h]
\centering
\includegraphics[scale=0.5]{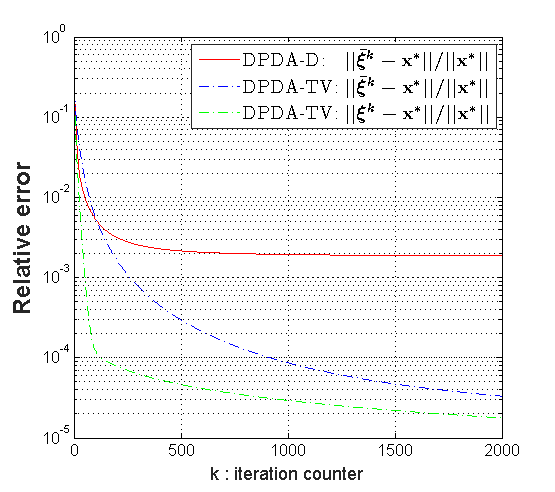}
\includegraphics[scale=0.5]{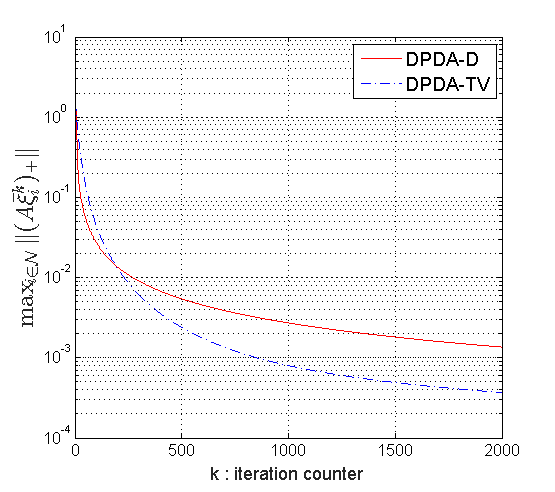}
\caption{Comparison of DPDA-TV and DPDA-D over undirected time-varying network}
\label{undirected-compare}
\end{figure}
\newpage
{\bf Time-varying directed network:} In this scenario, we generated time-varying communication networks similar to \cite{nedich2016achieving}. Let $\cG_d=(\cN,\cE_d)$ be the directed graph shown in Fig.~\ref{fig:Gd} where it has $|\cN|=12$ nodes and $|\cE_d|=12$ directed edges. {We set $\cG_0=\cG_d$, and we generate $\{\cG^t\}_{t\geq 0}$ generated as in the undirected case with parameters $M=5$ and $p=0.8$; hence, $\{\cG^t\}_{t\geq 0}$ is 
$M$-strongly-connected. Moreover, communication weight matrices $V^t$ are formed according to rule \eqref{eq:directed-weights}. We chose the initial step-sizes for DPDA-TV as in the time-varying undirected case, and the constant step-sizes of DPDA-D is set to the initial steps-sizes of DPDA-TV.
In Fig.~\ref{directed-compare} we compare DPDA-TV against DPDA-D. We observe that over time-varying directed networks DPDA-TV again outperforms DPDA-D for both statistics.}


\vspace*{-5mm}
\begin{figure}[h]
\centering
\includegraphics[scale=0.5]{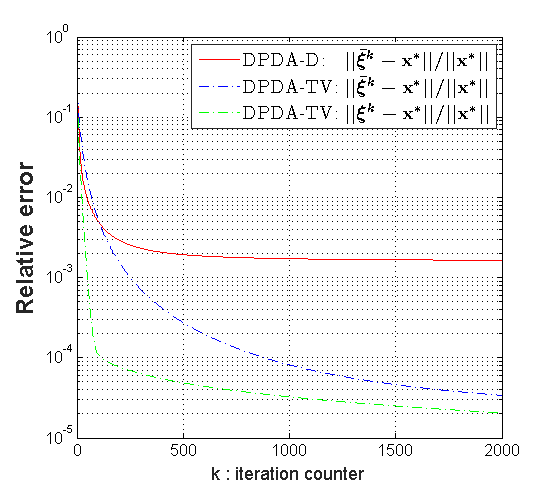}
\includegraphics[scale=0.5]{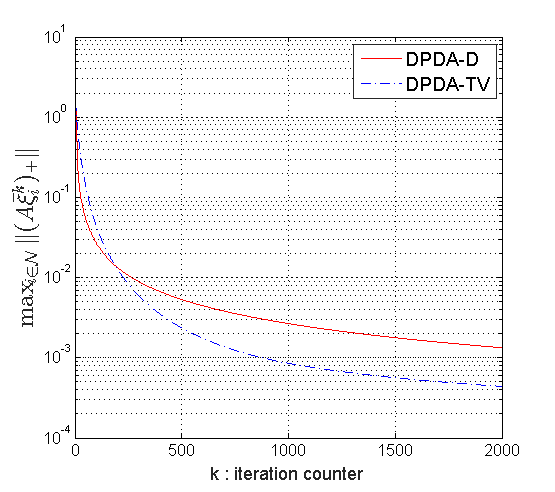}
\caption{Comparison of DPDA-TV and DPDA-D over directed time-varying network.}
\label{directed-compare}
\end{figure}

\begin{figure}
\centering
\begin{minipage}[c]{\textwidth}
    \begin{center}
      \begin{tikzpicture}[scale=1.25]
        \coordinate (x1) at (-1.5,0);
        \coordinate (x2) at (5.1,1.2);
        \coordinate (x3) at (4.75,-0.75);
        \coordinate (x4) at (1.9,2);
        \coordinate (x5) at (3.2,-1.5);
        \coordinate (x6) at (1.5,-1.5);
        \coordinate (x7) at (3.1,1.5);
        \coordinate (x8) at (-1,-1.5);
        \coordinate (x9) at (1.25,1.25);
        \coordinate (x10) at (0.5,0.5);
        \coordinate (x11) at (0.4,1.5);
        \coordinate (x12) at (4,0.2);


        \draw [arrows={- triangle 45}] (x1) -- (x10);
        \node[align=left, above] at (x10) {\small $10$};

        \draw [arrows={- triangle 45}] (x1) -- (x6);
        \node[align=left, below] at (x6) {\small $6$};

        \draw [arrows={- triangle 45}] (x8) -- (x1);
        \node[align=left, left] at (x1) {\small $1$};

        \draw [arrows={- triangle 45}] (x8) -- (x10);

        \draw [arrows={triangle 45 - triangle 45}] (x8) -- (x6);
        \node[align=left, left] at (x8) {\small $8$};

        \draw [arrows={- triangle 45}] (x6) -- (x3);
        \node[align=left, right] at (x3) {\small $3$};

        \draw [arrows={- triangle 45}] (x11) -- (x1);

        \draw [arrows={- triangle 45}] (x9) -- (x11);
        \node[align=left, left] at (x11) {\small $11$};

        \draw [arrows={- triangle 45}] (x9) -- (x3);

         \draw [arrows={- triangle 45}] (x9) -- (x5);
        \node[align=left, left] at (x5) {\small $5$};

        \draw [arrows={- triangle 45}] (x4) -- (x9);
        \node[align=left, right] at (x9) {\small $9$};

        \draw [arrows={- triangle 45}] (x4) -- (x11);

        \draw [arrows={- triangle 45}] (x7) -- (x4);
        \node[align=left, above] at (x4) {\small $4$};

        \draw [arrows={- triangle 45}] (x7) -- (x12);
        \node[align=left, right] at (x12) {\small $12$};

        \draw [arrows={- triangle 45}] (x7) -- (x6);

        \draw [arrows={- triangle 45}] (x2) -- (x10);

        \draw [arrows={- triangle 45}] (x12) -- (x6);

        \draw [arrows={- triangle 45}] (x12) -- (x2);
        \node[align=left, above] at (x2) {\small $2$};

        \draw [arrows={- triangle 45}] (x12) -- (x5);

        \draw [arrows={- triangle 45}] (x3) -- (x12);

        \draw [arrows={- triangle 45}] (x5) -- (x3);

        \draw [arrows={- triangle 45}] (x10) -- (x7);
        \node[align=left, above] at (x7) {\small $7$};

        \draw [arrows={- triangle 45}] (x10) -- (x5);

        \filldraw[fill=red!50!](x1) circle [radius=0.08];
        \filldraw[fill=red!50!] (x2) circle [radius=0.08];
        \filldraw[fill=red!50!] (x3) circle [radius=0.08];
        \filldraw[fill=red!50!] (x4) circle [radius=0.08];
        \filldraw[fill=red!50!] (x5) circle [radius=0.08];
        \filldraw[fill=red!50!](x6) circle [radius=0.08];
        \filldraw[fill=red!50!] (x7) circle [radius=0.08];
        \filldraw[fill=red!50!] (x8) circle [radius=0.08];
        \filldraw[fill=red!50!] (x9) circle [radius=0.08];
        \filldraw[fill=red!50!] (x10) circle [radius=0.08];
        \filldraw[fill=red!50!] (x11) circle [radius=0.08];
        \filldraw[fill=red!50!] (x12) circle [radius=0.08];
      \end{tikzpicture}
    \end{center}
\end{minipage}
 \caption{$\cG_d=(\cN,\cE_d)$ directed strongly connected graph} \label{fig:Gd}
\end{figure}

\singlespacing
\bibliographystyle{unsrt}
\bibliography{papers}
\newpage
\section{Appendix}
\subsection{Proof of Lemma~\ref{lem:restricted-convex-static}}
Let $\bx^*=\ones_{|\cN|}\otimes x^*$, where $x^*$ is the unique optimal solution to \eqref{eq:central_problem}, and according to Assumption~\ref{assump:sconvex}, $\bar{f}$ is {strongly convex}
with modulus $\bar{\mu}>0$. Note that any $W$ as given in Definition~\ref{def:W} is positive semidefinite, and ${\ns(W)}=\spn\{\ones\}$. In the rest, we will use these properties of $W$. Fix some arbitrary $\alpha>\frac{4}{{\lambda_2} \bar{\mu}}\sum_{i\in\cN}L_i^2$ and $\bx\in\reals^{n|\cN|}$.

$\bx\in\reals^{n|\cN|}$ can be decomposed into $\bu\in \spn\{\ones\}$ and $\bv\in \spn\{\ones\}^\perp$ where $\bx=\bu+\bv$ and $\norm{\bx}^2=\norm{\bu}^2+\norm{\bv}^2$. From definition of $f_\alpha$ we have that,
{\small
\begin{align}\label{grad-g-s}
\fprod{\nabla f_{\alpha}(\bx)-\nabla f_{\alpha}(\bx^*),~\bx-\bx^*}=\fprod{\nabla f(\bx)-\nabla f(\bx^*),~\bx-\bx^*}+\alpha~{\norm{\bx-\bx^*}^2_{W\otimes\id_n}.}
\end{align}
}%
Let $N\triangleq |\cN|$ and $\bar{L}\triangleq \sqrt{\frac{\sum_{i\in\cN}L_i^2}{N}}$. The inner product of $\fprod{\nabla f(\bx)-\nabla f(\bx^*),~\bx-\bx^*}$ can be bounded by using the following inequalities: 
{\small
\begin{subequations}\label{eq:fprod_bound}
\begin{align}
&\fprod{\nabla f(\bu)-\nabla f(\bx^*),~\bu-\bx^*}\geq\frac{\bar{\mu}}{N} \norm{\bu-\bx^*}^2, \\
&\fprod{\nabla f(\bu)-\nabla f(\bx^*),~\bx-\bu} \geq  -\sum_{i\in\cN} L_i\norm{u-x^*}\norm{x_i-u}\geq -\bar{L}\norm{\bu-\bx^*}\norm{\bv},\\
&\fprod{\nabla f(\bx)-\nabla f(\bu),~\bx-\bu}\geq 0, \\
&\fprod{\nabla f(\bx)-\nabla f(\bu),~\bu-\bx^*} \geq -\sum_{i\in\cN} L_i\norm{x_i-u}\norm{u-x^*}\geq \bar{L}\norm{\bu-\bx^*}\norm{\bv},
\end{align}
\end{subequations}
}%
which follow from convexity, Lipschitz differentiability, and strong convexity of $f$. Summing above inequalities leads to,
{\small
\begin{align}\label{grad-f-bound-s}
{\fprod{\nabla f(\bx)-\nabla f(\bx^*),~\bx-\bx^*} \geq \frac{\bar{\mu}}{N}\norm{\bu-\bx^*}^2-2\bar{L}\norm{\bu-\bx^*}\norm{\bv}.} 
\end{align}
}%
Hence, strong convexity of $f_\alpha$ follows from \eqref{grad-g-s}, \eqref{grad-f-bound-s}. Indeed, it follows from {$W\in\sy^N_+$} and ${\ns(W)}=\spn\{\ones\}$ that we have $\norm{\bx-\bx^*}^2_{W\otimes\id_n}=\bv^\top \left(W\otimes \id_n \right) \bv\geq {\lambda}_2 \norm{\bv}^2$, where ${\lambda}_2=\lambda^+_{\min}(W)$ is the second smallest eigenvalue of $W$. Therefore,
{\small
\begin{align}\label{g-strong-convex-s}
\fprod{\nabla f_{\alpha}(\bx)-\nabla f_{\alpha}(\bx^*),~\bx-\bx^*}\geq \frac{\bar{\mu}}{N}\norm{\bu-\bx^*}^2-2\bar{L}\norm{\bu-\bx^*}\norm{\bv}+ \alpha{\lambda}_2 \norm{\bv}^2.
\end{align}
}%

Next, fix some arbitrary $\omega\geq 0$. Then either \textbf{(i)} $\norm{\bv}\leq \omega \norm{\bu-\bx^*}$, or \textbf{(ii)} $\norm{\bv}\geq \omega \norm{\bu-\bx^*}$ holds. If \textbf{(i)} is true, then \eqref{g-strong-convex-s} implies
{\small
\begin{align}\label{g-strong-convex-s-i}
\fprod{\nabla f_{\alpha}(\bx)-\nabla f_{\alpha}(\bx^*),~\bx-\bx^*}&\geq \left(\frac{\bar{\mu}}{N}-2\omega\bar{L}\right)\norm{\bu- \bx^*}^2+\alpha {\lambda}_2 \norm{\bv}^2\nonumber  \\
& \geq \min\left\{\frac{\bar{\mu}}{N}-2\omega\bar{L}, ~\alpha {\lambda}_2\right\} \norm{\bx-\bx^*}^2;
\end{align}
}%
on the other hand, if \textbf{(ii)} is true, then \eqref{g-strong-convex-s} implies
{\small
\begin{align}\label{g-strong-convex-s-ii}
\fprod{\nabla f_{\alpha}(\bx)-\nabla f_{\alpha}(\bx^*),~\bx-\bx^*}&\geq \frac{\bar{\mu}}{N}\norm{\bu- \bx^*}^2+\left(\alpha {\lambda}_2-\frac{2\bar{L}}{\omega}\right)\norm{\bv}^2 \nonumber \\
& \geq \min\left\{\frac{\bar{\mu}}{N}, ~\alpha{\lambda}_2-\frac{2\bar{L}}{\omega}\right\} \norm{\bx-\bx^*}^2.
\end{align}
}%
Combining \eqref{g-strong-convex-s-i} and \eqref{g-strong-convex-s-ii} we conclude that,
{\small
\begin{align}\label{g-strong-convex-s-result}
\fprod{\nabla f_{\alpha}(\bx)-\nabla f_{\alpha}(\bx^*),~\bx-\bx^*}\geq \min\left\{\frac{\bar{\mu}}{N}-2\omega\bar{L}, ~\alpha {\lambda}_2-\frac{2\bar{L}}{\omega}\right\} \norm{\bx-\bx^*}^2,
\end{align}
}%
Since $\omega\geq 0$ is arbitrary, $f_{\alpha}$ is strongly convex
with modulus $\mu_\alpha=\max_{\omega\geq 0}\min $ $\left\{\frac{\bar{\mu}}{N}-2\bar{L}\omega, ~\alpha {\lambda}_2-\frac{2\bar{L}}{\omega}\right\}$.
Note $\mu_\alpha$ is attained for $\omega_\alpha\geq 0$ such that $\frac{{\bar{\mu}}}{N}-2\bar{L}\omega_\alpha=\alpha {\lambda}_2-\frac{2\bar{L}}{\omega_\alpha}$, which implies that $\omega_\alpha=\frac{1}{2}\left(\frac{\bar{\mu}/N~-\alpha{\lambda}_2}{2\bar{L}} + \sqrt{{\left(\frac{\bar{\mu}/N~-\alpha{\lambda}_2}{2\bar{L}}\right)^2} +4} \right)$. 
Moreover, $\mu_\alpha=\frac{\bar{\mu}}{N}-2\bar{L}\omega_\alpha$ is the value given in the statement of the lemma, and we have $\frac{\bar{\mu}}{N}>\mu_\alpha>0$ for any $\alpha>\frac{4}{\lambda_2 \bar{\mu}}\sum_{i\in\cN}L_i^2$. It is worth mentioning that $\mu_\alpha$ is a concave increasing function of $\alpha$ over $\reals_{++}$, and $\sup_{\alpha>0} \mu_\alpha = \lim_{\alpha\nearrow\infty}\mu_\alpha=\frac{\bar{\mu}}{N}$.
\subsection{Key Lemmas for the Proof of Theorem~\ref{thm:static-error-bounds}}
\begin{defn}\label{def:parameters}
Let $T=[A^\top~M^\top]^\top$ for $A\triangleq\diag([A_i]_{i\in\mathcal{N}})\in\reals^{m \times n|\cN|}$. Given $\alpha,\mu,\delta_1>0$, and arbitrary sequences $\{\tau^k\},\{\gamma^k\}\subset\reals_{++}$, $\{\kappa_i^k\}_{k\geq 0}\subset\reals_{++}$ for $i\in\mathcal{N}$. For $k\geq 0$, define {$\mathbf{D}_{\tau^k}\triangleq\frac{1}{\gamma^k}\id_{n|\cN|}$}, {$\mathbf{\widetilde{D}}_{\tau^{k}}\triangleq \left({1\over \tau^{k}}- \mu\right)\id_{n|\cN|}$}, $\mathbf{\bar{D}}_{\tau^k}\triangleq\diag([({1\over \tau^k} -{(L_i+2\alpha d_i)})\id_{n}]_{i\in \mathcal{N}})$, and
$\mathbf{\bar{Q}}_k\triangleq\begin{bmatrix}
\cA^k &-\eta^k T^\top\\
-\eta^k T & \mathbf{D}_{\kappa^k,\gamma^k}
\end{bmatrix}$, where $\cA^k\triangleq(\eta^k)^2\gamma^{k}\diag([(2d_i+\delta_1)\id_n]_{i\in\cN})\succ 0$, and $\mathbf{D}_{\kappa^k,\gamma^k}$ is defined in Definition~\ref{def:bregman-s}.
\end{defn}
In order to prove Theorem~\ref{thm:static-error-bounds}, we first prove Lemma~\ref{lem:pda-lemma} below which help us to appropriately bound $\mathcal{L}(\bar{\bx}^K,\by)-\mathcal{L}(\bx^*,\bar{\by}^K)$ {for any $\by\in\cY$ and $\norm{\bx^K-\bx^*}$. In order to prove Lemma~\ref{lem:pda-lemma}, we first need to show the following two lemmas, Lemma~\ref{lem:schur-s} and Lemma~\ref{lem:stepsize-s}, describing a proper choice for the step size sequences.}
\begin{lemma}\label{lem:schur-s}
Given $\delta_1>0$. For any $k\geq 0$, $\mathbf{\bar{Q}}_k\succeq \mathbf{0}$ if $\eta^k>0$, and positive numbers $\{\kappa_i^k\}_{i\in\mathcal{N}}$ and $\gamma^k$ are chosen such that
{\small
\begin{align}\label{eq:schur-cond-static}
\frac{\kappa_i^k \norm{A_i}^2}{\gamma^k} \leq \delta_1,\quad \forall~i\in\mathcal{N}.
\end{align}
}%
\end{lemma}
\begin{proof}
Let $\mathbf{D}_{\kappa^k,\gamma^k}$ be as in Definition~\ref{def:bregman-s}. Since $\mathbf{D}_{\gamma^k}\succ 0$, Schur complement condition implies that $\bar{\mathbf{Q}}_k\succeq 0$ if and only if
{\small
\begin{align}\label{eq:schur-cond-2-static}
\begin{bmatrix}\cA^k & -\eta^k A^\top\\ -\eta^k A & \mathbf{D}_{\kappa^k} \end{bmatrix}
-\gamma^k(\eta^k)^2 \begin{bmatrix} M^\top M & 0 \\ 0 & 0\end{bmatrix} \succeq 0.
\end{align}}%
Moreover, since $\mathbf{D}_{\kappa^k}\succ 0$, again using Schur complement and the fact that $ M^\top M=\Omega\otimes\id_n$, one can conclude that \eqref{eq:schur-cond-2-static} holds if and only if $\frac{1}{(\eta^k)^2}\cA^k-\gamma^k\Omega\otimes\id_n- A^\top \mathbf{D}_{\kappa^k}^{-1} A\succeq 0$. Moreover, by definition $\Omega=\diag([d_i]_{i\in\mathcal{N}})-E$, where $E_{ii}=0$ for all $i\in \mathcal{N}$ and $E_{ij}=E_{ji}=1$ if $(i,j)\in\mathcal{E}$ or $(j,i)\in \mathcal{E}$. Note that $\diag([d_i]_{i\in\mathcal{N}})+E \succeq 0$ since it is diagonally dominant. Therefore, $\Omega \preceq 2\diag([d_i]_{i\in\mathcal{N}})$. Hence, one can conclude that \eqref{eq:schur-cond-2-static} holds if $\frac{1}{(\eta^k)^2}\cA^k-2\gamma^k\diag([d_i\id_n]_{i\in\cN})- A^\top \mathbf{D}_{\kappa^k}^{-1} A\succeq 0$. This condition holds if \eqref{eq:schur-cond-static} is true.
\end{proof}
\begin{lemma}\label{lem:stepsize-s}
{Let $\bD_{\kappa^k,\gamma^k}$ be as given in Definition~\ref{def:bregman-s}, and ${\bD}_{\tau^k}$, $\widetilde{\bD}_{\tau^k}$, $\bar{\bD}_{\tau^k}$ and $\mathbf{\bar{Q}}_k$ be as in Definition~\ref{def:parameters} for $\alpha\geq 0$ chosen according to Lemma \ref{lem:restricted-convex-static} and Remark~\ref{rem:g}, and $\mu\in(0,~\max\{\ubar{\mu},~\mu_\alpha\}]$.}
Suppose $\{\tau^k\},~\{\eta^k\},~\{\gamma^k\}\subset\reals_{++}$, $\{\kappa_i^k\}_{k\geq 0}\subset\reals_{++}$ for $i\in\mathcal{N}$ are chosen as in DPDA diplayed in Fig.~\ref{alg:PDS}, then the following relations hold for all $k\geq 0$:
{\small
\begin{subequations}
\label{eq:cond-s}
\begin{align}
&\mathbf{\bar{Q}}_k=\begin{bmatrix}
\cA^k &-\eta^k T^\top\\
-\eta^k T & \mathbf{D}_{\kappa^k,\gamma^k}
\end{bmatrix}\succeq 0,\label{cond-0-s} \\
&{\gamma^k}{\bD}_{\tau^k}\succeq \gamma^{k+1}\widetilde{\bD}_{\tau^{k+1}}, \label{cond-1-s}\\
& \gamma^k\bD_{\kappa^k}\succeq \gamma^{k+1}\bD_{\kappa^{k+1}},\label{cond-2-s}\\
& \gamma^k\bD_{\gamma^k}\succeq \gamma^{k+1}\bD_{\gamma^{k+1}}, \label{cond-3-s}\\
& \gamma^k= \gamma^{k+1}\eta^{k+1},\label{cond-4-s}\\
& \gamma^k\bar{\bD}_{\tau^k}\succeq \gamma^{k+1}\cA^{k+1}\label{cond-5-s}.
\end{align}
\end{subequations}}%
{Moreover, $\eta^k\in(0,1)$, $0<\frac{1}{\tilde{\tau}^k}<\frac{1}{\tau^k}=\cO(k)$, and $0<\gamma^k=\cO(k)$.}
\end{lemma}
\begin{proof}
It is trivial to check that the parameter sequence constructed in Fig.~\ref{alg:PDS} {satisfies \eqref{eq:cond-s}}. Indeed, Lemma~\ref{lem:schur-s} shows that \eqref{cond-0-s} is true since $\kappa_i^k$ for $i\in\cN$ and $\gamma^k$ as chosen in Fig.~\ref{alg:PDS} satisfy \eqref{eq:schur-cond-static} for all $k\geq 0$. This specific choice of parameters satisfy \eqref{cond-1-s}, \eqref{cond-2-s}, \eqref{cond-3-s}, and \eqref{cond-4-s} with equality. Moreover, one can use induction to show \eqref{cond-5-s} using the relations $\tilde{\tau}^k>\tau^k$, $\tilde{\tau}^k>\tilde{\tau}^{k+1}$, and $\gamma^k\tilde{\tau}^k=\gamma^{k+1}\tilde{\tau}^{k+1}$ for all $k\geq 0$.
\end{proof}
\begin{lemma}
\label{lem:pda-lemma}
For any 
$\by\in\cY$, the iterate sequence $\{\bx^k,\by^k\}_{k\geq 1}$ generated using Algorithm DPDA as in Fig.~\ref{alg:PDS}, {where $\by^k=[{\btheta^k}^\top {\blambda^k}^\top]^\top$,} satisfies for all $k\geq 0$,
{\small
\begin{eqnarray}\label{eq:pda-lemma}
\lefteqn{\mathcal{L}({\bf x}^{k+1},\by)-\mathcal{L}({\bf x}^*,\by^{k+1})\leq}\\
& & \bigg[\frac{1}{2}\norm{\bx^*-\bx^{k}}_{\widetilde{\bD}_{\tau^k}}^2+\frac{1}{2}\norm{\by-\by^k}_{\bD_{\kappa^k,\gamma^k}}^2-{\eta^k}\fprod{ T(\bx^k-\bx^{k-1}),~\by-\by^k} +\frac{1}{2}\norm{\bx^k-\bx^{k-1}}^2_{\cA^k}\bigg] \nonumber \\
& & \mbox{} - \bigg[\frac{1}{2}\norm{\bx^*-\bx^{k+1}}_{{\bD}_{\tau^k}}^2+\frac{1}{2}\norm{\by-\by^{k+1}}_{\bD_{\kappa^k,\gamma^k}}^2-\fprod{T(\bx^{k+1}-\bx^{k}),~\by-\by^{k+1}} +\frac{1}{2}\norm{\bx^{k+1}-\bx^k}^2_{\bar{\bD}_{\tau^k}}\bigg]. \nonumber
\end{eqnarray}}%
\end{lemma}
\begin{proof}
Note that {$\bx$}-subproblem in \eqref{eq:PDA-x} is separable in local decisions $\{x_i\}_{i\in\cN}$; and for each $i\in\cN$ the local subproblem over $x_i$ is strongly convex with constant $1/\tau^k$. Indeed, let $\bp^k=T^\top \by^k$ and define $\{p_i^k\}_{i\in\cN}$ such that $p_i^k$ is the subvector corresponding to the components of $x_i$, i.e., $\bp^k=[p_i^k]_{i\in\cN}$. In addition, $\nabla g(\bx^k)=[\nabla g_i({\bx^k})]_{i\in\cN}$ where $\nabla g_i({\bx^k})\triangleq\nabla f_i(x^k_i)+\left[ (\Omega\otimes\id_n)\bx^k\right]_i$, where
{$\left[ (\Omega\otimes\id_n)\bx^k\right]_i=\sum_{j\in\cN_i}(x_i^k-x_j^k)$}. Thus, for all $i\in\cN$
{\small
\begin{equation}
\label{eq:x-subproblem}
x_i^{k+1}=\argmin_{x_i}\rho_i(x_i)
+\fprod{\grad g_i({\bx^k}),~x_i-x_i^k}+\fprod{p_i^{k+1},x_i}+\frac{1}{2\tau^k}\norm{x_i-x_i^k}^2.
\end{equation}}%
Therefore, {for $i\in\cN$}, the strong convexity of the objective in local subproblem \eqref{eq:x-subproblem} implies
{\small
\begin{eqnarray}\label{eq:rho-s}
\lefteqn{\rho_i(x^*)+\fprod{\grad g_i({\bx^k}),~x^*-x_i^k}+\fprod{p_i^{k+1},x^*}+\frac{1}{2\tau^k}\norm{x^*-x_i^k}^2\geq}\nonumber\\
& & \rho_i(x_i^{k+1})+\fprod{\grad g_i({\bx^k}),~x_i^{k+1}-x_i^k}+\fprod{p_i^{k+1},x_i^{k+1}}+\frac{1}{2\tau^k}\norm{x_i^{k+1}-x_i^k}^2+\frac{1}{2\tau^k}\norm{x^*-x_i^{k+1}}^2.
\end{eqnarray}}%
Now, we show that {$\grad g$} is Lipschitz continuous. First, recall that as we discussed in the proof of Lemma~\ref{lem:schur-s}, we have $\Omega \preceq 2\diag([d_i]_{i\in\mathcal{N}})$. Second, since $\norm{\bx}^2_{\Omega\otimes\id_n}$ is a quadratic term, for any $\bar{\bx}$ we have
{\small
\begin{align}\label{eq:quad-lip}
\frac{1}{2}\norm{\bx}^2_{\Omega\otimes\id_n}&=\frac{1}{2}\norm{\bar{\bx}}^2_{\Omega\otimes\id_n}+\fprod{(\Omega\otimes\id_n)\bar{\bx},~\bx-\bar{\bx}}+ \frac{1}{2}(\bx-\bar{\bx})^\top(\Omega\otimes\id_n)(\bx-\bar{\bx})  \nonumber\\
&\leq \frac{1}{2}\norm{\bar{\bx}}^2_{\Omega\otimes\id_n}+{\fprod{(\Omega\otimes\id_n)\bar{\bx},~\bx-\bar{\bx}}}+ (\bx-\bar{\bx})^\top \diag([d_i\id_n]_{i\in\cN})(\bx-\bar{\bx}).
\end{align}
}%
In addition, since each $f_i$ has a Lipschitz continuous gradient, we have for any $\bx$ and $\bar{\bx}$ that
{\small
\begin{align}\label{eq:f-lip-s}
f(\bx)\leq f(\bar{\bx})+\fprod{\nabla f(\bar{\bx}),~{\bx-\bar{\bx}}}+\sum_{i\in\cN}\frac{L_i}{2}\norm{x_i-\bar{x}_i}^2.
\end{align}
}%
Let $\bL_g\triangleq \diag([(L_i+2d_i\alpha){\id_n}]_{i\in\cN})\in\sy^{n|\cN|}$. Summing \eqref{eq:quad-lip} and \eqref{eq:f-lip-s}, for any $\bx$ and $\bar{\bx}$, we have
{\small
\begin{align}\label{eq:g-lip}
g(\bx)&\leq g(\bar{\bx})+\fprod{\nabla g(\bar{\bx}),~\bx-\bar{\bx}}+\sum_{i\in\cN}\frac{L_i+2d_i\alpha}{2}\norm{x_i-\bar{x}_i}^2 \nonumber \\
&=g(\bar{\bx})+\fprod{\nabla g(\bar{\bx}),~\bx-\bar{\bx}}+{\frac{1}{2}}\norm{\bx-{\bar{\bx}}}^2_{\bL_g}.
\end{align}
}%
It follows from strong convexity of $\bar{f}$ that choosing $\alpha\geq 0$ according to Lemma \ref{lem:restricted-convex-static} and Remark~\ref{rem:g}, we conclude that for any $\mu\in(0,~\max\{\ubar{\mu},~\mu_\alpha\})$ we have
{\small
\begin{align}\label{eq:g-strong-s}
g(\bx^*)&\geq g(\bx^{k})+\fprod{\grad g(\bx^{k}),~\bx^*-\bx^{k}} +\frac{{\mu}}{2}\norm{\bx^*-\bx^{k}}^2 \nonumber \\
&\geq g(\bx^{k+1})+\fprod{\grad g(\bx^{k}),~\bx^*-\bx^{k+1}} +\frac{{\mu}}{2}\norm{\bx^*-\bx^{k}}^2-\frac{1}{2}\norm{\bx^{k+1}-\bx^k}^2_{\bL_g}.
\end{align}}%
Since $\sum_{i\in\cN}\fprod{p_i^{k+1},~x^*}=\fprod{T\bx^*,~\by^{k+1}}$, 
{first summing \eqref{eq:rho-s} over $i\in\cN$, next summing the resulting inequality with \eqref{eq:g-strong-s}, and then adding $g(\bx^k)$ to both hand-sides}, we get
{\small
\begin{eqnarray}
\lefteqn{\Phi(\bx^*)+\frac{1}{2}\norm{\bx^*-\bx^{k}}_{{\mathbf{\widetilde{D}}_{\tau^k}}}^2\geq}\label{eq:lemma-x-s}\\
& & \Phi(\bx^{k+1})+\fprod{T(\bx^{k+1}-\bx^*),~\by^{k+1}}+\frac{1}{2}\norm{\bx^*-\bx^{k+1}}_{{\bD}_{\tau^k}}^2+\tfrac{1}{2}\norm{\bx^{k+1}-\bx^k}^2_{\mathbf{\bar{D}}_{\tau^k}}.
\nonumber
\end{eqnarray}}%
Similarly, let $\bq^k\triangleq T(\bx^{k}+\eta^k(\bx^k-\bx^{k-1}))$ and define $q_0^k\in\reals^{m_0}$ and $q_i^k\in\reals^{m_i}$ for $i\in\cN$ such that $q_0^k$ is the subvector corresponding to the components of $\blambda$, and $q_i^k$ is the subvector corresponding to the components of $\theta_i$ for $i\in\cN$, i.e., {$\bq^k=[{q_1^k}^\top \ldots {q_N^k}^\top {q_0^k}^\top]^\top$}. Thus, from \eqref{eq:pock-pd-theta-s} and \eqref{eq:pock-pd-lambda-s}, we have
{\small
\begin{align*}
\blambda^{k+1}&=\argmin_{\blambda} -\fprod{q_0^k,\blambda}+\frac{1}{2\gamma^k}\norm{\blambda-\blambda^k}^2,\\
\theta_i^{k+1}&=\argmin_{\theta_i}\sigma_{\cK_i}(\theta_i)-\fprod{q_i^k-b_i,\theta_i}+\frac{1}{2\kappa_i^k}\norm{\theta_i-\theta_i^k}^2,\quad \forall\ i\in\cN.
\end{align*}}%
Using the strong convexity of these subproblems, {for any $\by=[\btheta^\top,~\blambda^\top]^\top$}, we get
{\small
\begin{align*}
-\fprod{q_0^k,\blambda}+\frac{1}{2\gamma^k}\norm{\blambda-\blambda^k}^2
&\geq -\fprod{q_0^k,\blambda^{k+1}}+\frac{1}{2\gamma^k}\norm{\blambda^{k+1}-\blambda^k}^2+\frac{1}{2\gamma^k}\norm{\blambda-\blambda^{k+1}}^2,\\
\sigma_{\cK_i}(\theta_i)-\fprod{q_i^k-b_i,\theta_i}+\frac{1}{2\kappa_i^k}\norm{\theta_i-\theta_i^k}^2
&\geq \sigma_{\cK_i}(\theta_i^{k+1})-\fprod{q_i^k-b_i,\theta_i^{k+1}}+\frac{1}{2\kappa_i^k}\norm{\theta_i^{k+1}-\theta_i^k}^2+\frac{1}{2\kappa_i^k}\norm{\theta_i-\theta_i^{k+1}}^2.
\end{align*}}%
Since $\fprod{q_0^k,~\blambda}+\sum_{i\in\cN}\fprod{q_i^k,~\theta_i}=\fprod{T(\bx^{k}+{\eta}^k(\bx^k-\bx^{k-1})),~\by}$ for all $\by$, summing the second inequality over $i\in\cN$ and then summing the resulting inequality with the first one, we get
{\small
\begin{eqnarray}
\lefteqn{h(\by)+\frac{1}{2}\norm{\by-\by^k}_{\bD_{\kappa^k,\gamma^k}}^2\geq}\label{eq:lemma-y-s}\\
& & h(\by^{k+1})-\fprod{T(\bx^{k}+\eta^k(\bx^k-\bx^{k-1})),~\by^{k+1}-\by}+\frac{1}{2}\norm{\by-\by^{k+1}}_{\bD_{\kappa^k,\gamma^k}}^2+ \frac{1}{2}\norm{\by^{k+1}-\by^k}_{\bD_{\kappa^k,\gamma^k}}^2.
\nonumber
\end{eqnarray}}%
Next,  summing \eqref{eq:lemma-x-s}, \eqref{eq:lemma-y-s}, and rearranging the terms, we obtain
{\small
\begin{eqnarray}\label{eq:pda-ineq}
\lefteqn{\mathcal{L}({\bf x}^{k+1},\by)-\mathcal{L}({\bf x}^*,\by^{k+1})\leq}\\
& & \eta^k\fprod{ T(\bx^k-\bx^{k-1}),\by^{k+1}-\by}-\frac{1}{2}\norm{\by^{k+1}-\by^k}_{\bD_{\kappa^k,\gamma^k}}^2+\bigg[\frac{1}{2}\norm{\bx^*-\bx^{k}}_{\widetilde{\bD}_{\tau^k}}^2+\frac{1}{2}\norm{\by-\by^k}_{\bD_{\kappa^k,\gamma^k}}^2\bigg]\nonumber\\
& & \mbox{} - \bigg[\frac{1}{2}\norm{\bx^*-\bx^{k+1}}_{{\bD}_{\tau^k}}^2+\frac{1}{2}\norm{\by-\by^{k+1}}_{\bD_{\kappa^k,\gamma^k}}^2 -\fprod{T(\bx^{k+1}-\bx^k),~\by-\by^{k+1}} +\frac{1}{2}\norm{\bx^{k+1}-\bx^k}^2_{\bar{\bD}_{\tau^k}}\bigg]\nonumber
\end{eqnarray}}%
Note that we have
{\small
\begin{align}
\label{eq:pda-eq}
\eta^k\fprod{T(\bx^k-\bx^{k-1}),~\by^{k+1}-\by}=-\eta^k\fprod{T(\bx^k-\bx^{k-1}),~\by-\by^k}+\eta^k\fprod{T(\bx^k-\bx^{k-1}),~\by^{k+1}-\by^k};
\end{align}
}
moreover, using \eqref{cond-0-s}, i.e., $\bar{\bQ}_k\succeq 0$, the last term can be bounded as follows:
{\small
\begin{eqnarray}\label{eq:Q-ineq-s}
{\eta^k\fprod{T(\bx^k-\bx^{k-1}),~\by^{k+1}-\by^k}} \leq \frac{1}{2}\norm{\by^{k+1}-\by^k}_{\bD_{\kappa^k,\gamma^k}}^2+\frac{1}{2}\norm{\bx^k-\bx^{k-1}}^2_{\cA^k}
\end{eqnarray}
}%
{Then, combining \eqref{eq:pda-ineq}, \eqref{eq:pda-eq} and \eqref{eq:Q-ineq-s} gives the desired result.}
\end{proof}

\subsection{Proof of Theorem \ref{thm:static-error-bounds}}
Under Assumption~\ref{assump:saddle-point}, a saddle point $(\bx^*,\by^*)$ for $\min_{\bx\in\cX}\max_{\by\in\cY}\cL(\bx,\by)$ in~\eqref{eq:static-saddle} exists, where $\by^*=[{\btheta^*}^\top,{\blambda^*}^\top]^\top$; moreover, any saddle point $(\bx^*,\btheta^*,\blambda^*)$ satisfies that $\bx^*=\one \otimes x^*$ such that $(x^*,\btheta^*)$ is a primal-dual solution to \eqref{eq:central_problem}. Thus, $\theta_i^*\in\cK_i^\circ$ and $\mathcal{L}(\bx^*,\btheta^*,\blambda^*)=\Phi(\bx^*)$. {Recall Definition~\ref{def:problem-components-static}, since $\norm{\bx^*}^2_{\Omega\otimes\id_n}=0$, we have $g(\bx^*)=f(\bx^*)$; hence, $\Phi(\bx^*)=\varphi(\bx^*)=\sum_{i\in \mathcal{N}}\varphi_i(x^*)$. Therefore, $\mathcal{L}(\bx^*,\btheta^*,\blambda^*)=\varphi(\bx^*)$. Moreover,} note that if $(\bx^*,\btheta^*,\blambda^*)$ is a saddle point of $\cL$ such that $\blambda^*\neq\mathbf{0}$, then it trivially follows that $(\bx^*,\btheta^*,\mathbf{0})$ is another saddle point of $\mathcal{L}$.

Multiplying both sides of \eqref{eq:pda-lemma} by $\frac{\gamma^k}{\gamma^0}$ and using Lemma~\ref{lem:stepsize-s}, we get
{\small
\begin{eqnarray}\label{eq:static-saddle-rate-ineq}
\lefteqn{\frac{\gamma^k}{\gamma^0}\left[\mathcal{L}({\bf x}^{k+1},\by)-\mathcal{L}({\bf x}^*,\by^{k+1})\right]\leq}\\
& & \frac{\gamma^k}{\gamma^0}\bigg[\frac{1}{2}\norm{\bx^*-\bx^{k}}_{\widetilde{\bD}_{\tau^k}}^2+\frac{1}{2}\norm{\by-\by^k}_{\bD_{\kappa^k,\gamma^k}}^2-{\eta^k}\fprod{ T(\bx^k-\bx^{k-1}),~\by-\by^k} +\frac{1}{2}\norm{\bx^k-\bx^{k-1}}^2_{\cA^k}\bigg] \nonumber \\
& & \mbox{} - \frac{\gamma^{k+1}}{\gamma^0} \bigg[\frac{1}{2}\norm{\bx^*-\bx^{k+1}}_{\widetilde{\bD}_{\tau^k}}^2+\frac{1}{2}\norm{\by-\by^{k+1}}_{\bD_{\kappa^k,\gamma^k}}^2-\eta^{k+1}\fprod{T(\bx^{k+1}-\bx^{k}),~\by-\by^{k+1}} +\frac{1}{2}\norm{\bx^{k+1}-\bx^k}^2_{\cA^{k+1}}\bigg]. \nonumber
\end{eqnarray}}%
Next, {we sum \eqref{eq:static-saddle-rate-ineq} from $k=0$ to $K-1$; using Jensen inequality and the following facts: $\bar{\bQ}_K\succeq 0$ and $\bx^{-1}=\bx^0$,} we get
{\small
\begin{eqnarray*}
2 N_K\big(\mathcal{L}(\bar{\bx}^K,\by)-\mathcal{L}(\bx^*,\bar{\by}^K)\big) \leq \Big[ \norm{\bx^*-\bx^{0}}_{{\widetilde{\bD}}_{\tau^0}}^2+\norm{\by-\by^0}_{\bD_{\kappa^0,\gamma^0}}^2\Big] -\frac{\gamma^K}{\gamma^0}\Big[ \norm{\bx^*-\bx^{K}}_{\widetilde{\bD}_{\tau^K}}^2+\norm{\bz^K}_{\bar{\bQ}_K}^2\Big],
\end{eqnarray*}}%
where $\bz^K\triangleq[(\bx^K-\bx^{K-1})^\top~(\by-\by^K)^\top]^\top$, $N_K=\sum_{k=1}^{K}\frac{\gamma^{k-1}}{\gamma^0}$, $\bar{\bx}^{K}=N_K^{-1}\sum_{k=1}^{K}\frac{\gamma^{k-1}}{\gamma^0}\bx^k$, and $\bar{\by}^{K}=N_K^{-1}\sum_{k=1}^{K}\frac{\gamma^{k-1}}{\gamma^0}\by^k$. {Since $\norm{\bz^K}_{\bar{\bQ}_K}^2\geq 0$ and $\tilde{\tau}^k>\tau^k$ for $k\geq 0$,} we get the following bounds for all $K\geq 1$:
{\small
\begin{align}\label{eq:saddle-rate-static}
&\mathcal{L}(\bar{\bf x}^K,\btheta,\blambda)-\mathcal{L}({\bf x}^*,\bar{\btheta}^K,\bar{\blambda}^K)\leq \frac{1}{N_K}\ \Theta(\bx^*,\btheta,\blambda),\quad \frac{1}{2}\norm{\bx^{K}-\bx^*}^2\leq \frac{\tilde{\tau}^K}{\gamma^K}\ \gamma^0~\Theta(\bx^*,\btheta,\blambda),\\
&\Theta(\bx^*,\btheta,\blambda)\triangleq{1\over 2\gamma^0}\|\blambda-\blambda^0\|^2+\sum_{i\in\mathcal{N}}\bigg[{1\over 2\tau^0}\|x_i^0-x^*\|^2+{1\over 2\kappa_i^0}\|\theta_i-\theta_i^0\|^2 \bigg].\nonumber
\end{align}
}%
Under Assumption~\ref{assump:saddle-point}, one can construct a saddle point $(\bx^*,\btheta^*,\blambda^*)$ for $\cL$ in \eqref{eq:static-saddle} such that $\blambda^*=\mathbf{0}$;
hence, $\mathcal{L}(\bf x^*,\btheta^*,\blambda^*)={\varphi(\bx^*)}$ and $\theta_i^*\in\cK_i^\circ$ for $i\in\cN$. Define $\tilde{\btheta}=[\tilde{\theta}_i]_{i\in\cN}$ such that $\tilde{\theta}_i\triangleq 2\|\theta_i^*\|
\big( \|\cP_{\mathcal{K}_i^\circ}(A_i{\bar{x}_i^K}-b_i)\|\big)^{-1}~\cP_{\mathcal{K}_i^\circ}(A_i{\bar{x}_i^K}-b_i)\in\cK_i^\circ$, which implies
{
\begin{equation}
\label{eq:tilde-theta-static}
\langle A_i\bar{x}_i^K -b_i, \tilde{\theta}_i \rangle=2\|\theta_i^*\|~d_{\mathcal{K}_i}(A_i\bar{x}_i^K -b_i).
\end{equation}}%
Similarly, define $\tilde{\blambda}\triangleq M\bar{\bf x}^K/\norm{M\bar{\bf x}^K}$; hence, $\langle M\bar{\bf x}^K, \tilde{\blambda}\rangle=\norm{M\bar{\bf x}^K}$. Together with \eqref{eq:tilde-theta-static}, we get
{
\begin{equation}\label{eq:lagrange-equality-s}
\cL(\bar{\bx}^K,\tilde{\btheta},\tilde{\blambda})= \Phi(\bar{\bx}^K)+2\sum_{i\in\cN} \|\theta_i^*\|~d_{\mathcal{K}_i}(A_i\bar{\x}_i^K-b_i)+\norm{M\bar{\bf x}^K}.
\end{equation}
}%
Note that for any $i\in\cN$, $\bar{\theta}_i^K\in\cK_i^\circ$; hence, $\sigma_{\cK_i}(\bar{\theta}_i^K)=0$. In addition, since $\bar{\theta}^K_i\in\cK_i^\circ$ and $A_i x^*-b_i\in\cK_i$, we have
{
\begin{equation}\label{eq:theta-xstar-s}
\fprod{A_ix^*-b_i,~\bar{\theta}^K_i} \leq 0.
\end{equation}}
Therefore, using \eqref{eq:theta-xstar-s} and the fact that $M\bx^*=0$ we get that,
{
\begin{equation}\label{eq:lagrange-inequality-s}
\cL(\bx^*,\bar{\btheta}^K,\bar{\blambda}^K)~ {\leq}~ \varphi(\bx^*).
\end{equation}}
Thus, \eqref{eq:lagrange-equality-s}, \eqref{eq:lagrange-inequality-s} {and \eqref{eq:saddle-rate-static}} together with the definitions of $\tilde{\btheta}$, $\tilde{\blambda}$ and the fact that $\blambda^0=\mathbf{0}$ and $\btheta^0=\mathbf{0}$ imply that
{
\begin{align}\label{eq:upper-bound-static}
\Phi(\bar{\bx}^K)-{\varphi(\bx^*)}+2\sum_{i\in\cN} d_{\mathcal{K}_i}(A_i\bar{x}_i^K-b_i)\|\theta_i^*\|+\norm{M\bar{\bf x}^K} \leq  \frac{1}{N_K}~\Theta(\bx^*,\tilde{\btheta},\tilde{\blambda})=\frac{\Theta_0}{N_K}.
\end{align}}%
Since $({\bf x}^*,\btheta^*,\blambda^*)$ is a saddle-point for $\cL$ in \eqref{eq:static-saddle}, we have $\mathcal{L}(\bar{\bxi}^K,\btheta^*,\blambda^*)-\mathcal{L}({\bf x}^*,\btheta^*,\blambda^*) \geq 0$; therefore,
{
\begin{equation}\label{eq:aux-lower-static}
\Phi(\bar{\bx}^K)-\varphi({\bf x}^*)+\sum_{i\in\cN}\fprod{\theta_i^*,~A_i\bar{x}_i^K-b_i}\geq 0.
\end{equation}
}%
{Using the conic decomposition of $A_i\bar{x}_i^K-b_i$ and the fact that $\theta^*_i\in\cK_i^\circ$, we immediately get}
{
$$\langle A_i\bar{x}_i^K-b_i, \theta_i^*\rangle\leq \|\theta^*_i\|~d_{\mathcal{K}_i}(A_i\bar{x}_i^K-b_i).$$}%
Together with \eqref{eq:aux-lower-static}, we conclude that
{
\begin{equation}\label{eq:lower-bound-static}
\Phi(\bar{\bx}^K)-\varphi({\bf x}^*)+\sum_{i\in\cN}\|\theta^*_i\|~d_{\mathcal{K}_i}(A_i\bar{\xi}_i^K-b_i) \geq 0.
\end{equation}}%
Finally, combining inequalities \eqref{eq:upper-bound-static} and \eqref{eq:lower-bound-static} immediately implies the desired result. Moreover, the bound on $\norm{\bx^*-\bx^K}$ follows from \eqref{eq:saddle-rate-static}. In fact, possibly a tighter bound can be derived using $\Theta(\bx^*,\btheta^*,\blambda^*)$ for $\lambda^*=\mathbf{0}$.
\subsection{Proof of Lemma~\ref{lem:restricted-convex-general}}
Let $\bx^*=\ones_{|\cN|}\otimes x^*$, where $x^*$ is the unique optimal solution to \eqref{eq:central_problem}, and according to Assumption~\ref{assump:sconvex}, $f$ is strongly convex
with modulus $\bar{\mu}>0$. Fix some arbitrary $\alpha>\frac{4}{\bar{\mu}}\sum_{i\in\cN}L_i^2$ and $\bx\in\reals^{n|\cN|}$. Since $\cC$ is a closed convex cone, $\bx$ can be decomposed into $\bu=\cP_\cC(\bx)$ and $\bv=\cP_{\cC^\circ}(\bx)$, i.e., $\bx=\bu+\bv$ and $\norm{\bx}^2=\norm{\bu}^2+\norm{\bv}^2$. From the definition of $f_\alpha$,
{\small
\begin{align}\label{grad-g-d}
\fprod{\nabla f_\alpha(\bx)-\nabla f_\alpha(\bx^*),~\bx-\bx^*}=\fprod{\nabla f(\bx)-\nabla f(\bx^*),~\bx-\bx^*}+\alpha\fprod{\bx-\bx^*,~\bv},
\end{align}
}%
which follows from the fact that $\grad r(\bx)=\bx-\cP_\cC(\bx)$; hence $\grad r(\bx^*)=\zero$. 
Let $N\triangleq |\cN|$ and $\bar{L}\triangleq \sqrt{\frac{\sum_{i\in\cN}L_i^2}{N}}$. Since $\bx^*,\bu\in\cC$ and $f$ is convex, Lipschitz differentiable, and strongly convex, the same inequalities in \eqref{eq:fprod_bound} implies:
{\small
\begin{align}\label{grad-f-bound-d}
\fprod{\nabla f(\bx)-f(\bx^*),\bx-\bx^*} \geq \frac{\bar{\mu}}{N}\norm{\bu-\bx^*}^2-2\bar{L}\norm{\bu-\bx^*}\norm{\bv}.
\end{align}
}%
Note that $\bu-\bx^*\in\cC$; hence, $\fprod{\bu-\bx^*,\bv}=0$ since $\bv\in\cC^\circ$. Thus, $\fprod{\bx-\bx^*, \bv}=\norm{\bv}^2$; this together with  \eqref{grad-g-d} and \eqref{grad-f-bound-d}
implies that
{\small
\begin{align}\label{g-strong-convex}
\fprod{\nabla f_\alpha(\bx)-\nabla f_\alpha(\bx^*),\bx-\bx^*}\geq \frac{\bar{\mu}}{N}\norm{\bu-\bx^*}^2-2\bar{L}\norm{\bu-\bx^*}\norm{\bv}+\alpha \norm{\bv}^2.
\end{align}
}%
Next, fix some arbitrary $\omega\geq 0$. Then either \textbf{(i)} $\norm{\bv}\leq \omega \norm{\bu-\bx^*}$, or \textbf{(ii)} $\norm{\bv}\geq \omega \norm{\bu-\bx^*}$ holds. Using the same arguments to obtain \eqref{g-strong-convex-s-i}, \eqref{g-strong-convex-s-ii} and \eqref{g-strong-convex-s-result}, we can conclude that
{\small
\begin{align}\label{g-strong-convex-result}
\fprod{\nabla f_\alpha(\bx)-\nabla f_\alpha(\bx^*),~\bx-\bx^*}\geq \min\left\{\frac{\bar{\mu}}{N}-2\bar{L}\omega, ~\alpha-\frac{2\bar{L}}{\omega}\right\} \norm{\bx-\bx^*}^2.
\end{align}
}%
Since $\omega\geq 0$ is arbitrary, $f_\alpha$ is restricted strongly convex with respect to $\bx^*$ with modulus $\mu_\alpha=\max_{\omega\geq 0} \min\left\{\frac{\bar{\mu}}{N}-2\bar{L}\omega, ~\alpha-\frac{2\bar{L}}{\omega}\right\}$. Note $\mu_\alpha$ is attained for $\omega_\alpha\geq 0$ such that $\frac{\bar{\mu}}{N}-2\bar{L}\omega_\alpha= \alpha-\frac{2\bar{L}}{\omega_\alpha}$, which implies that $\omega_\alpha=\frac{\bar{\mu}/N-\alpha}{4\bar{L}} + \sqrt{\left(\frac{\bar{\mu}/N-\alpha}{4\bar{L}}\right)^2 +1}$. Moreover, $\mu_\alpha=\frac{\bar{\mu}}{N}-2\bar{L}\omega_\alpha$ is the value given in the statement of the lemma, and we have $\frac{\bar{\mu}}{N}>\mu_\alpha>0$ for any $\alpha>\frac{4}{ \bar{\mu}}\sum_{i\in\cN}L_i^2$. It is worth mentioning that $\mu_\alpha$ is a concave increasing function of $\alpha$ over $\reals_{++}$, and $\sup_{\alpha>0} \mu_\alpha = \lim_{\alpha\nearrow\infty}\mu_\alpha=\frac{\bar{\mu}}{N}$.
\subsection{Key Lemmas for the Proof of Theorem~\ref{thm:dynamic-error-bounds}}
We first define the proximal error sequences $\{\be_1^k\}_{k\geq 1}$, $\{\be_2^k\}_{k\geq 1}$, and $\{\be_3^k\}_{k\geq 1}$ which will be used for analyzing the convergence of Algorithm~DPDA-TV displayed in Fig.~\ref{alg:PDD}. For $k\geq 0$, let
{\small
\begin{align}
\be_1^{k+1} \triangleq \cP_{\Ct}\left(\bom^k\right)-\tcR^k\left(\bom^k\right),\qquad \be_2^{k+1} \triangleq \cP_\cC(\bxi^k)-\cR^k(\bxi^k),\qquad \be_3^{k+1} \triangleq \bxi^{k+1}-\bx^{k+1},  \label{eq:prox-error-seq-1}
\end{align}
}%
where $\bom^k=\tfrac{1}{\gamma^k}{\bmu}^k+{\bxi}^{k}+\eta^k(\bxi^k-{\bxi}^{k-1})$ and $\tcR^k(\bx)=\cP_{\cB}(\cR^k(\bx))$, i.e., $\tcR^k(\bx)=[\tcR_i^k(\bx)]_{i\in\cN}$ and $\tcR_i^k(\bx)=\cP_{\cB_0}(\cR_i^k(\bx))$, for $\bx\in\cX$. Thus, for $k\geq 0$, $\bmu^{k+1}=\blambda^{k+1}+\gamma^k \be_1^{k+1}$ since \eqref{eq:pock-pd-2-mu} is replaced with \eqref{eq:inexact-rule-mu}, and $\bxi^{k+1}=\bx^{k+1}+\be_3^{k+1}$ since \eqref{eq:pock-pd-2-xi} is replaced with \eqref{eq:inexact-rule-xi}. In the rest, we set 
$\bmu^0$ to $\mathbf{0}$.

The following observation will also be useful to prove error bounds for DPDA-TV iterate sequence. Note that \eqref{eq:inexact-rule-mu} implies for each $i\in\cN$,
{\small
\begin{align*}
\|\nu_i^{k+1}\| &\leq \gamma^k\|\omega_i^{k}\|+\gamma^k\|{\widetilde{\cR}}_i^k\big(\bom^k\big)\|\leq \|\nu_i^k\|+\gamma^k\big[(1+\eta^k)\|\xi_i^k\|+\eta^k\|\xi_i^{k-1}\|+2\Delta\big].
\end{align*}
}%
Thus, we trivially get the following bound on $\norm{\bmu^{k+1}}$:
{\small
\begin{equation}\label{eq:mu-bound}
\|\bmu^{k+1}\| \leq~\sum_{t=0}^k \gamma^t\bigg(2 \sqrt{N}\Delta+(1+\eta^t)\norm{\bxi^t}+\eta^t\norm{\bxi^{t-1}} \bigg).   
\end{equation}
}%
Moreover, we will also need the following relation: for any $\bmu$ and $\blambda$ we have that
{\small
\begin{equation} \label{eq:support-error}
\sigma_{\Ct}({\bmu})=\sup_{{\bf x}\in \Ct}~\langle \blambda,{\bf x} \rangle+\langle \bmu-\blambda,{\bf x} \rangle \leq \sigma_{\Ct}(\blambda)+2\sqrt{N}~\Delta~\|\bmu-\blambda\|.
\end{equation}
}%
\begin{defn}\label{definition}
Let $T=[A^\top~\id_{n|\cN|}]^\top$ for $A\triangleq\diag([A_i]_{i\in\mathcal{N}})\in\reals^{m \times n|\cN|}$. Given $\alpha,\mu,{\delta_1}>0$, and arbitrary sequences $\{\tau^k\},\{\gamma^k\}\subset\reals_{++}$, $\{\kappa_i^k\}_{k\geq 0}\subset\reals_{++}$ for $i\in\mathcal{N}$, define {$\mathbf{D}_{\tau^k}\triangleq\frac{1}{\tau^k}\diag([\id_n]_{i\in\cN})$,} $\mathbf{\widetilde{D}}_{\tau^{k}}\triangleq\diag([({1\over \tau^{k}}- \mu)\id_{n}]_{i\in \mathcal{N}})$, $\mathbf{\bar{D}}_{\tau^k}\triangleq\diag([({1\over \tau^k} -{(L_i+\alpha)})\id_{n}]_{i\in \mathcal{N}})$, and
$\mathbf{\bar{Q}}_k\triangleq\begin{bmatrix}
\cA^k &-\eta^k T^\top\\
-\eta^k T & \mathbf{D}_{\kappa^k,\gamma^k}
\end{bmatrix}$  for $k\geq 0$, where $\cA^k\triangleq(\eta^k)^2\gamma^{k}(1+{\delta_1})~\id_{n|\cN|}\succ 0$ and $\mathbf{D}_{\kappa^k,\gamma^k}$ is defined in Definition~\ref{def:bregman}.
\end{defn}
In order to prove Theorem~\ref{thm:dynamic-error-bounds}, we first prove Lemma~\ref{thm:dynamic-rate} below which help us to appropriately bound $\mathcal{L}(\bar{\bx}^K,\by)-\mathcal{L}(\bx^*,\bar{\by}^K)$ for any $\by\in\cY$ and $\norm{\bxi^K-\bx^*}$. That said to show the result in Lemma~\ref{thm:dynamic-rate}, we need to show the following two lemmas, Lemma \ref{lem:schur} and Lemma \ref{lem:stepsize}, describing a proper choice {for the primal-dual step size sequences.}
\begin{lemma}\label{lem:schur}
Given $\delta_1>0$. For any $k\geq 0$, $\mathbf{\bar{Q}}_k\succeq \mathbf{0}$ if $\eta^k>0$, and positive numbers $\{\kappa_i^k\}_{i\in\mathcal{N}}$, and $\gamma^k$ are chosen such that
{\small
\begin{align}\label{eq:schur-cond}
\frac{\kappa_i^k \norm{A_i}^2}{\gamma^k} \leq \delta_1,\quad \forall~i\in\mathcal{N}.
\end{align}
}%
\end{lemma}
\begin{proof}
Let $\mathbf{D}_{\gamma^k}$ and $\mathbf{D}_{\kappa^k}$ be as in Definition~\ref{def:bregman}. Since $\mathbf{D}_{\gamma^k}\succ 0$, Schur complement condition implies that $\bar{\mathbf{Q}}_k\succeq 0$ if and only if
{\small
\begin{align}\label{eq:schur-cond-2}
\begin{bmatrix}\cA^k & -\eta^k A^\top\\ -\eta^k A & \mathbf{D}_{\kappa^k} \end{bmatrix}
-\gamma^k(\eta^k)^2 \begin{bmatrix} \id_n & 0 \\ 0 & 0\end{bmatrix} \succeq 0.
\end{align}}%
Moreover, since $\mathbf{D}_{\kappa^k}\succ 0$, again using Schur complement one can conclude that \eqref{eq:schur-cond-2} holds if and only if $\frac{1}{(\eta^k)^2}\cA^k-\gamma^k\id_n- A^\top \mathbf{D}_{\kappa^k}^{-1} A\succeq 0$. This condition holds if \eqref{eq:schur-cond} is true.
\end{proof}
\begin{lemma}\label{lem:stepsize}
Let $\bD_{\tau^k}$ and $\bD_{\kappa^k,\gamma^k}$ be as given in Definition \ref{def:bregman}, and $\widetilde{\bD}_{\tau^k}$, $\bar{\bD}_{\tau^k}$ and $\bar{\bQ}_k$ be as in Definition \ref{definition} for $\alpha>0$ chosen according to Lemma \ref{lem:restricted-convex-general} and Remark \ref{rem:g}, and $\mu\in(0,~\max\{\ubar{\mu},~\mu_{\alpha}\})$. Suppose $\{\tau^k\},~\{\eta^k\},~\{\gamma^k\}\subset\reals_{++}$, $\{\kappa_i^k\}_{k\geq 0}\subset\reals_{++}$ for $i\in\mathcal{N}$ are chosen as in DPDA-TV diplayed in Fig.~\ref{alg:PDD}, then the following relations hold for all $k\geq 0$:
{\small
\begin{subequations}
\label{eq:cond-d}
\begin{align}
&\mathbf{\bar{Q}}_k=\begin{bmatrix}
\cA^k &-\eta^k T^\top\\
-\eta^k T & \mathbf{D}_{\kappa^k,\gamma^k}
\end{bmatrix}\succeq 0,\label{cond-0-d} \\
&{\gamma^k}{\bD}_{\tau^k}\succeq \gamma^{k+1}\widetilde{\bD}_{\tau^{k+1}}, \label{cond-1-d}\\
& \gamma^k\bD_{\kappa^k}\succeq \gamma^{k+1}\bD_{\kappa^{k+1}},\label{cond-2-d}\\
& \gamma^k\bD_{\gamma^k}\succeq \gamma^{k+1}\bD_{\gamma^{k+1}}, \label{cond-3-d}\\
& \gamma^k= \gamma^{k+1}\eta^{k+1},\label{cond-4-d}\\
& \gamma^k\bar{\bD}_{\tau^k}\succeq \gamma^{k+1}\cA^{k+1}\label{cond-5-d}.
\end{align}
\end{subequations}}%
{Moreover, $\eta^k\in(0,1)$, $0<\frac{1}{\tilde{\tau}^k}<\frac{1}{\tau^k}=\cO(k)$, and $0<\gamma^k=\cO(k)$.}
\end{lemma}
\begin{proof}
Using the result of Lemma~\ref{lem:schur}, it is trivial to check that the parameter sequence constructed in Fig.~\ref{alg:PDD} {satisfies \eqref{eq:cond-d}} -- see also the discussion in the proof of Lemma~\ref{lem:stepsize-s}.
\end{proof}
In order to prove Theorem~\ref{thm:dynamic-error-bounds}, we need Lemma~\ref{thm:dynamic-rate} which help us to appropriately bound $\mathcal{L}(\bar{\bx}^K,\by)-\mathcal{L}(\bx^*,\bar{\by}^K)$ {for all $\by\in\cY$ and $\norm{\bxi^K-\bx^*}$ for all $K\geq 1$. In particular, Lemma~\ref{thm:dynamic-rate} is similar to Lemma~\ref{lem:pda-lemma} for the static case, but it also accounts for the approximation errors for the time-varying case, arising due to use of $\cR^k$.}
\begin{lemma}\label{thm:dynamic-rate}
Let $\{\bxi^k,\by^k\}_{k\geq 0}$ be the iterate sequence generated using Algorithm DPDA-TV as displayed in Fig.~\ref{alg:PDD} which is initialized from an arbitrary $\bx^0$ and $\by^0$, where $\by^k=[{\btheta^k}^\top {\bmu^k}^\top]^\top$ for $k\geq 0$; and let $\{\be_1^k\}_{k\geq 1}$ and $\{\be_2^k\}_{k\geq 1}$ be the error sequence defined as in \eqref{eq:prox-error-seq-1}. For any $\by\in\cY$, the iterate sequence $\{\bxi^k,\by^k\}_{k\geq 0}$ satisfies for all $k\geq 0$,
{\small
\begin{align}\label{lemeq:inexact-lagrangian-bound}
\mathcal{L}&(\bxi^{k+1},\by)-\mathcal{L}(\bx^*,\by^{k+1})\leq E_1^{k+1}(\bmu)+E_2^{k+1}\\ &+\bigg[\frac{1}{2}\norm{\bx^*-\bxi^{k}}_{\widetilde{\bD}_{\tau^k}}^2+\frac{1}{2}\norm{\by-\by^k}_{\bD_{\kappa^k,\gamma^k}}^2-\eta^k\fprod{ T(\bxi^k-\bxi^{k-1}),~\by-\by^k}+\frac{1}{2}\norm{\bxi^k-\bxi^{k-1}}^2_{\cA^k}\bigg]\nonumber\\
&-\bigg[\frac{1}{2}\norm{\bx^*-\bxi^{k+1}}_{{\bD}_{\tau^k}}^2+\frac{1}{2}\norm{\by-\by^{k+1}}_{\bD_{\kappa^k,\gamma^k}}^2-\fprod{T(\bxi^{k+1}-\bxi^{k}),~\by-\by^{k+1}} +\frac{1}{2}\norm{\bxi^{k+1}-\bxi^k}^2_{\bar{\bD}_{\tau^k}}\bigg].\nonumber
\end{align}
}%
where 
 $E_1^{k+1}(\bmu)\triangleq\|\be^{k+1}\| \left({4}\gamma^{k}\sqrt{N}~\Delta+\|\bmu-\bmu^{k+1}\|\right)$, and $E_2^{k+1}\triangleq \norm{\be_3^{k+1}}\left(\frac{2}{\tau^k}\sqrt{N}\Delta+\alpha\norm{\be_2^{k+1}}\right)$ for {$k\geq 0$}.
\end{lemma}
\begin{proof}
{Fix $\by=[\btheta^\top~\bmu^\top]^\top\in\cY$.} For $k\geq 0$, {let $\bq^k\triangleq \bxi^{k}+\eta^k(\bxi^k-\bxi^{k-1})$ and define $q_i^k\in\reals^{n}$ for $i\in\cN$ such that {$\bq^k=[{q_1^k}^\top \ldots {q_N^k}^\top]^\top$}}. {It follows from~\eqref{eq:pock-pd-2-lambda} that using} strong convexity of ${\sigma_{\Ct}}({\bmu})-\langle \bq^k,~{\bmu}\rangle+{1\over 2\gamma^k}\|{\bmu}-{\bmu}^k\|_2^2$ in $\bmu$ and the fact that ${\blambda}^{k+1}$ is its minimizer, we conclude that
{\small
\begin{equation*}
{\sigma_{\Ct}}(\bmu)-\langle \bq^k,~\bmu\rangle+\tfrac{1}{2\gamma^k}\|{\bmu}-{\bmu}^k\|^2 \geq {\sigma_{\Ct}}({\blambda}^{k+1})-\langle \bq^k,~{\blambda}^{k+1}\rangle+\tfrac{1}{2\gamma^k}\|{\blambda}^{k+1}-{\bmu}^k\|^2+\tfrac{1}{2\gamma^k}\|\bmu-{\blambda}^{k+1}\|^2.
\end{equation*}}%
According to \eqref{eq:prox-error-seq-1}, ${\bmu}^{k+1}={\blambda}^{k+1}+\gamma^k \be_1^{k+1}$ for all $k\geq 1$; hence, from \eqref{eq:support-error} we have
{\small
\begin{eqnarray}
\lefteqn{\sigma_\cC({\bmu})-\langle \bq^k,~{\bmu}\rangle+\tfrac{1}{2\gamma^k}\|{\bmu}-{\bmu}^k\|^2 \geq}\nonumber\\
& & \sigma_\cC({\bmu}^{k+1})-\langle \bq^k,~{\bmu}^{k+1}\rangle +\tfrac{1}{2\gamma^k}\|{\bmu}^{k+1}-{\bmu}^k\|^2+\tfrac{1}{2\gamma^k}\|{\bmu}-{\bmu}^{k+1}\|^2- S_1^{k+1}(\bmu),\label{eq:support-function-bound}
\end{eqnarray}
}%
where the error term $S_1^{k+1}(\bmu)$ is defined as
{\small
\begin{equation}
\label{eq:S1k}
S_1^{k+1}(\bmu)\triangleq 2\gamma^k \sqrt{N}~\Delta \|\be_1^{k+1}\|-\gamma^k\|\be_1^{k+1}\|^2-\fprod{\be_1^{k+1},~\bmu-2\bmu^{k+1}+\bmu^k+\gamma^k\bq^k}.
\end{equation}
}%
Note that for all $k\geq 0$, we have $\bmu^k+\gamma^k\bq^k=\gamma^k\bom^k$, $\bmu^{k+1}=\blambda^{k+1}+\gamma^k\be_1^{k+1}$, and $\blambda^{k+1}=\gamma^k(\bom^k-\cP_{\Ct}(\bom^k))$. Using these we get $\bmu^k+\gamma^k\bq^k-\bmu^{k+1}=\gamma^k(\cP_{\Ct}(\bom^k)-\be_1^{k+1})$; therefore, \eqref{eq:S1k} can be 
written as
\begin{equation}
\label{eq:Sk-bound}
S_1^{k+1}(\bmu)={2\gamma^k\sqrt{N}~\Delta}~\|\be_1^{k+1}\|-\fprod{\be_1^{k+1},~\bmu-\bmu^{k+1}+\gamma^k{\cP_{\Ct}(\bom^k)}}\leq E_1^{k+1}(\bmu),
\end{equation}
where the inequality follows from Cauchy-Schwarz and $\|\cP_{\Ct}(\bom^k)\|\leq 2\sqrt{N}~\Delta$ since $\cP_{\Ct}(\bom^k)\in \Ct$.
Moreover, it follows from the strong convexity of the objective in \eqref{eq:pock-pd-2-theta} that
{\small
\begin{eqnarray*}
\lefteqn{\sigma_{\cK_i}(\theta_i)-\fprod{{A_i}q_i^k-b_i,\theta_i}+\frac{1}{2\kappa_i^k}\norm{\theta_i-\theta_i^k}^2}\\
& & \geq \sigma_{\cK_i}(\theta_i^{k+1})-\fprod{{A_i}q_i^k-b_i,\theta_i^{k+1}}+\frac{1}{2\kappa_i^k}\norm{\theta_i^{k+1}-\theta_i^k}^2+\frac{1}{2\kappa_i^k}\norm{\theta_i-\theta_i^{k+1}}^2.
\end{eqnarray*}}%
Summing the above inequality over $i\in\cN$, then summing the resulting inequality with \eqref{eq:support-function-bound} and using \eqref{eq:Sk-bound}, we get
{\small
\begin{eqnarray}
\lefteqn{h(\by)+\frac{1}{2}\norm{\by-\by^k}_{\bD_{\kappa^k,\gamma^k}}^2+E_1^{k+1}(\bmu)\geq}\label{eq:lemma-y}\\
& & h(\by^{k+1})-\fprod{T{(\bxi^{k}+\eta^k(\bxi^k-\bxi^{k-1}))},~\by^{k+1}-\by}+\frac{1}{2}\norm{\by-\by^{k+1}}_{\bD_{\kappa^k,\gamma^k}}^2+ \frac{1}{2}\norm{\by^{k+1}-\by^k}_{\bD_{\kappa^k,\gamma^k}}^2. \nonumber
\end{eqnarray}}%

Let $\bp^k=T^\top \by^k$ for $k\geq 1$. Strong convexity of the objective in \eqref{eq:pock-pd-2-x} implies that
{\small
\begin{eqnarray}\label{eq:rho}
\lefteqn{\rho(\bx^*)+\fprod{\grad g(\bxi^k),~\bx^*}+\fprod{\bp^{k+1},\bx^*}+\frac{1}{2\tau^k}\norm{\bx^*-\bxi^k}^2\geq}\nonumber\\
& & \rho(\bx^{k+1})+\fprod{\grad g(\bxi^k),~\bx^{k+1}}+\fprod{\bp^{k+1},\bx^{k+1}}+\frac{1}{2\tau^k}\norm{\bx^{k+1}-\bxi^k}^2+\frac{1}{2\tau^k}\norm{\bx^*-\bx^{k+1}}^2,
\end{eqnarray}}%
where $\nabla g(\bxi^k)=\nabla f(\bxi^k)+\alpha(\bxi^k-\cP_{\cC}(\bxi^k))$. Also, the optimality condition of \eqref{eq:inexact-rule-xi} implies that, there exist $\bs^{k+1}\in\partial \rho(\bxi^{k+1})$ such that
{\small
\begin{equation*}
\bs^{k+1}+\nabla f(\bxi^k)+\alpha(\bxi^k-{\cR^k}(\bxi^k))+\bp^{k+1}+\frac{1}{\tau^k}(\bxi^{k+1}-\bxi^k)=0,
\end{equation*}
}%
which {together with \eqref{eq:prox-error-seq-1}} implies that,
{\small
\begin{equation}\label{opt-cond-xi}
{\bs^{k+1}+\grad g(\bxi^k)+\bp^{k+1}=~}\bs^{k+1}+\nabla f(\bxi^k)+\alpha(\bxi^k-{\cP_\cC}(\bxi^k))+\bp^{k+1}=\frac{1}{\tau^k}(\bxi^{k}-\bxi^{k+1})-\alpha \be_2^{k+1}.
\end{equation}
}%
{Moreover, since $\rho(\cdot)$ is a convex function and $\bs^{k+1}\in\partial\rho(\bxi^{k+1})$, using \eqref{eq:prox-error-seq-1}} we obtain
{\small
\begin{align}\label{eq:subderivative}
\rho(\bx^{k+1})\geq \rho(\bxi^{k+1})+\fprod{\bs^{k+1},~\bx^{k+1}-\bxi^{k+1}}=\rho(\bxi^{k+1})-\fprod{\bs^{k+1},~\be_3^{k+1}}.
\end{align}
}
Now, 
{using \eqref{eq:prox-error-seq-1} and \eqref{eq:subderivative} within \eqref{eq:rho}}, we conclude that
{\small
\begin{eqnarray}\label{eq:rho_inexact}
\lefteqn{\rho(\bx^*)+\fprod{\grad g(\bxi^k),~\bx^*}+\fprod{\bp^{k+1},\bx^*}+\frac{1}{2\tau^k}\norm{\bx^*-\bxi^k}^2}\\
& & \geq \rho(\bxi^{k+1})+\fprod{\grad g(\bxi^k),~\bxi^{k+1}}+\fprod{\bp^{k+1},~\bxi^{k+1}}+\frac{1}{2\tau^k}\norm{\bxi^{k+1}-\bxi^k}^2+\frac{1}{2\tau^k}\norm{\bx^*-\bxi^{k+1}}^2-S_2^{k+1},\nonumber
\end{eqnarray}}%
where the error term $S_2^{k+1}$ is given as follows
{\small
\begin{align}
S_2^{k+1}\triangleq & -\frac{1}{\tau^k}\norm{\be_3^{k+1}}^2+\fprod{\be_3^{k+1},~\bs^{k+1}+\grad g(\bxi^k)+\bp^{k+1}}+\frac{1}{\tau^k}\fprod{\be^{k+1},~2\bxi^{k+1}-\bxi^k-\bx^*},
\end{align}}%
{Note that using \eqref{opt-cond-xi}, the definition of $S_2^{k+1}$ can be simplified:}
{\small
\begin{equation}\label{eq:S2k}
S_2^{k+1}= \fprod{\be_3^{k+1},~\frac{1}{\tau^k}(\bx^{k+1}-\bx^*)-\alpha\be_2^{k+1}} \leq E_2^{k+1},
\end{equation}}%
where we used the fact that $\norm{\bx^{k+1}}\leq \sqrt{N}\Delta$.
{
In addition, since each $f_i$ has a Lipschitz continuous gradient {with constant $L_i$ and $\tfrac{1}{2}d^2_{\cC}(\bx)$ has a Lipschitz continuous gradient with constant $1$}, we have for any $\bx$ and $\bar{\bx}$ that
{\small
\begin{align}\label{eq:f-lip}
{g(\bx)\leq g(\bar{\bx})+\fprod{\nabla g(\bar{\bx}),~\bx-\bar{\bx}}+\sum_{i\in\cN}\frac{L_i+\alpha}{2}\norm{x_i-\bar{x}_i}^2.}
\end{align}
}%
{Define $\bL_g=\diag([(L_i+\alpha)\id_n]_{i\in\cN})$.} It follows from strong convexity of $\bar{f}$ that choosing $\alpha\geq 0$ according to Lemma \ref{lem:restricted-convex-general} and Remark~\ref{rem:g}, we conclude that for any $\mu\in(0,~\max\{\ubar{\mu},~\mu_\alpha\})$ we have
{\small
\begin{align}\label{eq:g-strong}
g(\bx^*)&\geq g(\bxi^{k})+\fprod{\grad g(\bxi^{k}),~\bx^*-\bxi^{k}} +\frac{\mu}{2}\norm{\bx^*-\bxi^{k}}^2 \nonumber \\
&\geq g(\bxi^{k+1})+\fprod{\grad g(\bxi^{k}),~\bx^*-\bxi^{k+1}} +\frac{\mu}{2}\norm{\bx^*-\bxi^{k}}^2-\frac{1}{2}\norm{\bxi^{k+1}-\bxi^k}^2_{\bL_g}.
\end{align}}%
{where the last inequality follows from \eqref{eq:f-lip}.} 
Next, {summing inequalities \eqref{eq:rho_inexact}} and \eqref{eq:g-strong}, and using \eqref{eq:S2k}, we get
{\small
\begin{eqnarray}
\lefteqn{\Phi(\bx^*)+\frac{1}{2}\norm{\bx^*-\bxi^{k}}_{\widetilde{\bD}_{\tau^k}}^2+E_2^{k+1}\geq}\label{eq:lemma-x}\\
& & \Phi(\bxi^{k+1})+\fprod{T(\bxi^{k+1}-\bx^*),~\by^{k+1}}+\frac{1}{2}\norm{\bx^*-\bxi^{k+1}}_{\bD_{\tau^k}}^2+\frac{1}{2}\norm{\bxi^{k+1}-\bxi^k}^{{2}}_{\mathbf{\bar{D}}_\tau^k}.
\nonumber
\end{eqnarray}}%
Next, summing \eqref{eq:lemma-y} and \eqref{eq:lemma-x}, and rearranging terms, we obtain
{\small
\begin{eqnarray}\label{eq:inexact-lagrangian-bound}
\lefteqn{\mathcal{L}(\bxi^{k+1},\by)-\mathcal{L}(\bx^*,\by^{k+1})}\\
&\leq & E_1^{k+1}(\bmu)+E_2^{k+1}\nonumber \\
& & \mbox{} + \eta^k\fprod{T(\bxi^k-\bxi^{k-1}),~\by^{k+1}-\by}{-\frac{1}{2}\norm{\by^{k+1}-\by^k}_{\bD_{\kappa^k,\gamma^k}}^2}+ \bigg[\frac{1}{2}\norm{\bx^*-\bxi^{k}}_{\widetilde{\bD}_{\tau^k}}^2+\frac{1}{2}\norm{\by-\by^k}_{\bD_{\kappa^k,\gamma^k}}^2\bigg] \nonumber \\
& & \mbox{} -\bigg[ \frac{1}{2}\norm{\bx^*-\bxi^{k+1}}_{{\bD}_{\tau^k}}^2+\frac{1}{2}\norm{\by-\by^{k+1}}_{\bD_{\kappa^k,\gamma^k}}^2
-\fprod{T(\bxi^{k+1}-\bxi^{k}),~\by-\by^{k+1}}+\frac{1}{2}\norm{\bxi^{k+1}-\bxi^k}^2_{\bar{\bD}_{\tau^k}}\bigg].\nonumber
\end{eqnarray}
}%
}
Note that we have,
{\small
\begin{equation*}
\eta^k\fprod{T(\bxi^{k}-\bxi^{k-1}),~\by^{k+1}-\by}=-\eta^k\fprod{ T(\bxi^k-\bxi^{k-1}),~\by-\by^k}+\eta^k\fprod{ T(\bxi^k-\bxi^{k-1}),~\by^{k+1}-\by^k};
\end{equation*}
}
moreover, the last term can be bounded using the fact that $\bar{\bQ}_k \succeq 0$ as follows:
{\small
\begin{eqnarray}\label{eq:Q-ineq-d}
{\eta^k\fprod{ T(\bxi^k-\bxi^{k-1}),\by^{k+1}-\by^k}} \leq \frac{1}{2}\norm{\by^{k+1}-\by^k}_{\bD_{\kappa^k,\gamma^k}}^2+\frac{1}{2}\norm{\bxi^k-\bxi^{k-1}}^2_{\cA^k}.
\end{eqnarray}}%
Combining \eqref{eq:inexact-lagrangian-bound} and \eqref{eq:Q-ineq-d} gives the desired result.
\end{proof}
Now we are ready to prove Theorem~\ref{thm:dynamic-error-bounds}.
\subsection{Proof of Theorem~\ref{thm:dynamic-error-bounds}}
Under Assumption~\ref{assump:saddle-point}, a saddle point $(\bx^*,\by^*)$ for $\min_{\bx\in\cX}\max_{\by\in\cY}\cL(\bx,\by)$ in~\eqref{eq:dynamic-saddle} exists, where $\by^*=[{\btheta^*}^\top,{\blambda^*}^\top]^\top$; moreover, any saddle point $(\bx^*,\btheta^*,\blambda^*)$ satisfies that $\bx^*=\one \otimes x^*$ such that $(x^*,\btheta^*)$ is a primal-dual solution to \eqref{eq:central_problem}. Thus, $\theta_i^*\in\cK_i^\circ$ and $\mathcal{L}(\bx^*,\btheta^*,\blambda^*)=\Phi(\bx^*)$. Recall Definition~\ref{def:problem-components}, we have $g(\bx^*)=f(\bx^*)$ since $d_\cC(\bx^*)=0$; hence, $\Phi(\bx^*)=\varphi(\bx^*)=\sum_{i\in \mathcal{N}}\varphi_i(x^*)$. Therefore, $\mathcal{L}(\bx^*,\btheta^*,\blambda^*)=\varphi(\bx^*)$. Indeed, this implies $\fprod{\bx^*,\blambda^*}-\sigma_{\Ct}(\blambda^*)=0$ which leads to {$\sum_{i\in\cN}\lambda_i^*=\mathbf{0}$}, i.e., $\blambda^*\in\cC^\circ$. Therefore, we have $0=\fprod{\bx^*,\blambda^*}=\sigma_{\Ct}(\blambda^*)$. {In the rest of the proof, we provide the error bounds for a saddle point $(\bx^*,\btheta^*,\blambda^*)$ of $\cL$ such that $\blambda^*=\mathbf{0}$. Note that if $(\bx^*,\btheta^*,\blambda^*)$ is a saddle point of $\cL$ such that $\blambda^*\neq\mathbf{0}$, then it trivially follows that $(\bx^*,\btheta^*,\mathbf{0})$ is another saddle point of $\mathcal{L}$.}

Multiplying both sides of \eqref{lemeq:inexact-lagrangian-bound} by $\frac{\gamma^k}{\gamma^0}$ and using Lemma \ref{lem:stepsize}, we get
{\small
\begin{align}\label{eq:lagrangian-sum}
\frac{\gamma^k}{\gamma^0}[\cL&(\bxi^{k+1},\by)-\cL(\bx^*,\by^{k+1})]\leq \frac{\gamma^k}{\gamma^0}(E_1^{k+1}(\bmu)+E_2^{k+1})\\ &+\frac{\gamma^k}{\gamma^0}\bigg[\frac{1}{2}\norm{\bx^*-\bxi^{k}}_{\widetilde{\bD}_{\tau^k}}^2+\frac{1}{2}\norm{\by-\by^k}_{\bD_{\kappa^k,\gamma^k}}^2-\eta^k\fprod{ T(\bxi^k-\bxi^{k-1}),~\by-\by^k}+\frac{1}{2}\norm{\bxi^k-\bxi^{k-1}}^2_{\cA^k}\bigg]\nonumber\\
&-\frac{\gamma^{k+1}}{\gamma^0}\bigg[\frac{1}{2}\norm{\bx^*-\bxi^{k+1}}_{{\widetilde{\bD}}_{\tau^{k+1}}}^2+\frac{1}{2}\norm{\by-\by^{k+1}}_{\bD_{\kappa^{k+1},\gamma^{k+1}}}^2-\eta^{k+1}\fprod{T(\bxi^{k+1}-\bxi^{k}),~\by-\by^{k+1}} +\frac{1}{2}\norm{\bxi^{k+1}-\bxi^k}^2_{\cA^{k+1}}\bigg].\nonumber
\end{align}
}%
Next, {we sum \eqref{eq:lagrangian-sum} over $k=0$ to $K-1$;
using Jensen's inequality and the following facts: $\bar{\bQ}_K\succeq 0$ and $\bxi^{-1}=\bxi^0=\bx^0$}, we get
{\small
\begin{align}\label{eq:dynamic-saddle-rate}
N_K(\mathcal{L}(\bar{\bxi}^K,\by)-\mathcal{L}(\bx^*,\bar{\by}^K)) \leq & \Big[ \frac{1}{2}\norm{\bx^*-\bx^{0}}_{{\widetilde{\bD}}_{\tau^0}}^2+\frac{1}{2}\norm{\by-\by^0}_{\bD_{\kappa^0,\gamma^0}}^2+\sum_{k=0}^{K-1}\frac{\gamma^k}{\gamma^0} \left( E_1^{{k+1}}(\bmu) + E_2^{{k+1}}\right)\Big] \nonumber \\
&-\frac{\gamma^K}{\gamma^0}\Big[ \frac{1}{2}\norm{\bx^*-\bxi^{K}}_{\tilde{\bD}_{\tau^K}}^2+\frac{1}{2}\norm{\bz^K}^2_{\bar{\bQ}_K}\Big],
\end{align}}%
where $\bz^K=[(\bxi^K-\bxi^{K-1})^\top~(\by-\by^K)^\top]^\top$, $N_K=\sum_{k=1}^{K}\frac{\gamma^{k-1}}{\gamma^0}$, $\bar{\bxi}^{K}=N_K^{-1}\sum_{k=1}^{K}\frac{\gamma^{k-1}}{\gamma^0}\bxi^k$ and {$\bar{\by}^{K}=N_K^{-1}\sum_{k=1}^{K}\frac{\gamma^{k-1}}{\gamma^0}\by^k$ for $\by^k=[{\btheta^k}^\top {\bmu^k}^\top]^\top$ for $k\geq 0$}.
Note that $E_1^{{k+1}}(\bmu)$ and $E_2^{k+1}$ appearing in \eqref{eq:dynamic-saddle-rate} are the error terms due to approximating {$\cP_{\cC}$ with $\cR^k$} in the $k$-th iteration of the algorithm for $k\geq 0$. 
Furthermore, since $\norm{\bz^K}_{\bar{\bQ}_K}\geq 0$ and $\tilde{\tau}^k>\tau^k$ for $k\geq 0$, \eqref{eq:dynamic-saddle-rate} can be written more explicitly as follows: for any {$[\btheta^\top,~\bmu^\top]\in\cY$ and for all $K\geq 1$}, we have
{\small
\begin{align} \label{eq:saddle-rate-dynamic}
\mathcal{L}(\bar{\bxi}^K,\btheta,\bmu)-&\mathcal{L}({\bf x}^*,\bar{\btheta}^K,\bar{\bmu}^K)\leq \Theta(\bx^*,\btheta,\bmu)/N_K, \quad \norm{{\bxi}^K-\bx^*}\leq \frac{\tilde{\tau}^K}{\gamma^K}2\gamma^0 \Theta(\bx^*,\btheta,\bmu)\\
\Theta(\bx^*,\btheta,\bmu)\triangleq&
{1\over 2\gamma^0}\|\bmu-\blambda^0\|^2+\sum_{i\in\mathcal{N}}\bigg[{1\over 2\tau_i^0}\|x^*-x_i^0\|^2+{1\over 2\kappa_i^0}\|\theta_i-\theta_i^0\|^2 \bigg]+\sum_{k=0}^{K-1} \frac{\gamma^k}{\gamma^0}\left( E_1^{k+1}(\bmu) +E_2^{k+1}\right).\nonumber
\end{align}
}%
Under Assumption~\ref{assump:saddle-point}, one can construct a saddle point $(\bx^*,\btheta^*,\blambda^*)$ for $\cL$ in \eqref{eq:dynamic-saddle} such that $\blambda^*=\mathbf{0}$;
hence, $\mathcal{L}(\bf x^*,\btheta^*,\blambda^*)=\varphi(\bx^*)$ and $\theta_i^*\in\cK_i^\circ$ for $i\in\cN$. Define $\tilde{\btheta}=[\tilde{\theta}_i]_{i\in\cN}$ such that $\tilde{\theta}_i\triangleq 2\|\theta_i^*\|
\big( \|\cP_{\mathcal{K}_i^\circ}(A_i\bar{\xi}_i^K-b_i)\|\big)^{-1}~\cP_{\mathcal{K}_i^\circ}(A_i\bar{\xi}_i^K-b_i)\in\cK_i^\circ$, which implies
{
\begin{equation}
\label{eq:tilde-theta-dynamic}
\langle A_i\bar{\xi}_i^K -b_i, \tilde{\theta}_i \rangle=2\|\theta_i^*\|~d_{\mathcal{K}_i}(A_i\bar{\xi}_i^K -b_i).
\end{equation}}%
Note that ${\cC}$ is a closed convex cone, and the projection $\mathcal{P}_{\cC}(\bx)=\one\otimes {p}(\bx)$ 
 -- see \eqref{eq:average}. Similarly, define $\tilde{\bmu}={\mathcal{P}_{\cC^\circ}(\bar{\bxi}^K)\over \|\mathcal{P}_{\cC^\circ}(\bar{\bxi}^K)\|}\in \cC^\circ$, where $\cC^\circ$ denotes polar cone of $\cC$. Hence, it can be verified that $\langle \tilde{\bmu}, \bar{\bxi}^K\rangle= d_{\cC}(\bar{\bxi}^K)$. 
Note that $\tilde{\bmu} \in \cC^\circ$ implies that $\sigma_{\cC}(\tilde{\bmu})=0$; moreover, we also have $\Ct \subseteq \cC$; hence, $\sigma_{\Ct} (\tilde{\bmu})\leq\sigma_{\cC}(\tilde{\bmu})=0$. Therefore, we can conclude that $\sigma_{\Ct}(\tilde{\bmu})=0$ since $\zero\in \Ct$. Together with \eqref{eq:tilde-theta-dynamic}, we get
{
\begin{equation}\label{eq:lagrange-equality}
\cL(\bar{\bxi}^K,\tilde{\btheta},\tilde{\bmu})= \Phi(\bar{\bxi}^K)+2\sum_{i\in\cN} d_{\mathcal{K}_i}(A_i\bar{\xi}_i^K-b_i)\|\theta_i^*\|+d_{\cC}(\bar{\bxi}^K).
\end{equation}
}
Since, $\bx^*\in\Ct$ we also have that
{
\begin{equation}\label{eq:v-bar}
\fprod{\bx^*,~\bar{\bv}^K}-\sigma_{\Ct}(\bar{\bv}^K) \leq \sup_{\bv} \fprod{\bx^*,~{\bv}}-\sigma_{\Ct}({\bv}) =\ind{\Ct}(\bx^*)=0.
\end{equation}}%
Note that for any $i\in\cN$, $\bar{\theta}_i^K\in\cK_i$; hence, $\sigma_{\cK_i}(\bar{\theta}_i^K)=0$. In addition, since $\bar{\theta}^K_i\in\cK_i^\circ$, and $A_i x^*-b_i\in\cK_i$, we have
{
\begin{equation}\label{eq:theta-xstar}
\fprod{A_ix^*-b_i,~\bar{\theta}^K_i} \leq 0.
\end{equation}}%
Therefore, \eqref{eq:v-bar} and \eqref{eq:theta-xstar} imply
{
\begin{equation}\label{eq:lagrange-inequality}
\cL(\bx^*,\bar{\btheta}^K,\bar{\bmu}^K) {\leq} \varphi(\bx^*).
\end{equation}}%
Thus, \eqref{eq:lagrange-equality}, {\eqref{eq:lagrange-inequality}} and \eqref{eq:saddle-rate-dynamic} together with the definitions of $\tilde{\btheta}$, $\tilde{\bmu}$ and the fact that $\bmu^0=\mathbf{0}$ and $\btheta^0=\mathbf{0}$ imply that
{
\begin{eqnarray}\label{eq:upper-bound-dynamic}
\lefteqn{\Phi(\bar{\bxi}^K)-\varphi(\bx^*)+2\sum_{i\in\cN}\|\theta_i^*\|~ d_{\mathcal{K}_i}(A_i\bar{\xi}_i^K-b_i)+d_{\cC}(\bar{\bxi}^K)} \nonumber \\
& &\leq  \Theta(\bx^*,\tilde{\btheta},\tilde{\bmu})/N_K \leq \frac{1}{N_K}\bigg(\Theta_1+\sum_{k=0}^{K-1}\frac{\gamma^k}{\gamma^0}\left( E_1^{k+1}(\tilde{\bmu})+E_2^{k+1} \right)\bigg).
\end{eqnarray}}%
Since $({\bf x}^*,\btheta^*,\blambda^*)$ is a saddle-point for $\cL$ in \eqref{eq:dynamic-saddle}, we have $\mathcal{L}(\bar{\bxi}^K,\btheta^*,\blambda^*)-\mathcal{L}({\bf x}^*,\btheta^*,\blambda^*) \geq 0$; therefore,
{
\begin{equation}\label{eq:aux-lower-dynamic}
\Phi(\bar{\bxi}^K)-\varphi({\bf x}^*)+\sum_{i\in\cN}\fprod{\theta_i^*,~A_i\bar{\xi}_i^K-b_i}\geq 0.
\end{equation}
}%
Using conic decomposition of $\fprod{A_i\bar{\xi}_i^K-b_i,\theta_i^*}$ and the fact that $\theta_i^*\in\cK_i^\circ$, we have that,
{
$$\langle A_i\bar{\xi}_i^K-b_i, \theta_i^*\rangle\leq \|\theta^*_i\|~d_{\mathcal{K}_i}(A_i\bar{\xi}_i^K-b_i).$$}%
Thus, together with \eqref{eq:aux-lower-dynamic}, we conclude that
{\small
\begin{equation}\label{eq:lower-bound-dynamic}
\Phi(\bar{\bxi}^K)-\varphi({\bf x}^*)+\sum_{i\in\cN}\|\theta^*_i\|~d_{\mathcal{K}_i}(A_i\bar{\xi}_i^K-b_i) \geq 0.
\end{equation}}%
{Provided that we show $\Theta_1+\sum_{k=0}^{K-1}\frac{\gamma^k}{\gamma^0}\left( E_1^{k+1}(\tilde{\bmu})+E_2^{k+1} \right)\leq\Theta(K)$ for some $\Theta(K)=\cO\big(\sum_{k=0}^{K-1}\beta^{q_k}k^4\big)$, the desired result in \eqref{eq:rate_result-d} follows from \eqref{eq:upper-bound-dynamic} and \eqref{eq:lower-bound-dynamic}. Moreover, the bound on $\norm{\bxi^K-\bx^*}$ follows from \eqref{eq:saddle-rate-dynamic}. In fact, possibly a tighter bound can be derived using $\Theta(\bx^*,\btheta^*,\blambda^*)$ for $\lambda^*=\mathbf{0}$. In the rest of the proof, we construct the $\Theta(K)$ bound with properties as specified above.}

Note that using \eqref{eq:approx_error-for-full-vector-x} and the non-expansivity of projection, $\cP_{\cB}(\cdot)$, we conclude that $$\|{\tcR}^k(\bx)-\mathcal{P}_{\Ct}(\bx)\|\leq N~\Gamma \beta^{q_k}\norm{\bx}\quad \forall\bx.$$
Moreover, since we assumed that each $\rho_i$ has a compact domain with diameter at most $\Delta$, we immediately conclude that $\norm{\bx^k}\leq\sqrt{N}~\Delta$ and $\norm{\bxi^k}\leq \sqrt{N}\Delta$ for $k\geq 1$. Hence, from \eqref{eq:prox-error-seq-1} and nonexpansivity of prox operator we obtain
{\small
\begin{align}\label{eq:bound-error-3}
\norm{\be_3^{k+1}}\leq \alpha \norm{\be_2^{k+1}} \leq \alpha N\Gamma \beta^{q_k} \norm{\bxi^k}\leq \alpha N^{3\over 2}\Delta\Gamma \beta^{q_k}.
\end{align}
}%
Let $\bq^k=\bxi^k+\eta^k(\bxi^k-\bxi^{k-1})$ for $k\geq 0$. Note that for $\{\eta^k\}$ as specified in Algorithm DPDA-TV displayed in Fig.~\ref{alg:PDD}, we have $\eta^k\leq 1$. Therefore, it follows from \eqref{eq:prox-error-seq-1} and \eqref{eq:mu-bound} that
{\small
\begin{align}
\label{eq:e-bound}
\norm{\be_1^{k+1}}&=\norm{\cP_{\Ct}\left(\tfrac{1}{\gamma^k}{\bmu}^k+\bq^k \right)-\tcR^k\left(\tfrac{1}{\gamma^k}{\bmu}^k+\bq^k\right)}\leq N~\Gamma \beta^{q_k}\norm{\tfrac{1}{\gamma^k}{\bmu}^k+\bq^k} \nonumber \\
&\leq N\Gamma \beta^{q_k}\left(\frac{5\sqrt{N}\Delta}{\gamma^k}\sum_{t=0}^{k-1} \gamma^t+ 3\sqrt{N} \Delta \right) = N^{3\over 2}\Delta\Gamma \beta^{q_k}\left(\frac{5}{\gamma^k}\sum_{t=0}^{k-1} \gamma^t + 3\right).
\end{align}
}%
Therefore, using \eqref{eq:mu-bound} and \eqref{eq:e-bound} we obtain
{\small
\begin{align}\label{eq:E_mu_bound}
\sum_{k=0}^{K-1} \frac{\gamma^{k}}{\gamma^0} E_1^{k+1}(\tilde{\bmu})&=\sum_{k=0}^{K-1}\frac{\gamma^{k}}{\gamma^0}\norm{\be_1^{k+1}} \big(4\gamma^{k}\sqrt{N}~\Delta+\norm{\tilde{\bmu}-\bmu^{k+1}}\big) \nonumber \\
&\leq \sum_{k=0}^{K-1} \frac{\Delta}{\gamma^0}N^{3\over 2} \Gamma\beta^{q_k}\left(5\sum_{t=0}^{k-1} \gamma^t + {3\gamma^k}\right) \left(4\gamma^k \sqrt{N}\Delta+1+5\sqrt{N}\Delta\sum_{t=0}^{K-1}\gamma^t\right)\nonumber\\
&\leq  \sum_{k=0}^{K-1} \frac{\Delta^2}{\gamma^0}N^{3} \Gamma\beta^{q_k}\left(5\sum_{t=0}^{k-1} \gamma^t + {3\gamma^k}\right)\left(5\sum_{t=0}^{k-1} \gamma^t + 1+ {4\gamma^k}\right)   \triangleq \Theta_2(K).
\end{align}}%
Moreover, from \eqref{eq:bound-error-3} we obtain
{\small
\begin{align}\label{eq:E_x_bound}
\sum_{k=0}^{K-1} \frac{\gamma^{k}}{\gamma^0} E_2^{k+1}&= \sum_{k=0}^{K-1} \frac{\gamma^{k}}{\gamma^0}\norm{\be_3^{k+1}}\left(\frac{2}{\tau^{k}}\sqrt{N}\Delta+\alpha\norm{\be_2^{k+1}}\right) \nonumber \\
& \leq   \sum_{k=0}^{K-1} \alpha N^{3\over 2}\Delta\Gamma \beta^{q_k} \frac{\gamma^{k}}{\gamma^0} \left(\frac{2}{\tau^{k}}\sqrt{N}\Delta+\alpha N^{3\over 2}\Delta\Gamma \beta^{q_k} \right) \nonumber \\
&\leq  \sum_{k=0}^{K-1} \alpha N^{3}\Delta^2\Gamma \beta^{q_k} \frac{\gamma^{k}}{\gamma^0} \left(\frac{2}{\tau^{k}}+\alpha N\Gamma \beta^{q_k} \right) \triangleq \Theta_3(K).
\end{align}}%
Therefore, by letting $\Theta(K)=\Theta_1+\Theta_2(K)+\Theta_3(K)$ it is easy to see that $\Theta(K)=\cO(\sum_{k=0}^{K-1}\beta^{q_k}k^4)$; thus, $\sup_{K\in\integers_+}\Theta(K)<\infty$ due to our choice of $\{q_k\}$, and this completes the proof.
\end{document}